\documentclass{article}
\usepackage{amsfonts}
\usepackage{amssymb}
\usepackage{graphicx}
\usepackage{amsmath}

\newtheorem{theorem}{Theorem}

\newtheorem{corollary}[theorem]{Corollary}

\newtheorem{lemma}[theorem]{Lemma}

\newtheorem{proposition}[theorem]{Proposition}
\newtheorem{remark}[theorem]{Remark}

\newenvironment{proof}[1][Proof]{\textbf{#1.} }{\ \rule{0.5em}{0.5em}}
\input{tcilatex}

\begin{document}

\title{On bivariate $s$-Fibopolynomials}
\author{Claudio de Jes\'{u}s Pita Ruiz Velasco \\
Universidad Panamericana\\
Mexico City, Mexico\\
email: cpita@up.edu.mx}
\date{}
\maketitle

\begin{abstract}
In this article we study the mathematical objects $\binom{n}{p}_{F_{s}\left(
x,y\right) }=\frac{F_{sn}\left( x,y\right) F_{s\left( n-1\right) }\left(
x,y\right) \cdots F_{s\left( n-p+1\right) }\left( x,y\right) }{F_{s}\left(
x,y\right) F_{2s}\left( x,y\right) \cdots F_{ps}\left( x,y\right) }$, where $%
s\in \mathbb{N}$ and $F_{sn}\left( x,y\right) $ is the bivariate $s$%
-Fibonacci polynomial sequence. We call these objects \textquotedblleft
bivariate $s$\textit{-}Fibopolynomials\textquotedblright . It turns out that
they are in fact polynomials, and when $x=y=1$ they become the known $s$%
-Fibonomials, studied in a previous work. We obtain the $Z$ transform of
sequences of the form $\prod_{i=1}^{l}F_{st_{i}n+m_{i}}^{k_{i}}\left(
x,y\right) $, and from this result we obtain the $Z$ transform of the
sequence of bivariate $s$-Fibopolynomials. Then we establish connections
between these two kind of sequences. We also obtain expressions for the
partial derivatives of $\binom{n}{p}_{F_{s}\left( x,y\right) }$.
\end{abstract}

\section{\label{Sec1}Introduction}

We use $\mathbb{N}$ for the natural numbers and $\mathbb{N}^{\prime }$ for $%
\mathbb{N\cup }\left\{ 0\right\} $. Throughout the work $s$ will denote a
given natural number.

We use the standard notation $F_{n}\left( x,y\right) $ and $L_{n}\left(
x,y\right) $ for the sequences of bivariate Fibonacci polynomials and
bivariate Lucas polynomials, defined by the recurrences $F_{n+2}\left(
x,y\right) =xF_{n-1}\left( x,y\right) +yF_{n}\left( x,y\right) $, $%
F_{0}\left( x,y\right) =0$, $F_{1}\left( x,y\right) =1$, and $L_{n+2}\left(
x,y\right) =xL_{n-1}\left( x,y\right) +yL_{n}\left( x,y\right) $, $%
L_{0}\left( x,y\right) =2$, $L_{1}\left( x,y\right) =x$, respectively, and
extended to negative integers as $F_{-n}\left( x,y\right) =-\left( -y\right)
^{-n}F_{n}\left( x,y\right) $ and $L_{-n}\!\left( x,y\right) =\left(
-y\right) ^{-n}L_{n}\!\left( x,y\right) $, $n\in \mathbb{N}$. Clearly $%
F_{n}\left( x,y\right) $ and $L_{n}\left( x,y\right) $ are monic
polynomials. The degree of $F_{n}\left( x,y\right) $ is $n-1$ in $x$ and $%
\lfloor \frac{n-1}{2}\rfloor $ in $y$, and the degree of $L_{n}\left(
x,y\right) $ is $n$ in $x$ and $\lfloor \frac{n}{2}\rfloor $ in $y$. Observe
that (for $n\in \mathbb{N}$), $F_{-n}\left( x,y\right) $ and $L_{-n}\left(
x,y\right) $ are polynomials in $x$ with negative powers of $y$ (times some
constants) as coefficients. It is clear that the case $y=1$ corresponds to
the Fibonacci and Lucas polynomials $F_{n}\left( x\right) $ and $L_{n}\left(
x\right) $ (see A011973 and A034807 of Sloane's \textit{Encyclopedia}), and
the case $x=y=1$ corresponds to the Fibonacci and Lucas number sequences $%
F_{n}$ and $L_{n}$ (A000045 and A000032 of Sloane's \textit{Encyclopedia},
respectively). Some positive indexed bivariate Fibonacci polynomials are $%
F_{2}\left( x,y\right) =x$, $F_{3}\left( x,y\right) =x^{2}+y$, $F_{4}\left(
x,y\right) =x^{3}+2xy$, $F_{5}\left( x,y\right) =x^{4}+3x^{2}y+y^{2}$, and
so on, and some positive indexed bivariate Lucas polynomials are $%
L_{2}\left( x,y\right) =x^{2}+2y$, $L_{3}\left( x,y\right) =x^{3}+3xy$, $%
L_{4}\left( x,y\right) =x^{4}+4x^{2}y+2y^{2}$, $L_{5}\left( x,y\right)
=x^{5}+5x^{3}y+5xy^{2}$, and so on. Some negative indexed bivariate
Fibonacci and Lucas polynomials are $F_{-1}\left( x,y\right) =y^{-1}$, $%
L_{-1}\left( x,y\right) =-xy^{-1}$, $F_{-2}\left( x,y\right) =-xy^{-2}$, $%
L_{-2}\left( x,y\right) =\left( x^{2}+2y\right) y^{-2}$, $F_{-3}\left(
x,y\right) =\left( x^{2}+y\right) y^{-3}$, $L_{-4}\left( x,y\right) =\left(
x^{4}+4x^{2}y+2y^{2}\right) y^{-4}$, and so on. A \textit{bivariate
generalized Fibonacci polynomial} (or \textit{bivariate Gibonacci polynomial}%
), denoted by $G_{n}\left( x,y\right) $, is defined by the recurrence $%
G_{n}\left( x,y\right) =xG_{n-1}\left( x,y\right) +yG_{n-2}\left( x,y\right) 
$, $n\geq 2$, where $G_{0}\left( x,y\right) $ and $G_{1}\left( x,y\right) $
are given (arbitrary) initial conditions. It is easy to see that 
\begin{equation}
G_{n}\left( x,y\right) =yG_{0}\left( x,y\right) F_{n-1}\left( x,y\right)
+G_{1}\left( x,y\right) F_{n}\left( x,y\right) .  \label{1.1}
\end{equation}

We will be using extensively Binet's formulas (without further comments):%
\begin{equation}
F_{n}\left( x,y\right) =\frac{1}{\sqrt{x^{2}+4}}\left( \alpha ^{n}\left(
x,y\right) -\beta ^{n}\left( x,y\right) \right) \ \ \ \text{and}\ \ \
L_{n}\left( x,y\right) =\alpha ^{n}\left( x,y\right) +\beta ^{n}\left(
x,y\right) ,  \label{1.2}
\end{equation}%
where 
\begin{equation}
\alpha \left( x,y\right) =\frac{1}{2}\left( x+\sqrt{x^{2}+4y}\right) \text{
\ \ \ and \ \ \ }\beta \left( x,y\right) =\frac{1}{2}\left( x-\sqrt{x^{2}+4y}%
\right) \text{,}  \label{1.3}
\end{equation}

We will use also some relations involving $\alpha \!\left( x,y\right) $ and $%
\beta \!\left( x,y\right) $, as $\alpha \!\left( x,y\right) +\beta \!\left(
x,y\right) =x$ and $\alpha \!\left( x,y\right) \beta \!\left( x,y\right) =-y$%
, with no additional comments. The basics of the Fibonacci world is
contained in the famous references \cite{K} and \cite{V}. What we will use
about bivariate Fibonacci and Lucas polynomials is contained in \cite{Ca}, 
\cite{Sw} and \cite{Yu}.

There are certainly lots of identities involving Fibonacci and Lucas
numbers, and the list continues increasing trough the years. But the list is
not so large when bivariate Fibonacci and Lucas polynomials are involved. We
will need some of these identities, to be used in the proofs of the results
presented in this work, and in some of the given examples as well. We give
now a short list.

For $p\in \mathbb{N}$ we have%
\begin{equation}
\frac{F_{\left( 2p-1\right) s}\left( x,y\right) }{F_{s}\left( x,y\right) }%
=\sum_{k=0}^{p-1}\left( -y\right) ^{sk}L_{2\left( p-k-1\right) s}\left(
x,y\right) -\left( -y\right) ^{s\left( p-1\right) }.  \label{1.6}
\end{equation}%
\begin{equation}
\frac{F_{2ps}\left( x,y\right) }{F_{s}\left( x,y\right) }=\sum_{k=0}^{p-1}%
\left( -y\right) ^{sk}L_{\left( 2p-2k-1\right) s}\left( x,y\right) .
\label{1.7}
\end{equation}

(The proofs of (\ref{1.6}) and (\ref{1.7}) are easy exercises by using
Binet's formulas.) Moreover, we have also that%
\begin{equation}
\frac{F_{ps}\left( x,y\right) }{F_{s}\left( x,y\right) }=F_{p}\left(
L_{s}\left( x,y\right) ,-\left( -y\right) ^{s}\right) .  \label{1.71}
\end{equation}

(See \cite{Ca}.) We comment in passing that formulas (\ref{1.6}) and (\ref%
{1.7}) (or (\ref{1.71})) shows the well-known fact that $F_{ps}\left(
x,y\right) $ is divisible $F_{s}\left( x,y\right) $ (see \cite{Hog2},
Theorem 6). Some examples are the following%
\begin{equation*}
\frac{F_{2s}\left( x,y\right) }{F_{s}\left( x,y\right) }=L_{s}\left(
x,y\right) .
\end{equation*}%
\begin{equation*}
\frac{F_{3s}\left( x,y\right) }{F_{s}\left( x,y\right) }=L_{2s}\left(
x,y\right) +\left( -y\right) ^{s}=L_{s}^{2}\left( x,y\right) -\left(
-y\right) ^{s}.
\end{equation*}%
\begin{equation*}
\frac{F_{4s}\left( x,y\right) }{F_{s}\left( x,y\right) }=L_{3s}\left(
x,y\right) +\left( -y\right) ^{s}L_{s}\left( x,y\right) =L_{s}\left(
x,y\right) \left( L_{s}^{2}\left( x,y\right) -2\left( -y\right) ^{s}\right) .
\end{equation*}%
\begin{equation*}
\frac{F_{5s}\left( x,y\right) }{F_{s}\left( x,y\right) }=L_{4s}\left(
x,y\right) +\left( -y\right) ^{s}L_{2s}\left( x,y\right)
+y^{2s}=L_{s}^{4}\left( x,y\right) -3\left( -y\right) ^{s}L_{s}^{2}\left(
x,y\right) +y^{2s}.
\end{equation*}

We have also the identities%
\begin{equation}
F_{s\left( n+1\right) }\left( x,y\right) -\left( -y\right) ^{s}F_{s\left(
n-1\right) }\left( x,y\right) =F_{s}\left( x,y\right) L_{sn}\left(
x,y\right) .  \label{1.8}
\end{equation}%
\begin{equation}
F_{s\left( n+1\right) }\left( x,y\right) +\left( -y\right) ^{s}F_{s\left(
n-1\right) }\left( x,y\right) =L_{s}\left( x,y\right) F_{sn}\left(
x,y\right) .  \label{1.81}
\end{equation}

(The version $x=y=s=1$ of (\ref{1.8}), that is $F_{n+1}+F_{n-1}=L_{n}$, is a
famous identity. The same version for (\ref{1.81}) gives us simply the
Fibonacci recurrence.)

For $a,b,c,d,r\in \mathbb{Z}$ such that $a+b=c+d$, we have the so-called
\textquotedblleft index-reduction formulas\textquotedblright :%
\begin{equation}
F_{a}\left( x,y\right) F_{b}\left( x,y\right) -F_{c}\left( x,y\right)
F_{d}\left( x,y\right) =\left( -y\right) ^{r}\left( F_{a-r}\left( x,y\right)
F_{b-r}\left( x,y\right) -F_{c-r}\left( x,y\right) F_{d-r}\left( x,y\right)
\right) .  \label{1.9}
\end{equation}%
\begin{equation}
L_{a}\left( x,y\right) F_{b}\left( x,y\right) -L_{c}\left( x,y\right)
F_{d}\left( x,y\right) =\left( -y\right) ^{r}\left( L_{a-r}\left( x,y\right)
F_{b-r}\left( x,y\right) -L_{c-r}\left( x,y\right) F_{d-r}\left( x,y\right)
\right) .  \label{1.10}
\end{equation}

(See \cite{Jh}, where the case $x=y=1$ is discussed.) Two versions of (\ref%
{1.9}) and (\ref{1.10}), which will be used several times in this work, are
obtained by setting $a=M$, $b=N$, $c=M+K$, $d=r=N-K$, with $M,N,K\in \mathbb{%
Z}$. What we get is%
\begin{equation}
F_{M}\left( x,y\right) F_{N}\left( x,y\right) -F_{M+K}\left( x,y\right)
F_{N-K}\left( x,y\right) =\left( -y\right) ^{N-K}F_{M+K-N}\left( x,y\right)
F_{K}\left( x,y\right) ,  \label{1.11}
\end{equation}%
and%
\begin{equation}
L_{M}\left( x,y\right) F_{N}\left( x,y\right) -L_{M+K}\left( x,y\right)
F_{N-K}\left( x,y\right) =\left( -y\right) ^{N-K}L_{M+K-N}\left( x,y\right)
F_{K}\left( x,y\right) ,  \label{1.12}
\end{equation}%
respectively. (We give below a proof of these identities. See (\ref{2.12})
and (\ref{2.14}).) In fact, what we will be using are identities which in
turn are versions of (\ref{1.11}) and (\ref{1.12}), obtained from them with
some identification of the indices $M$, $N$, $K$ with other indices.

For a given bivariate Gibonacci polynomial sequence $G_{n}\left( x,y\right)
=\left( G_{0}\left( x,y\right) ,G_{1}\left( x,y\right) ,G_{2}\left(
x,y\right) ,\ldots \right) $, the \textit{bivariate }$s$-\textit{Gibonacci
polynomial sequence }$G_{sn}\left( x,y\right) $ is $G_{sn}\left( x,y\right)
=\left( G_{0}\left( x,y\right) ,G_{s}\left( x,y\right) ,G_{2s}\left(
x,y\right) ,\ldots \right) $, and the\textit{\ bivariate }$s$\textit{%
-Gibonacci polynomial factorial} of $G_{sn}\left( x,y\right) $, denoted by $%
\left( G_{n}\left( x,y\right) !\right) _{s}$, is $\left( G_{n}\left(
x,y\right) !\right) _{s}=G_{sn}\left( x,y\right) G_{s\left( n-1\right)
}\left( x,y\right) \cdots G_{s}\left( x,y\right) $. Given $n\in \mathbb{N}%
^{\prime }$ and $k\in \left\{ 0,1,\ldots ,n\right\} $, the \textit{bivariate 
}$s$-\textit{Gibopolynomial }(or bivariate $s$-Gibopolynomial coefficient),
denoted by $\binom{n}{k}_{G_{s}\left( x,y\right) }$, is defined by $\binom{n%
}{0}_{G_{s}\left( x,y\right) }=\binom{n}{n}_{G_{s}\left( x,y\right) }=1$, and%
\begin{equation}
\binom{n}{k}_{G_{s}\left( x,y\right) }=\frac{\left( G_{n}\left( x,y\right)
!\right) _{s}}{\left( G_{k}\left( x,y\right) !\right) _{s}\left(
G_{n-k}\left( x,y\right) !\right) _{s}},\text{ \ \ }k=1,2,\ldots ,n-1.
\label{1.13}
\end{equation}

That is, for $k\in \left\{ 1,2,\ldots ,n-1\right\} $ we have that%
\begin{equation}
\binom{n}{k}_{G_{s}\left( x,y\right) }=\frac{G_{sn}\left( x,y\right)
G_{s\left( n-1\right) }\left( x,y\right) \cdots G_{s\left( n-k+1\right)
}\left( x,y\right) }{G_{s}\left( x,y\right) G_{2s}\left( x,y\right) \cdots
G_{ks}\left( x,y\right) }.  \label{1.14}
\end{equation}

Plainly we have symmetry for bivariate $s$-Gibopolynomials%
\begin{equation*}
\binom{n}{k}_{G_{s}\left( x,y\right) }=\binom{n}{n-k}_{G_{s}\left(
x,y\right) }.
\end{equation*}

In the case of bivariate $s$-Fibopolynomials, we can use the identity%
\begin{equation*}
F_{s\left( n-k\right) +1}\left( x,y\right) F_{sk}\left( x,y\right)
+yF_{sk-1}\left( x,y\right) F_{s\left( n-k\right) }\left( x,y\right)
=F_{sn}\left( x,y\right) ,
\end{equation*}%
(which comes from (\ref{1.11}) with $M=sn$, $N=1$ and $K=-sk+1$), to
conclude that%
\begin{equation}
\binom{n}{k}_{F\!_{s}\!\left( x,y\right) }=F_{s\left( n-k\right) +1}\left(
x,y\right) \binom{n-1}{k-1}_{F\!_{s}\!\left( x,y\right) }+yF_{sk-1}\left(
x,y\right) \binom{n-1}{k}_{F\!_{s}\!\left( x,y\right) }.  \label{2.121}
\end{equation}

Formula (\ref{2.121}) shows (with a simple induction argument) that
bivariate $s$-Fibopolynomials are in fact polynomials in $x$ and $y$.
Moreover, $\binom{n}{k}_{F\!_{s}\!\left( x,y\right) }$ is a polynomial of
degree $sk\left( n-k\right) $ in $x$, and of degree $\lfloor \frac{sk\left(
n-k\right) }{2}\rfloor $ in $y$. The case $s=x=y=1$ corresponds to
Fibonomial coefficients $\binom{n}{k}_{F\!\!}$ \negthinspace , introduced by
V. E. Hoggatt, Jr. \cite{Hog} in 1967, and the case $x=y=1$ corresponds to $%
s $-Fibonomials $\binom{n}{k}_{F_{s}\left( 1,1\right) }$, first mentioned
also in \cite{Hog}, and studied recently in \cite{Pi2}. (For $s=1,2,3,$ the $%
s$-Fibonomial sequences are A010048, A034801 and A034802 of Sloane's \textit{%
Encyclopedia}, respectively.) However, bivariate $s$-Gibopolynomials are in
general rational functions in $x$ and $y$. For example, the bivariate $2$%
-Lucapolynomial $\binom{4}{2}_{L_{2}\left( x,y\right) }$ is%
\begin{eqnarray*}
\binom{4}{2}_{L_{2}\left( x,y\right) } &=&\frac{L_{8}\left( x,y\right)
L_{6}\left( x,y\right) }{L_{2}\left( x,y\right) L_{4}\left( x,y\right) } \\
&=&\frac{\left( x^{4}+4x^{2}y+y^{2}\right) \left(
x^{8}+8x^{6}y+20x^{4}y^{2}+16x^{2}y^{3}+2y^{4}\right) }{x^{4}+4x^{2}y+2y^{2}}%
.
\end{eqnarray*}

Observe that identities (\ref{1.6}), (\ref{1.7}) and (\ref{1.71}) refer to
bivariate $s$-Fibopolynomials $\binom{n}{1}_{F\!_{s}\!\left( x,y\right) }$,
with $n=2p+1$, $n=2p$ and $n=p$, respectively. We present now some examples
of bivariate $s$-Fibopolynomals $\binom{n}{k}_{F\!_{s}\!\left( x,y\right) }$%
, as triangular arrays, where $n$ stands for lines and $k=0,1,\ldots ,n$
stands for columns.

For $s=1$ we have%
\begin{equation*}
\begin{array}{ccccccccccccc}
&  &  &  &  &  & ^{1} &  &  &  &  &  &  \\ 
&  &  &  &  & ^{1} &  & ^{1} &  &  &  &  &  \\ 
&  &  &  & ^{1} &  & ^{x} &  & ^{1} &  &  &  &  \\ 
&  &  & ^{1} &  & ^{x^{2}+y} &  & ^{x^{2}+y} &  & ^{1} &  &  &  \\ 
&  & ^{1} &  & ^{x^{3}+2xy} &  & ^{x^{4}+3x^{2}y+2y^{2}} &  & ^{x^{3}+2xy} & 
& ^{1} &  &  \\ 
& ^{1}%
\begin{array}{cc}
& 
\end{array}
&  & ^{x^{4}+3x^{2}y+y^{2}} &  & ^{\substack{ x^{6}+5x^{4}y  \\ %
+7x^{2}y^{2}+2y^{3}}} &  & ^{\substack{ x^{6}+5x^{4}y  \\ %
+7x^{2}y^{2}+2y^{3} }} &  & ^{x^{4}+3x^{2}y+y^{2}} &  & 
\begin{array}{cc}
& 
\end{array}%
^{1} &  \\ 
&  & \vdots &  & \vdots &  & \vdots &  & \vdots &  & \vdots &  & 
\end{array}%
\end{equation*}

For $s=2$ we have%
\begin{equation*}
\begin{array}{ccccccccccc}
&  &  &  &  & ^{1} &  &  &  &  &  \\ 
&  &  &  & ^{1} &  & ^{1} &  &  &  &  \\ 
&  &  & ^{1} &  & ^{x^{2}+2y} &  & ^{1} &  &  &  \\ 
&  & ^{1} &  & ^{\substack{ x^{4}+4x^{2}y  \\ +3y^{2}}} &  & ^{\substack{ %
x^{4}+4x^{2}y  \\ +3y^{2}}} &  & ^{1} &  &  \\ 
& ^{1} &  & _{\substack{ +10x^{2}y^{2}  \\ +4y^{3}}}^{x^{6}+6x^{4}y} &  & 
_{\substack{ +21x^{4}y^{2}  \\ +20x^{2}y^{3}  \\ +6y^{4}}}^{x^{8}+8x^{6}y} & 
& _{\substack{ +10x^{2}y^{2}  \\ +4y^{3}}}^{x^{6}+6x^{4}y} &  & ^{1} &  \\ 
^{1%
\begin{array}{cc}
& 
\end{array}%
} &  & _{\substack{ +21x^{4}y^{2}  \\ +20x^{2}y^{3}  \\ +5y^{4}}}%
^{x^{8}+8x^{6}y} &  & _{\substack{ +55x^{8}y^{2}+120x^{6}y^{3}  \\ %
+127x^{4}y^{4}+60x^{2}y^{5}  \\ +10y^{6}}}^{x^{12}+12x^{10}y} &  & 
_{\substack{ +55x^{8}y^{2}+120x^{6}y^{3}  \\ +127x^{4}y^{4}+60x^{2}y^{5}  \\ %
+10y^{6}}}^{x^{12}+12x^{10}y} &  & _{\substack{ +21x^{4}y^{2}  \\ %
+20x^{2}y^{3}  \\ +5y^{4}}}^{x^{8}+8x^{6}y} &  & 
\begin{array}{cc}
& 
\end{array}%
^{1} \\ 
& \vdots &  & \vdots &  & \vdots &  & \vdots &  & \vdots & 
\end{array}%
\end{equation*}

In this article we work with the $Z$ transform of complex sequences $a_{n}$.
Some definitions and basic facts about this tool, and some preliminary
results as well, are presented in section 2. Naively we can think of the $Z$
transform of a sequence $a_{n}=\left( a_{0},a_{1},\ldots \right) $ as if
this were the complex function $F\left( z\right) $ which comes from the
generating function $G\left( x\right) $ of $a_{n}$, with $x$ replaced by $%
z^{-1}$. For example, it is well-known that the generating function of the
Fibonacci sequence $F_{n}$ is given by $G\left( x\right) =x\left(
1-x-x^{2}\right) ^{-1}$. It turns out that the $Z$ transform of $F_{n}$ (as
we will see in section \ref{Sec2}) is the complex function $F\left( z\right)
=z\left( z^{2}-z-1\right) ^{-1}=G\left( z^{-1}\right) $.

The problem of finding the generating function of the $k$-th power of a
`Fibonacci-type' sequence, was first considered by Riordan \cite{R}, Carlitz 
\cite{C}, and Horadam \cite{Hor}. However, in their works there are no
explicit reference to the standard Fibonacci sequence $F_{n}$ case. Later,
Shannon \cite{Sh} obtains an explicit generating function for $F_{n}^{k}$.
Now we know that the corresponding $Z$ transform of the sequence $F_{n}^{k}$
is given by%
\begin{equation}
\mathcal{Z}\left( F_{n}^{k}\right) =z\frac{\sum\limits_{i=0}^{k}\!\sum%
\limits_{j=0}^{i}\left( -1\right) ^{\frac{j\left( j+1\right) }{2}}\!\dbinom{%
k+1}{j}_{F\!}\!F_{i-j}^{k}\!z^{k-i}}{\sum\limits_{i=0}^{k+1}\left( -1\right)
^{\frac{i\left( i+1\right) }{2}}\dbinom{k+1}{i}_{F\!}\!z^{k+1-i}}.
\label{1.15}
\end{equation}

In \cite{Pi1} we proved the following more general result on the $Z$
transform of the sequence $F_{n+m_{1}}^{k_{1}}\cdots F_{n+m_{l}}^{k_{l}}$,
where $k_{1},k_{2},\ldots ,k_{l}\in \mathbb{N}^{\prime }$ and $%
m_{1},m_{2},\ldots ,m_{l}\in \mathbb{Z}$ are given,%
\begin{equation}
\mathcal{Z}\left( F_{n+m_{1}}^{k_{1}}\cdots F_{n+m_{l}}^{k_{l}}\right) =z%
\frac{\sum\limits_{i=0}^{k_{1}+\cdots +k_{l}}\!\sum\limits_{j=0}^{i}\left(
-1\right) ^{\frac{j\left( j+1\right) }{2}}\!\dbinom{k_{1}+\cdots +k_{l}+1}{j}%
_{F\!\!}\!F_{m_{1}+i-j}^{k_{1}}\!\cdots
\!F_{m_{l}+i-j}^{k_{l}}z^{k_{1}+\cdots +k_{l}-i}}{\sum\limits_{i=0}^{k_{1}+%
\cdots +k_{l}+1}\left( -1\right) ^{\frac{i\left( i+1\right) }{2}}\dbinom{%
k_{1}+\cdots +k_{l}+1}{i}_{F\!}\!z^{k_{1}+\cdots +k_{l}+1-i}}.  \label{1.16}
\end{equation}

From (\ref{1.16}) we obtained as corollary that the $Z$ transform of the
Fibonomial sequence $\binom{n}{p}_{F\!\!}\!$ \ is as follows%
\begin{equation}
\mathcal{Z}\left( \binom{n}{p}_{\!F}\right) =\frac{z}{\sum_{i=0}^{p+1}\left(
-1\right) ^{\frac{i\left( i+1\right) }{2}}\dbinom{p+1}{i}_{F\!}\!z^{p+1-i}}.
\label{1.17}
\end{equation}

(This result was demonstrated earlier by I. Strazdins \cite{St}, working
with a different approach.) With (\ref{1.16}) and (\ref{1.17}), was a
natural task to establish connections between products of powers of
Fibonacci sequences $F_{n+m_{1}}^{k_{1}}\cdots F_{n+m_{l}}^{k_{l}}$ and
Fibonomial sequences $\binom{n}{p}_{F\!\!}\!$ \ \negthinspace . For example,
it is possible to see (from (\ref{1.16}) and (\ref{1.17})) that the sequence 
$F_{n}^{4}=\left( 0,1,1,16,81,625,\ldots \right) $ can be written as a
linear combination of the Fibonomial sequence $\binom{n}{4}_{F\!\!}\!=\left(
0,0,0,0,1,5,40,260,\ldots \right) $ and its shifted sequences $\binom{n+i}{4}%
_{F\!\!}\!$ , $i=1,2,3$, as%
\begin{equation}
F_{n}^{4}=\dbinom{n+3}{4}_{F\!\!}\!-4\dbinom{n+2}{4}_{F\!\!}\!-4\dbinom{n+1}{%
4}_{F\!\!}\!+\dbinom{n}{4}_{F\!\!}\!.  \label{1.172}
\end{equation}

(This particular identity is known since 1970. See \cite{Phi}.)

In \cite{Pi2} we showed that (\ref{1.17}) is in fact a particular case of
the following result%
\begin{eqnarray}
&&\mathcal{Z}\left( F_{t_{1}sn+m_{1}}^{k_{1}}\cdots
F_{t_{l}sn+m_{l}}^{k_{l}}\right)  \label{1.18} \\
&=&z\frac{\sum\limits_{i=0}^{k_{1}t_{1}+\cdots
+k_{l}t_{l}}\!\sum\limits_{j=0}^{i}\left( -1\right) ^{\frac{\left(
sj+2(s+1)\right) \left( j+1\right) }{2}}\!\dbinom{k_{1}t_{1}+\cdots
+k_{l}t_{l}+1}{j}_{F\!_{s}\!}\!F_{m_{1}+t_{1}s\left( i-j\right)
}^{k_{1}}\!\cdots \!F_{m_{l}+t_{l}s\left( i-j\right)
}^{k_{l}}z^{k_{1}t_{1}+\cdots +k_{l}t_{l}-i}}{\sum\limits_{i=0}^{k_{1}t_{1}+%
\cdots +k_{l}t_{l}+1}\left( -1\right) ^{\frac{\left( si+2(s+1)\right) \left(
i+1\right) }{2}}\dbinom{k_{1}t_{1}+\cdots +k_{l}t_{l}+1}{i}%
_{F\!_{s}\!}\!z^{k_{1}t_{1}+\cdots +k_{l}t_{l}+1-i}}.  \notag
\end{eqnarray}

Now we have new parameters $s\in \mathbb{N}$ and $t_{1},t_{2},\ldots
,t_{l}\in \mathbb{N}^{\prime }$, and (\ref{1.16}) becomes the case $%
s=t_{1}=\cdots =t_{l}=1$ of (\ref{1.18}). Observe that in (\ref{1.18}) are
now involved $s$-Fibonomials $\binom{n}{p}_{F\!\!_{s}}\!$ (in this context
the $1$-Fibonomials are just the Fibonomials). Then we could see that (\ref%
{1.172}) is simply the case $s=1$ of the following identity (formula (58) of 
\cite{Pi2})%
\begin{equation}
F_{sn}^{4}=F_{s}^{4}\left( \dbinom{n+3}{4}_{F\!\!_{s}}\!+\dbinom{n}{4}%
_{F\!\!_{s}}\!+\left( \frac{3\left( -1\right) ^{s}F_{3s}}{F_{s}}+2\right)
\left( \dbinom{n+2}{4}_{F\!\!_{s}}\!+\dbinom{n+1}{4}_{F\!\!_{s}}\!\right)
\right) .  \label{1.181}
\end{equation}

In this article we will show that (\ref{1.18}) is in fact the particular
case $x=y=1$ of (formula (\ref{3.11}) of section \ref{Sec3})%
\begin{eqnarray*}
&&\mathcal{Z}\left( F_{t_{1}sn+m_{1}}^{k_{1}}\left( x,y\right) \cdots
F_{t_{l}sn+m_{l}}^{k_{l}}\left( x,y\right) \right) \\
&& \\
&=&z\frac{%
\begin{array}{c}
\sum\limits_{i=0}^{k_{1}t_{1}+\cdots
+k_{l}t_{l}}\!\sum\limits_{j=0}^{i}\left( -1\right) ^{\frac{\left(
sj+2(s+1)\right) \left( j+1\right) }{2}}\!\dbinom{k_{1}t_{1}+\cdots
+k_{l}t_{l}+1}{j}_{F\!_{s}\!\left( x,y\right) }\! \\ 
\\ 
\times F_{m_{1}+t_{1}s\left( i-j\right) }^{k_{1}}\!\left( x,y\right) \cdots
\!F_{m_{l}+t_{l}s\left( i-j\right) }^{k_{l}}\left( x,y\right) y^{\frac{%
sj\left( j-1\right) }{2}}z^{k_{1}t_{1}+\cdots +k_{l}t_{l}-i}%
\end{array}%
}{\sum\limits_{i=0}^{k_{1}t_{1}+\cdots +k_{l}t_{l}+1}\left( -1\right) ^{%
\frac{\left( si+2(s+1)\right) \left( i+1\right) }{2}}\dbinom{%
k_{1}t_{1}+\cdots +k_{l}t_{l}+1}{i}_{F\!_{s}\!\left( x,y\right) }\!y^{\frac{%
si\left( i-1\right) }{2}}z^{k_{1}t_{1}+\cdots +k_{l}t_{l}+1-i}}.
\end{eqnarray*}

This formula involve now bivariate $s$-Fibopolynomials $\binom{n}{p}%
_{F\!\!_{s}\left( x,y\right) }\!$, which are the main mathematical objects
studied in this work. Now it is possible to see that (\ref{1.181}) is in
fact the particular case $x=y=1$ of the following identity between two
bivariate polynomials 
\begin{equation}
F_{sn}^{4}\left( x,y\right) =F_{s}^{4}\left( x,y\right) \left( 
\begin{array}{c}
\dbinom{n+3}{4}_{F_{s}\left( x,y\right) }+y^{6s}\dbinom{n}{4}_{F_{s}\left(
x,y\right) } \\ 
+\left( \frac{3\left( -y\right) ^{s}F_{3s}\left( x,y\right) }{F_{s}\left(
x,y\right) }+2y^{2s}\right) \left( \dbinom{n+2}{4}_{F_{s}\left( x,y\right)
}+y^{2s}\dbinom{n+1}{4}_{F_{s}\left( x,y\right) }\right)%
\end{array}%
\right) .  \label{1.191}
\end{equation}

(See (\ref{4.52}) and examples (\ref{4.1126}) to (\ref{4.119}) in section %
\ref{Sec4}.)

In a previous article \cite{Pi3} we considered the one variable $s$%
-Fibopolynomials $\binom{n}{p}_{F\!\!_{s}\left( x,1\right) }\!$ (also
commented in \cite{Hog}), since they appear naturally as parts of the closed
formulas of sums of products of $s$-Fibonacci polynomial sequences $%
F_{sn}\left( x\right) $ presented there. However we did not study them as we
do here with bivariate $s$-Fibopolynomials $\binom{n}{p}_{F\!\!_{s}\left(
x,y\right) }\!$. In the same manner as the results of \cite{Pi2} generalized
those of \cite{Pi1}, now this article presents results that generalize those
of \cite{Pi2}. We follow the same structure and the same kind of arguments
of the proofs presented in \cite{Pi2}, in order to prove the
\textquotedblleft bivariate polynomial generalizations of \ the results in 
\cite{Pi2}\textquotedblright . This happens mainly in sections \ref{Sec2}, %
\ref{Sec3} and \ref{Sec4}. However, two results of \cite{Pi2} are improved
here:

(1) Proposition \ref{Prop2.5} (corresponding to propositions 6 and 7 of \cite%
{Pi2}), is now demonstrated with an easier induction argument.

(2) Corollary \ref{Cor4.10} (corresponding to corollary (18) of \cite{Pi2})
is now improved in the clarity of its statement and in the clarity of its
proof as well.

After we recall the basics of $Z$ transform and establish some preliminary
results in section \ref{Sec2}, we prove our main results in section \ref%
{Sec3}. In section \ref{Sec4} we establish some corollaries of the results
proved in section \ref{Sec3}. Finally, in section \ref{Sec5} we obtain
expressions for the partial derivatives of the bivariate $s$-Fibopolynomials 
$\binom{n}{p}_{F\!\!_{s}\left( x,y\right) }$.

\section{\label{Sec2}Preliminaries}

We begin this section recalling some basic facts of the main tool used in
this article, namely the $Z$ transform. (For more details see \cite{G} and 
\cite{Vi}.) The $Z$ transform maps complex sequences $a_{n}$ into complex
(holomorphic) functions $A:U\subset \mathbb{C\rightarrow C}$ given by the
Laurent series $A\left( z\right) =\sum_{n=0}^{\infty }a_{n}z^{-n}$ (also
denoted as $\mathcal{Z}\left( a_{n}\right) $; defined outside the closure $%
\overline{D}$ of the disk $D$ of convergence of the Taylor series $%
\sum_{n=0}^{\infty }a_{n}z^{n}$). If $A\left( z\right) =\mathcal{Z}\left(
a_{n}\right) $, we also write $a_{n}=\mathcal{Z}^{-1}\left( A\left( z\right)
\right) $, and we say that the sequence $a_{n}$ is the \textit{inverse }$Z$%
\textit{\ transform} of $A\left( z\right) $. Some properties of the $Z$
transform which we will be using throughout this work are the following:
(avoiding the details of regions of convergence)

(a) $\mathcal{Z}$ is linear and injective. (Same for $\mathcal{Z}^{-1}$.)

(b) \textit{Advance-shifting property.} For $k\in \mathbb{N}$ we have%
\begin{equation}
\mathcal{Z}\left( a_{n+k}\right) =z^{k}\left( \mathcal{Z}\left( a_{n}\right)
-a_{0}-\frac{a_{1}}{z}-\cdots -\frac{a_{k-1}}{z^{k-1}}\right) .  \label{2.1}
\end{equation}

Here $a_{n+k}$ is the sequence $a_{n+k}=\left( a_{k},a_{k+1},\ldots \right) $%
.

(c) \textit{Multiplication by the sequence }$\lambda ^{n}$\textit{.} If $%
\mathcal{Z}\left( a_{n}\right) =A\left( z\right) $, then%
\begin{equation}
\mathcal{Z}\left( \lambda ^{n}a_{n}\right) =A\left( \frac{z}{\lambda }%
\right) .  \label{2.2}
\end{equation}

(d) \textit{Multiplication by the sequence }$n$\textit{.} If $\mathcal{Z}%
\left( a_{n}\right) =A\left( z\right) $, then%
\begin{equation}
\mathcal{Z}\left( na_{n}\right) =-z\frac{d}{dz}A\left( z\right) .
\label{2.201}
\end{equation}

(e) \textit{Convolution theorem.} If $a_{n}$ and $b_{n}$ are two given
sequences, then%
\begin{equation}
\mathcal{Z}\left( a_{n}\ast b_{n}\right) =\mathcal{Z}\left( a_{n}\right) 
\mathcal{Z}\left( b_{n}\right) ,  \label{2.3}
\end{equation}%
where $a_{n}\ast b_{n}=\sum_{t=0}^{n}a_{t}b_{n-t}$ is the convolution of the
sequences $a_{n}$ and $b_{n}$.

Observe that according to (\ref{2.2}), if $\mathcal{Z}\left( a_{n}\right)
=A\left( z\right) $ then%
\begin{equation}
\mathcal{Z}\left( \left( -1\right) ^{n}a_{n}\right) =A\left( -z\right) ,
\label{2.4}
\end{equation}%
and%
\begin{equation}
\mathcal{Z}\left( L_{sn+m}\left( x,y\right) a_{n}\right) =\alpha ^{m}\left(
x,y\right) A\left( \frac{z}{\alpha ^{s}\left( x,y\right) }\right) +\beta
^{m}\left( x,y\right) A\left( \frac{z}{\beta ^{s}\left( x,y\right) }\right) .
\label{2.5}
\end{equation}

For given $\lambda \in \mathbb{C}$, $\lambda \neq 0$, the $Z$ transform of
the sequence $\lambda ^{n}$ is plainly%
\begin{equation}
\mathcal{Z}\left( \lambda ^{n}\right) =\sum_{n=0}^{\infty }\frac{\lambda ^{n}%
}{z^{n}}=\frac{z}{z-\lambda },  \label{2.6}
\end{equation}%
(defined for $\left\vert z\right\vert >\left\vert \lambda \right\vert $). In
particular we have that the $Z$ transform of the constant sequence $1$ is%
\begin{equation}
\mathcal{Z}\left( 1\right) =\frac{z}{z-1}.  \label{2.7}
\end{equation}

For $m\in \mathbb{Z}$ given, the $Z$ transforms of the sequences $%
F_{sn+m}\left( x,y\right) $ and $L_{sn+m}\left( x,y\right) $ are 
\begin{equation}
\mathcal{Z}\left( F_{sn+m}\left( x,y\right) \right) =\frac{z\left(
F_{m}\left( x,y\right) z+\left( -y\right) ^{m}F_{s-m}\left( x,y\right)
\right) }{z^{2}-L_{s}\!\left( x,y\right) z+\left( -y\right) ^{s}},
\label{2.8}
\end{equation}%
and%
\begin{equation}
\mathcal{Z}\left( L_{sn+m}\left( x,y\right) \right) =\frac{z\left(
L_{m}\left( x,y\right) z-\left( -y\right) ^{m}L_{s-m}\left( x,y\right)
\right) }{z^{2}-L_{s}\!\left( x,y\right) z+\left( -y\right) ^{s}}.
\label{2.9}
\end{equation}

In fact, by using Binet's formulas and (\ref{2.6}), we have that:%
\begin{eqnarray*}
&&\mathcal{Z}\left( F_{sn+m}\left( x,y\right) \right) \\
&=&\frac{1}{\sqrt{x^{2}+4y}}\mathcal{Z}\left( \alpha ^{m}\left( x,y\right)
\left( \alpha ^{s}\left( x,y\right) \right) ^{n}-\beta ^{m}\left( x,y\right)
\left( \beta ^{s}\left( x,y\right) \right) ^{n}\right) \\
&=&\frac{1}{\sqrt{x^{2}+4y}}\left( \alpha ^{m}\left( x,y\right) \frac{z}{%
z-\alpha ^{s}\left( x,y\right) }-\beta ^{m}\left( x,y\right) \frac{z}{%
z-\beta ^{s}\left( x,y\right) }\right) \\
&=&\frac{z}{\sqrt{x^{2}+4y}}\left( \frac{\left( \alpha ^{m}\left( x,y\right)
-\beta ^{m}\left( x,y\right) \right) z+\alpha ^{m}\left( x,y\right) \beta
^{m}\left( x,y\right) \left( -\beta ^{s-m}\left( x,y\right) +\alpha
^{s-m}\left( x,y\right) \right) }{z^{2}-L_{s}\!\left( x,y\right) z+\left(
-y\right) ^{s}}\right) \\
&=&\frac{z\left( F_{m}\left( x,y\right) z+\left( -y\right) ^{m}F_{s-m}\left(
x,y\right) \right) }{z^{2}-L_{s}\!\left( x,y\right) z+\left( -y\right) ^{s}},
\end{eqnarray*}%
which shows (\ref{2.7}). Similarly%
\begin{eqnarray*}
&&\mathcal{Z}\left( L_{sn+m}\left( x,y\right) \right) \\
&=&\mathcal{Z}\left( \alpha ^{m}\left( x,y\right) \left( \alpha ^{s}\left(
x,y\right) \right) ^{n}+\beta ^{m}\left( x,y\right) \left( \beta ^{s}\left(
x,y\right) \right) ^{n}\right) \\
&=&\alpha ^{m}\left( x,y\right) \frac{z}{z-\alpha ^{s}\left( x,y\right) }%
+\beta ^{m}\left( x,y\right) \frac{z}{z-\beta ^{s}\left( x,y\right) } \\
&=&z\left( \frac{\left( \alpha ^{m}\left( x,y\right) +\beta ^{m}\left(
x,y\right) \right) z-\alpha ^{m}\left( x,y\right) \beta ^{m}\left(
x,y\right) \left( \beta ^{s-m}\left( x,y\right) +\alpha ^{s-m}\left(
x,y\right) \right) }{z^{2}-L_{s}\!\left( x,y\right) z+\left( -y\right) ^{s}}%
\right) \\
&=&\frac{z\left( L_{m}\left( x,y\right) z-\left( -y\right) ^{m}L_{s-m}\left(
x,y\right) \right) }{z^{2}-L_{s}\!\left( x,y\right) z+\left( -y\right) ^{s}},
\end{eqnarray*}%
which shows (\ref{2.8}). In particular we have%
\begin{equation}
\mathcal{Z}\left( F_{sn}\left( x,y\right) \right) =\frac{zF\!_{s}\!\left(
x,y\right) }{z^{2}-L_{s}\!\left( x,y\right) z+\left( -y\right) ^{s}},
\label{2.10}
\end{equation}%
and%
\begin{equation}
\mathcal{Z}\left( L_{sn}\left( x,y\right) \right) =\frac{z\left(
2z-L_{s}\!\left( x,y\right) \right) }{z^{2}-L_{s}\!\left( x,y\right)
z+\left( -y\right) ^{s}}.  \label{2.11}
\end{equation}

If we write $\mathcal{Z}\left( F_{sn+m}\left( x,y\right) \right) $ as%
\begin{equation}
\mathcal{Z}\left( F_{sn+m}\left( x,y\right) \right) =\frac{F_{m}\left(
x,y\right) }{F\!_{s}\!\left( x,y\right) }\frac{z^{2}F\!_{s}\!\left(
x,y\right) }{z^{2}-L_{s}\!\left( x,y\right) z+\left( -y\right) ^{s}}+\frac{%
\left( -y\right) ^{m}F_{s-m}\left( x,y\right) }{F\!_{s}\!\left( x,y\right) }%
\frac{zF\!_{s}\!\left( x,y\right) }{z^{2}-L_{s}\!\left( x,y\right) z+\left(
-y\right) ^{s}},  \label{2.12}
\end{equation}%
we see at once that%
\begin{equation*}
F\!_{s}\!\left( x,y\right) F_{sn+m}\left( x,y\right) -F_{m}\left( x,y\right)
F_{s\left( n+1\right) }\left( x,y\right) =\left( -y\right) ^{m}F_{s-m}\left(
x,y\right) F_{sn}\left( x,y\right) ,
\end{equation*}%
which is essentially (\ref{1.11}). Similarly, if we write $\mathcal{Z}\left(
L_{sn+m}\left( x,y\right) \right) $ as%
\begin{equation}
\mathcal{Z}\left( L_{sn+m}\left( x,y\right) \right) =\frac{L_{m}\left(
x,y\right) }{F\!_{s}\!\left( x,y\right) }\frac{z^{2}F\!_{s}\!\left(
x,y\right) }{z^{2}-L_{s}\!\left( x,y\right) z+\left( -y\right) ^{s}}-\frac{%
\left( -y\right) ^{m}L_{s-m}\left( x,y\right) }{F\!_{s}\!\left( x,y\right) }%
\frac{zF\!_{s}\!\left( x,y\right) }{z^{2}-L_{s}\!\left( x,y\right) z+\left(
-y\right) ^{s}},  \label{2.14}
\end{equation}%
we obtain that%
\begin{equation*}
L_{sn+m}\left( x,y\right) F\!_{s}\!\left( x,y\right) -L_{m}\left( x,y\right)
F_{s\left( n+1\right) }\left( x,y\right) =-\left( -y\right)
^{m}L_{s-m}\left( x,y\right) F_{sn}\left( x,y\right) ,
\end{equation*}%
which is essentially (\ref{1.12}).

Let us use the $Z$ transform to prove that%
\begin{equation}
nL_{n}\left( x,y\right) -xF_{n}\left( x,y\right) =\left( x^{2}+4y\right)
F_{n}\left( x,y\right) \ast F_{n}\left( x,y\right) .  \label{2.156}
\end{equation}

We will use (\ref{2.10}) and (\ref{2.11}) with $s=1$. First note the
according to (\ref{2.201}) we have that%
\begin{equation*}
\mathcal{Z}\left( nL_{n}\left( x,y\right) \right) =-z\frac{d}{dz}\frac{%
z\left( 2z-x\right) }{z^{2}-xz+\left( -y\right) ^{s}}=z\frac{xz^{2}+4yz-xy}{%
\left( z^{2}-xz-y\right) ^{2}}.
\end{equation*}

Thus we have 
\begin{equation}
\mathcal{Z}\left( nL_{n}\left( x,y\right) \right) -\mathcal{Z}\left(
xF_{n}\left( x,y\right) \right) =z\frac{xz^{2}+4yz-xy}{\left(
z^{2}-xz-y\right) ^{2}}-\frac{xz}{z^{2}-xz-y}=\frac{\left( x^{2}+4y\right)
z^{2}}{\left( z^{2}-xz-y\right) ^{2}}.  \label{2.157}
\end{equation}

Then, (\ref{2.156}) follows from (\ref{2.157}) and convolution theorem (\ref%
{2.3}).

Now we begin with a list of preliminary results that will be used in
sections \ref{Sec3} and \ref{Sec4}.

\begin{proposition}
\label{Prop2.1}\textit{Let }$k\in \mathbb{N}^{\prime }$\textit{\ be given.
We have }%
\begin{equation}
\left( -1\right) ^{s+1}\dprod\limits_{j=0}^{k}\left( z-\alpha ^{sj}\left(
x,y\right) \beta ^{s\left( k-j\right) }\left( x,y\right) \right)
=\sum_{i=0}^{k+1}\left( -1\right) ^{\frac{\left( si+2(s+1)\right) \left(
i+1\right) }{2}}\dbinom{k+1}{i}_{F\!_{s}\!\left( x,y\right) }y^{\frac{%
si\left( i-1\right) }{2}}z^{k+1-i}.  \label{2.16}
\end{equation}
\end{proposition}

\begin{proof}
We proceed by induction on $k$. For $k=0$ the result is clearly true (both
sides are equal to $\left( -1\right) ^{s+1}\left( z-1\right) $). Let us
suppose the formula is true for a given $k\in \mathbb{N}$. Then we have%
\begin{eqnarray*}
&&\left( -1\right) ^{s+1}\dprod\limits_{j=0}^{k+1}\left( z-\alpha
^{sj}\left( x,y\right) \beta ^{s\left( k+1-j\right) }\left( x,y\right)
\right) \\
&=&\left( -1\right) ^{s+1}\left( z-\alpha ^{s\left( k+1\right) }\left(
x,y\right) \right) \beta ^{s\left( k+1\right) }\left( x,y\right)
\dprod\limits_{j=0}^{k}\left( \frac{z}{\beta ^{s}\left( x,y\right) }-\alpha
^{sj}\left( x,y\right) \beta ^{s\left( k-j\right) }\left( x,y\right) \right)
,
\end{eqnarray*}

The induction hypothesis allows us to write%
\begin{eqnarray*}
&&\left( -1\right) ^{s+1}\dprod\limits_{j=0}^{k+1}\left( z-\alpha
^{sj}\left( x,y\right) \beta ^{s\left( k+1-j\right) }\left( x,y\right)
\right) \\
&=&\left( z-\alpha ^{s\left( k+1\right) }\left( x,y\right) \right) \beta
^{s\left( k+1\right) }\left( x,y\right) \sum_{i=0}^{k+1}\left( -1\right) ^{%
\frac{\left( si+2(s+1)\right) \left( i+1\right) }{2}}\dbinom{k+1}{i}%
_{F\!_{s}\!\left( x,y\right) }y^{\frac{si\left( i-1\right) }{2}}\left( \frac{%
z}{\beta ^{s}\left( x,y\right) }\right) ^{k+1-i} \\
&=&\left( z-\alpha ^{s\left( k+1\right) }\left( x,y\right) \right)
\sum_{i=0}^{k+1}\left( -1\right) ^{\frac{\left( si+2(s+1)\right) \left(
i+1\right) }{2}}\dbinom{k+1}{i}_{F\!_{s}\!\left( x,y\right) }\beta
^{si}\left( x,y\right) y^{\frac{si\left( i-1\right) }{2}}z^{k+1-i}.
\end{eqnarray*}

Some further simplifications give us%
\begin{eqnarray*}
&&\left( -1\right) ^{s+1}\dprod\limits_{j=0}^{k+1}\left( z-\alpha
^{sj}\left( x,y\right) \beta ^{s\left( k+1-j\right) }\left( x,y\right)
\right) \\
&=&\sum_{i=0}^{k+1}\left( -1\right) ^{\frac{\left( si+2(s+1)\right) \left(
i+1\right) }{2}}\dbinom{k+1}{i}_{F\!_{s}\!\left( x,y\right) }\beta
^{si}\left( x,y\right) y^{\frac{si\left( i-1\right) }{2}}z^{k+2-i} \\
&&-\alpha ^{s\left( k+1\right) }\left( x,y\right) \sum_{i=1}^{k+2}\left(
-1\right) ^{\frac{\left( s\left( i-1\right) +2(s+1)\right) i}{2}}\dbinom{k+1%
}{i-1}_{F\!_{s}\!\left( x,y\right) }\beta ^{s\left( i-1\right) }\left(
x,y\right) y^{\frac{s\left( i-1\right) \left( i-2\right) }{2}}z^{k+2-i} \\
&=&\sum_{i=0}^{k+2}\left( -1\right) ^{\frac{\left( si+2(s+1)\right) \left(
i+1\right) }{2}}\dbinom{k+2}{i}_{F\!_{s}\!\left( x,y\right) }\frac{1}{%
F_{s\left( k+2\right) }\left( x,y\right) } \\
&&\left( 
\begin{array}{c}
\beta ^{si}\left( x,y\right) F_{s\left( k+2-i\right) }\left( x,y\right) \\ 
+\left( -1\right) ^{-s\left( i+1\right) }\alpha ^{s\left( k+1\right) }\left(
x,y\right) \beta ^{s\left( i-1\right) }\left( x,y\right) y^{s\left(
1-i\right) }F_{si}\left( x,y\right)%
\end{array}%
\right) y^{\frac{si\left( i-1\right) }{2}}z^{k+2-i} \\
&=&\sum_{i=0}^{k+2}\left( -1\right) ^{\frac{\left( si+2(s+1)\right) \left(
i+1\right) }{2}}\dbinom{k+2}{i}_{F\!_{s}\!\left( x,y\right) }y^{\frac{%
si\left( i-1\right) }{2}}z^{k+2-i}
\end{eqnarray*}

as wanted. Here we used that%
\begin{equation*}
\beta ^{si}\left( x,y\right) F_{s\left( k+2-i\right) }\left( x,y\right)
+\left( -1\right) ^{-s\left( i+1\right) }\alpha ^{s\left( k+1\right) }\left(
x,y\right) \beta ^{s\left( i-1\right) }\left( x,y\right) y^{s\left(
1-i\right) }F_{si}\left( x,y\right) =F_{s\left( k+2\right) }\left(
x,y\right) ,
\end{equation*}%
which can be proved easily by using Binet's formulas.
\end{proof}

We will denote the $\left( k\!+\!1\right) $-th degree $z$-polynomial of the
right-hand side (or left-hand side) of (\ref{2.16}) as $D_{s,k+1}\!\left(
x,y;z\right) $.

We claim that if $k$ is even, $k=2p$ say, then 
\begin{equation}
D_{s,2p+1}\left( x,y;z\right) =\left( -1\right) ^{s+1}\left( z-\left(
-y\right) ^{sp}\right) \prod_{j=0}^{p-1}\left( z^{2}-\left( -y\right)
^{sj}L_{2s\left( p-j\right) }\left( x,y\right) z+y^{2ps}\right) .  \notag
\end{equation}

In fact, we have%
\begin{eqnarray*}
&&D_{s,2p+1}\left( x,y;z\right) \\
&=&\left( -1\right) ^{s+1}\prod_{j=0}^{2p}\left( z-\alpha ^{sj}\left(
x,y\right) \beta ^{s\left( 2p-j\right) }\left( x,y\right) \right) \\
&=&\left( -1\right) ^{s+1}\left( z-\alpha ^{sp}\left( x,y\right) \beta
^{sp}\left( x,y\right) \right) \\
&&\times \left( \prod_{j=0}^{p-1}\left( z-\alpha ^{sj}\left( x,y\right)
\beta ^{s\left( 2p-j\right) }\left( x,y\right) \right) \right) \left(
\prod_{j=p+1}^{2p}\left( z-\alpha ^{sj}\left( x,y\right) \beta ^{s\left(
2p-j\right) }\left( x,y\right) \right) \right) \\
&=&\left( -1\right) ^{s+1}\left( z-\left( -y\right) ^{sp}\right)
\prod_{j=0}^{p-1}\left( z-\alpha ^{sj}\left( x,y\right) \beta ^{s\left(
2p-j\right) }\left( x,y\right) \right) \left( z-\alpha ^{s\left( 2p-j\right)
}\left( x,y\right) \beta ^{sj}\left( x,y\right) \right) \\
&=&\left( -1\right) ^{s+1}\left( z-\left( -y\right) ^{sp}\right)
\prod_{j=0}^{p-1}\left( z^{2}-\left( -y\right) ^{sj}L_{2s\left( p-j\right)
}\left( x,y\right) z+y^{2ps}\right) ,
\end{eqnarray*}%
as claimed. On the other hand, if $k$ is odd, $k=2p-1$ say, then%
\begin{equation}
D_{s,2p}\left( x,y;z\right) =\left( -1\right) ^{s+1}\prod_{j=0}^{p-1}\left(
z^{2}-\left( -y\right) ^{sj}L_{s\left( 2p-1-2j\right) }\left( x,y\right)
z+\left( -y\right) ^{\left( 2p-1\right) s}\right) .  \notag
\end{equation}

In fact, we have%
\begin{eqnarray*}
D_{s,2p}\left( x,y;z\right) &=&\left( -1\right)
^{s+1}\prod_{j=0}^{2p-1}\left( z-\alpha ^{sj}\left( x,y\right) \beta
^{s\left( 2p-1-j\right) }\left( x,y\right) \right) \\
&=&\left( -1\right) ^{s+1}\prod_{j=0}^{p-1}\left( z-\alpha ^{sj}\left(
x,y\right) \beta ^{s\left( 2p-1-j\right) }\left( x,y\right) \right)
\prod_{j=p}^{2p-1}\left( z-\alpha ^{sj}\left( x,y\right) \beta ^{s\left(
2p-1-j\right) }\left( x,y\right) \right) \\
&=&\left( -1\right) ^{s+1}\prod_{j=0}^{p-1}\left( z-\alpha ^{sj}\left(
x,y\right) \beta ^{s\left( 2p-1-j\right) }\left( x,y\right) \right) \left(
z-\alpha ^{s\left( 2p-1-j\right) }\left( x,y\right) \beta ^{sj}\left(
x,y\right) \right) \\
&=&\left( -1\right) ^{s+1}\prod_{j=0}^{p-1}\left( z^{2}-\left( -y\right)
^{sj}L_{s\left( 2p-1-2j\right) }\left( x,y\right) z+\left( -y\right)
^{\left( 2p-1\right) s}\right) ,
\end{eqnarray*}%
as claimed.

Summarizing, we have that%
\begin{eqnarray}
D_{s,2p+1}\left( x,y;z\right) &=&\sum_{i=0}^{2p+1}\left( -1\right) ^{\frac{%
\left( si+2(s+1)\right) \left( i+1\right) }{2}}\dbinom{2p+1}{i}%
_{F\!_{s}\!\left( x,y\right) }y^{\frac{si\left( i-1\right) }{2}}z^{2p+1-i}
\label{2.17} \\
&=&\left( -1\right) ^{s+1}\left( z-\left( -y\right) ^{sp}\right)
\prod_{j=0}^{p-1}\left( z^{2}-\left( -y\right) ^{sj}L_{2s\left( p-j\right)
}\left( x,y\right) z+y^{2ps}\right) ,  \notag
\end{eqnarray}%
and%
\begin{eqnarray}
D_{s,2p+1}\left( x,y;z\right) &=&\sum_{i=0}^{2p+1}\left( -1\right) ^{\frac{%
\left( si+2(s+1)\right) \left( i+1\right) }{2}}\dbinom{2p+1}{i}%
_{F\!_{s}\!\left( x,y\right) }y^{\frac{si\left( i-1\right) }{2}}z^{2p+1-i}
\label{2.18} \\
&=&\left( -1\right) ^{s+1}\left( z-\left( -y\right) ^{sp}\right)
\prod_{j=0}^{p-1}\left( z^{2}-\left( -y\right) ^{sj}L_{2s\left( p-j\right)
}\left( x,y\right) z+y^{2ps}\right) .  \notag
\end{eqnarray}

We can obtain some additional facts by setting $z=y^{sp}$ in (\ref{2.17}).
We have

\begin{itemize}
\item If $s$ or $p$ is even, we see at once that%
\begin{equation*}
\sum_{i=0}^{2p+1}\left( -1\right) ^{\frac{\left( si+2(s+1)\right) \left(
i+1\right) }{2}}\dbinom{2p+1}{i}_{F\!_{s}\!\left( x,y\right) }y^{\frac{%
si\left( i-1\right) }{2}-spi}=0.
\end{equation*}

\item If $s$ and $p$ are odd, then%
\begin{eqnarray*}
&&\sum_{i=0}^{2p+1}\left( -1\right) ^{\frac{\left( si+2(s+1)\right) \left(
i+1\right) }{2}}\dbinom{2p+1}{i}_{F\!_{s}\!\left( x,y\right) }y^{\frac{%
si\left( i-1\right) }{2}}\left( y^{sp}\right) ^{2p+1-i} \\
&=&\left( -1\right) ^{s+1}\left( y^{sp}-\left( -y\right) ^{sp}\right)
\prod_{j=0}^{p-1}\left( 2y^{2sp}-\left( -1\right) ^{sp}\left( -y\right)
^{sj+sp}L_{2s\left( p-j\right) }\left( x,y\right) \right) \\
&=&\left( -1\right) ^{s+1}2y^{sp}\prod_{j=0}^{p-1}\left( -y\right) ^{s\left(
p+j\right) }\left( 2\left( -y\right) ^{s\left( p-j\right) }+L_{2s\left(
p-j\right) }\left( x,y\right) \right) \\
&=&2y^{sp}\prod_{j=0}^{p-1}\left( -y\right) ^{s\left( p+j\right) }L_{s\left(
p-j\right) }^{2}\left( x,y\right) .
\end{eqnarray*}

That is, if $s$ and $p$ are odd we have that%
\begin{equation*}
\sum_{i=0}^{2p+1}\left( -1\right) ^{\frac{\left( si+2(s+1)\right) \left(
i+1\right) }{2}}\dbinom{2p+1}{i}_{F\!_{s}\!\left( x,y\right) }y^{\frac{%
si\left( i-1\right) }{2}+sp\left( 2p-i\right) }=2\prod_{j=0}^{p-1}\left(
-y\right) ^{s\left( p+j\right) }L_{s\left( p-j\right) }^{2}\left( x,y\right)
.
\end{equation*}
\end{itemize}

\begin{proposition}
\label{Prop2.2}\textit{Let }$t,k\in \mathbb{N}^{\prime }$\textit{, }$m\in 
\mathbb{Z}$\textit{\ be given. Then}

\textit{(a)}%
\begin{eqnarray}
&&\frac{\alpha ^{sk}\left( x,y\right) }{\sum_{i=0}^{t+1}\left( -1\right) ^{%
\frac{\left( si+2(s+1)\right) \left( i+1\right) }{2}}\dbinom{t+1}{i}%
_{F\!_{s}\!\left( x,y\right) }\alpha ^{si}\left( x,y\right) y^{\frac{%
si\left( i-1\right) }{2}}z^{t+1-i}}  \notag \\
&&+\frac{\beta ^{sk}\left( x,y\right) }{\sum_{i=0}^{t+1}\left( -1\right) ^{%
\frac{\left( si+2(s+1)\right) \left( i+1\right) }{2}}\dbinom{t+1}{i}%
_{F\!_{s}\!\left( x,y\right) }\beta ^{si}\left( x,y\right) y^{\frac{si\left(
i-1\right) }{2}}z^{t+1-i}}  \notag \\
&=&\frac{L_{sk}\left( x,y\right) z-\left( -y\right) ^{sk}L_{s\left(
t-k+1\right) }\left( x,y\right) }{\sum_{i=0}^{t+2}\left( -1\right) ^{\frac{%
\left( si+2(s+1)\right) \left( i+1\right) }{2}}\dbinom{t+2}{i}%
_{F\!_{s}\!\left( x,y\right) }y^{\frac{i\left( i-1\right) }{2}}z^{t+2-i}}.
\label{2.24}
\end{eqnarray}

\textit{(b)}%
\begin{eqnarray}
&&\frac{\alpha ^{m+sk}\left( x,y\right) }{\sum_{i=0}^{t+1}\left( -1\right) ^{%
\frac{\left( si+2(s+1)\right) \left( i+1\right) }{2}}\dbinom{t+1}{i}%
_{F\!_{s}\!\left( x,y\right) }\alpha ^{si}\left( x,y\right) y^{\frac{%
si\left( i-1\right) }{2}}z^{t+1-i}}  \notag \\
&&-\frac{\beta ^{m+sk}\left( x,y\right) }{\sum_{i=0}^{t+1}\left( -1\right) ^{%
\frac{\left( si+2(s+1)\right) \left( i+1\right) }{2}}\dbinom{t+1}{i}%
_{F\!_{s}\!\left( x,y\right) }\beta ^{si}\left( x,y\right) y^{\frac{si\left(
i-1\right) }{2}}z^{t+1-i}}  \notag \\
&=&\frac{\sqrt{x^{2}+4y}\left( F_{sk+m}\left( x,y\right) z+\left( -y\right)
^{sk+m}F_{s\left( t-k+1\right) -m}\left( x,y\right) \right) }{%
\sum_{i=0}^{t+2}\left( -1\right) ^{\frac{\left( si+2(s+1)\right) \left(
i+1\right) }{2}}\dbinom{t+2}{i}_{F\!_{s}\!\left( x,y\right) }y^{\frac{%
si\left( i-1\right) }{2}}z^{t+2-i}}.  \label{2.25}
\end{eqnarray}
\end{proposition}

\begin{proof}
(a) We begin by writing the left-hand side of (\ref{2.24}) (we write LHS$_{%
\ref{2.24}}$) as%
\begin{eqnarray*}
\text{LHS}_{\ref{2.24}} &=&\frac{\alpha ^{sk}\left( x,y\right) }{\alpha
^{s\left( t+1\right) }\left( x,y\right) \sum_{i=0}^{t+1}\left( -1\right) ^{%
\frac{\left( si+2(s+1)\right) \left( i+1\right) }{2}}\dbinom{t+1}{i}%
_{F\!_{s}\!\left( x,y\right) }y^{\frac{si\left( i-1\right) }{2}}\left( \frac{%
z}{\alpha ^{s}\left( x,y\right) }\right) ^{t+1-i}} \\
&&+\frac{\beta ^{sk}\left( x,y\right) }{\beta ^{s\left( t+1\right) }\left(
x,y\right) \sum_{i=0}^{t+1}\left( -1\right) ^{\frac{\left( si+2(s+1)\right)
\left( i+1\right) }{2}}\dbinom{t+1}{i}_{F\!_{s}\!\left( x,y\right) }y^{\frac{%
si\left( i-1\right) }{2}}\left( \frac{z}{\beta ^{s}\left( x,y\right) }%
\right) ^{t+1-i}},
\end{eqnarray*}%
or, by using (\ref{2.16})%
\begin{eqnarray*}
&&\text{LHS}_{\ref{2.24}} \\
&=&\frac{\left( -1\right) ^{s+1}\alpha ^{sk}\left( x,y\right) }{\alpha
^{s\left( t+1\right) }\left( x,y\right) \dprod\limits_{j=0}^{t}\left( \frac{z%
}{\alpha ^{s}\left( x,y\right) }-\alpha ^{sj}\left( x,y\right) \beta
^{s\left( t-j\right) }\left( x,y\right) \right) } \\
&&+\frac{\left( -1\right) ^{s+1}\beta ^{sk}\left( x,y\right) }{\beta
^{s\left( t+1\right) }\left( x,y\right) \dprod\limits_{j=0}^{t}\left( \frac{z%
}{\beta ^{s}\left( x,y\right) }-\alpha ^{sj}\left( x,y\right) \beta
^{s\left( t-j\right) }\left( x,y\right) \right) } \\
&=&\!\!\frac{\left( -1\right) ^{s+1}\alpha ^{sk}\left( x,y\right) }{%
\dprod\limits_{j=0}^{t}\left( z-\alpha ^{s\left( j+1\right) }\left(
x,y\right) \beta ^{s\left( t-j\right) }\left( x,y\right) \right) }+\frac{%
\left( -1\right) ^{s+1}\beta ^{sk}\left( x,y\right) }{\dprod%
\limits_{j=0}^{t}\left( z-\alpha ^{sj}\left( x,y\right) \beta ^{s\left(
t+1-j\right) }\left( x,y\right) \right) } \\
&=&\frac{\left( -1\right) ^{s+1}\alpha ^{sk}\left( x,y\right) }{%
\dprod\limits_{j=1}^{t+1}\left( z-\alpha ^{sj}\left( x,y\right) \beta
^{s\left( t+1-j\right) }\left( x,y\right) \right) }+\frac{\left( -1\right)
^{s+1}\beta ^{sk}\left( x,y\right) }{\dprod\limits_{j=0}^{t}\left( z-\alpha
^{sj}\left( x,y\right) \beta ^{s\left( t+1-j\right) }\left( x,y\right)
\right) }.
\end{eqnarray*}%
Some further algebraic manipulation gives us%
\begin{eqnarray*}
&&\text{LHS}_{\ref{2.24}} \\
&=&\frac{\left( -1\right) ^{s+1}}{\dprod\limits_{j=1}^{t}\left( z-\alpha
^{sj}\left( x,y\right) \beta ^{s\left( t+1-j\right) }\left( x,y\right)
\right) }\left( \frac{\alpha ^{sk}\left( x,y\right) }{z-\alpha ^{s\left(
t+1\right) }\left( x,y\right) }+\frac{\beta ^{sk}\left( x,y\right) }{z-\beta
^{s\left( t+1\right) }\left( x,y\right) }\right) \\
&=&\frac{\alpha ^{sk}\left( x,y\right) \left( z-\beta ^{s\left( t+1\right)
}\left( x,y\right) \right) +\beta ^{sk}\left( x,y\right) \left( z-\alpha
^{s\left( t+1\right) }\left( x,y\right) \right) }{\left( -1\right)
^{s+1}\dprod\limits_{j=0}^{t+1}\left( z-\alpha ^{sj}\left( x,y\right) \beta
^{s\left( t+1-j\right) }\left( x,y\right) \right) } \\
&=&\frac{L_{sk}\left( x,y\right) z-\left( -y\right) ^{sk}L_{s\left(
t-k+1\right) }\left( x,y\right) }{\sum_{i=0}^{t+2}\left( -1\right) ^{\frac{%
\left( si+2(s+1)\right) \left( i+1\right) }{2}}\dbinom{t+2}{i}%
_{F\!_{s}\!\left( x,y\right) }y^{\frac{i\left( i-1\right) }{2}}z^{t+2-i}},
\end{eqnarray*}%
as wanted.

(b) We write the left-hand side of (\ref{2.25}) (LHS$_{\ref{2.25}}$) as%
\begin{eqnarray*}
\text{LHS}_{\ref{2.25}} &=&\frac{\alpha ^{m+sk}\left( x,y\right) }{\alpha
^{s\left( t+1\right) }\left( x,y\right) \sum_{i=0}^{t+1}\left( -1\right) ^{%
\frac{\left( si+2(s+1)\right) \left( i+1\right) }{2}}\dbinom{t+1}{i}%
_{F\!_{s}\!\left( x,y\right) }y^{\frac{i\left( i-1\right) }{2}}\left( \frac{z%
}{\alpha ^{s}\left( x,y\right) }\right) ^{t+1-i}} \\
&&-\frac{\beta ^{m+sk}\left( x,y\right) }{\beta ^{s\left( t+1\right) }\left(
x,y\right) \sum_{i=0}^{t+1}\left( -1\right) ^{\frac{\left( si+2(s+1)\right)
\left( i+1\right) }{2}}\dbinom{t+1}{i}_{F\!_{s}\!\left( x,y\right) }y^{\frac{%
i\left( i-1\right) }{2}}\left( \frac{z}{\beta ^{s}\left( x,y\right) }\right)
^{t+1-i}},
\end{eqnarray*}%
and use (\ref{2.16}) to write%
\begin{eqnarray*}
&&\text{LHS}_{\ref{2.25}} \\
&=&\frac{\alpha ^{m+sk}\left( x,y\right) }{\alpha ^{s\left( t+1\right)
}\left( x,y\right) \left( -1\right) ^{s+1}\dprod\limits_{j=0}^{t}\left( 
\frac{z}{\alpha ^{s}\left( x,y\right) }-\alpha ^{sj}\left( x,y\right) \beta
^{s\left( t-j\right) }\left( x,y\right) \right) } \\
&&-\frac{\beta ^{m+sk}\left( x,y\right) }{\beta ^{s\left( t+1\right) }\left(
x,y\right) \left( -1\right) ^{s+1}\dprod\limits_{j=0}^{t}\left( \frac{z}{%
\beta ^{s}\left( x,y\right) }-\alpha ^{sj}\left( x,y\right) \beta ^{s\left(
t-j\right) }\left( x,y\right) \right) } \\
&=&\frac{\alpha ^{m+sk}\left( x,y\right) }{\left( -1\right)
^{s+1}\dprod\limits_{j=0}^{t}\left( z-\alpha ^{s\left( j+1\right) }\left(
x,y\right) \beta ^{s\left( t-j\right) }\left( x,y\right) \right) }-\frac{%
\beta ^{m+sk}\left( x,y\right) }{\left( -1\right)
^{s+1}\dprod\limits_{j=0}^{t}\left( z-\alpha ^{sj}\left( x,y\right) \beta
^{s\left( t+1-j\right) }\left( x,y\right) \right) }.
\end{eqnarray*}%
Some further simplifications give us%
\begin{eqnarray*}
&&\text{LHS}_{\ref{2.25}} \\
&=&\frac{1}{\left( -1\right) ^{s+1}\dprod\limits_{j=1}^{t}\left( z-\alpha
^{sj}\left( x,y\right) \beta ^{s\left( t+1-j\right) }\left( x,y\right)
\right) }\left( \frac{\alpha ^{m+sk}\left( x,y\right) }{z-\alpha ^{s\left(
t+1\right) }\left( x,y\right) }-\frac{\beta ^{m+sk}\left( x,y\right) }{%
z-\beta ^{s\left( t+1\right) }\left( x,y\right) }\right) \\
&=&\frac{\alpha ^{m+sk}\left( x,y\right) \left( z-\beta ^{s\left( t+1\right)
}\left( x,y\right) \right) -\beta ^{m+sk}\left( x,y\right) \left( z-\alpha
^{s\left( t+1\right) }\left( x,y\right) \right) }{\left( -1\right)
^{s+1}\dprod\limits_{j=0}^{t+1}\left( z-\alpha ^{sj}\left( x,y\right) \beta
^{s\left( t+1-j\right) }\left( x,y\right) \right) } \\
&=&\frac{\sqrt{x^{2}+4y}F_{sk+m}\left( x,y\right) z+\beta ^{m+sk}\left(
x,y\right) \alpha ^{m+sk}\left( x,y\right) \left( \alpha ^{s\left(
t-k+1\right) -m}\left( x,y\right) -\beta ^{s\left( t-k+1\right) -m}\left(
x,y\right) \right) }{\sum_{i=0}^{t+2}\left( -1\right) ^{\frac{\left(
si+2(s+1)\right) \left( i+1\right) }{2}}\dbinom{t+2}{i}_{F\!_{s}\!\left(
x,y\right) }y^{\frac{i\left( i-1\right) }{2}}z^{t+2-i}} \\
&=&\frac{\sqrt{x^{2}+4y}\left( F_{sk+m}\left( x,y\right) z+\left( -y\right)
^{sk+m}F_{s\left( t-k+1\right) -m}\left( x,y\right) \right) }{%
\sum_{i=0}^{t+2}\left( -1\right) ^{\frac{\left( si+2(s+1)\right) \left(
i+1\right) }{2}}\dbinom{t+2}{i}_{F\!_{s}\!\left( x,y\right) }y^{\frac{%
i\left( i-1\right) }{2}}z^{t+2-i}},
\end{eqnarray*}%
as wanted.
\end{proof}

\begin{lemma}
\label{Lemma2.3}Let $t,i\in \mathbb{N}^{\prime }$ be given. The following
identity holds 
\begin{eqnarray}
F_{s\left( t+2\right) }\left( x,y\right) F_{s\left( t+1\right) }\left(
x,y\right) &=&\left( -y\right) ^{si}F_{s\left( t+2-i\right) }\left(
x,y\right) F_{s\left( t+1-i\right) }\left( x,y\right)  \label{2.26} \\
&&+L_{s\left( t+1\right) }\left( x,y\right) F_{s\left( t+2-i\right) }\left(
x,y\right) F_{si}\left( x,y\right)  \notag \\
&&+\left( -y\right) ^{s\left( t-i+2\right) }F_{si}\left( x,y\right)
F_{s\left( i-1\right) }\left( x,y\right) .  \notag
\end{eqnarray}
\end{lemma}

\begin{proof}
Use Binet's formulas to prove that 
\begin{equation*}
\left( -y\right) ^{si}F_{s\left( t+1-i\right) }\left( x,y\right) +L_{s\left(
t+1\right) }\left( x,y\right) F_{si}\left( x,y\right) =F_{s\left(
t+1+i\right) }\left( x,y\right) ,
\end{equation*}%
then write the right-hand side of (\ref{2.26}) as 
\begin{equation*}
F_{s\left( t+2-i\right) }\left( x,y\right) F_{s\left( t+1+i\right) }\left(
x,y\right) +\left( -y\right) ^{s\left( t+2-i\right) }F_{si}\left( x,y\right)
F_{s\left( i-1\right) }\left( x,y\right) .
\end{equation*}

Now use (1.11) to obtain (\ref{2.26}).
\end{proof}

\begin{proposition}
\label{Prop2.4}\textit{Let }$t\in \mathbb{N}^{\prime }$\textit{\ be given.
Then}%
\begin{eqnarray}
&&\sum_{i=0}^{t+2}\left( -1\right) ^{\frac{\left( si+2(s+1)\right) \left(
i+1\right) }{2}}\dbinom{t+2}{i}_{F\!_{s}\!\left( x,y\right) }y^{\frac{%
si\left( i-1\right) }{2}}z^{t+2-i}  \label{2.27} \\
&=&\left( z^{2}-L_{s\left( t+1\right) }\left( x,y\right) z+\left( -y\right)
^{s\left( t+1\right) }\right) \sum_{i=0}^{t}\left( -1\right) ^{\frac{\left(
si+2(s+1)\right) \left( i+1\right) }{2}}\dbinom{t}{i}_{F\!_{s}\!\left(
x,y\right) }\left( -1\right) ^{si}y^{\frac{si\left( i+1\right) }{2}}z^{t-i}.
\notag
\end{eqnarray}
\end{proposition}

\begin{proof}
We have%
\begin{eqnarray*}
&&\left( z^{2}-L_{s\left( t+1\right) }\left( x,y\right) z+\left( -y\right)
^{s\left( t+1\right) }\right) \sum_{i=0}^{t}\left( -1\right) ^{\frac{\left(
si+2(s+1)\right) \left( i+1\right) }{2}}\dbinom{t}{i}_{F\!_{s}\!\left(
x,y\right) }\left( -1\right) ^{si}y^{\frac{si\left( i+1\right) }{2}}z^{t-i}
\\
&=&\sum_{i=0}^{t}\left( -1\right) ^{\frac{\left( si+2(s+1)\right) \left(
i+1\right) }{2}}\dbinom{t}{i}_{F\!_{s}\!\left( x,y\right) }\left( -1\right)
^{si}y^{\frac{si\left( i+1\right) }{2}}z^{t+2-i} \\
&&-L_{s\left( t+1\right) }\left( x,y\right) \sum_{i=1}^{t+1}\left( -1\right)
^{\frac{\left( si+s+2\right) i}{2}}\dbinom{t}{i-1}_{F\!_{s}\!\left(
x,y\right) }\left( -1\right) ^{s\left( i-1\right) }y^{\frac{si\left(
i-1\right) }{2}}z^{t+2-i} \\
&&+\left( -y\right) ^{s\left( t+1\right) }\sum_{i=2}^{t+2}\left( -1\right) ^{%
\frac{\left( si+2)\right) \left( i-1\right) }{2}}\dbinom{t}{i-2}%
_{F\!_{s}\!\left( x,y\right) }\left( -1\right) ^{si}y^{\frac{s\left(
i-2\right) \left( i-1\right) }{2}}z^{t+2-i} \\
&=&\sum_{i=0}^{t+2}\left( -1\right) ^{\frac{\left( si+2(s+1)\right) \left(
i+1\right) }{2}}\dbinom{t+2}{i}_{F\!_{s}\!\left( x,y\right) }\frac{1}{%
F_{s\left( t+2\right) }\left( x,y\right) F_{s\left( t+1\right) }\left(
x,y\right) } \\
&&\times \left( 
\begin{array}{c}
\left( -y\right) ^{si}F_{s\left( t+2-i\right) }\left( x,y\right) F_{s\left(
t+1-i\right) }\left( x,y\right) \\ 
+L_{s\left( t+1\right) }\left( x,y\right) F_{s\left( t+2-i\right) }\left(
x,y\right) F_{si}\left( x,y\right) \\ 
+\left( -y\right) ^{s\left( t-i+2\right) }F_{si}\left( x,y\right) F_{s\left(
i-1\right) }\left( x,y\right)%
\end{array}%
\right) y^{\frac{si\left( i-1\right) }{2}}z^{t+2-i}.
\end{eqnarray*}

Finally use lemma \ref{Lemma2.3} to obtain (\ref{2.27}).
\end{proof}

\begin{proposition}
\label{Prop2.5}Let $i,t\in \mathbb{N}^{\prime }$ and $m\in \mathbb{Z}$ be
given. The following identities hold

(a)%
\begin{eqnarray}
&&\sum_{j=0}^{i}\left( -1\right) ^{\frac{\left( sj+2\left( s+1\right)
\right) \left( j+1\right) }{2}}\dbinom{t+1}{j}_{F\!_{s}\!\left( x,y\right)
}F_{ts\left( i-j\right) +m}\left( x,y\right) y^{\frac{sj\left( j-1\right) }{2%
}}  \label{2.28} \\
&=&\left( -1\right) ^{s+i+1+\frac{si\left( i+1\right) }{2}}\dbinom{t}{i}%
_{F\!_{s}\!\left( x,y\right) }F_{m-is}\left( x,y\right) y^{\frac{si\left(
i+1\right) }{2}}.  \notag
\end{eqnarray}
\ 

(b)%
\begin{eqnarray}
&&\sum_{j=0}^{i}\left( -1\right) ^{\frac{\left( sj+2\left( s+1\right)
\right) \left( j+1\right) }{2}}\dbinom{t+1}{j}_{F\!_{s}\!\left( x,y\right)
}L_{ts\left( i-j\right) +m}\left( x,y\right) y^{\frac{sj\left( j-1\right) }{2%
}}  \label{2.29} \\
&=&\left( -1\right) ^{s+i+1+\frac{si\left( i+1\right) }{2}}\dbinom{t}{i}%
_{F\!_{s}\!\left( x,y\right) }L_{m-is}\left( x,y\right) y^{\frac{si\left(
i+1\right) }{2}}.  \notag
\end{eqnarray}
\end{proposition}

\begin{proof}
(a) We proceed by induction on $i$. For $i=0$ both sides are equal to $%
\left( -1\right) ^{s+1}F_{m}\left( x,y\right) $. Let us suppose the result
is true for a given $i\in \mathbb{N}$. Then 
\begin{eqnarray*}
&&\sum_{j=0}^{i+1}\left( -1\right) ^{\frac{\left( sj+2\left( s+1\right)
\right) \left( j+1\right) }{2}}\dbinom{t+1}{j}_{F\!_{s}\!\left( x,y\right)
}F_{ts\left( i+1-j\right) +m}\left( x,y\right) y^{\frac{sj\left( j-1\right) 
}{2}} \\
&=&\sum_{j=0}^{i}\left( -1\right) ^{\frac{\left( sj+2\left( s+1\right)
\right) \left( j+1\right) }{2}}\dbinom{t+1}{j}_{F\!_{s}\!\left( x,y\right)
}F_{ts\left( i-j\right) +m+ts}\left( x,y\right) y^{\frac{sj\left( j-1\right) 
}{2}} \\
&&+\left( -1\right) ^{\frac{\left( s\left( i+1\right) +2\left( s+1\right)
\right) \left( i+2\right) }{2}}\dbinom{t+1}{i+1}_{F\!_{s}\!\left( x,y\right)
}F_{m}\left( x,y\right) y^{\frac{si\left( i+1\right) }{2}} \\
&=&\left( -1\right) ^{i+s+1+\frac{si\left( i+1\right) }{2}}\dbinom{t}{i}%
_{F\!_{s}\!\left( x,y\right) }F_{m+ts-is}\left( x,y\right) y^{\frac{si\left(
i+1\right) }{2}} \\
&&+\left( -1\right) ^{\frac{\left( s\left( i+1\right) +2\left( s+1\right)
\right) \left( i+2\right) }{2}}\dbinom{t+1}{i+1}_{F\!_{s}\!\left( x,y\right)
}F_{m}\left( x,y\right) y^{\frac{si\left( i+1\right) }{2}} \\
&=&\left( -1\right) ^{i+s+2+\frac{s\left( i+1\right) \left( i+2\right) }{2}}%
\dbinom{t}{i+1}_{F\!_{s}\!\left( x,y\right) }y^{\frac{s\left( i+1\right)
\left( i+2\right) }{2}}\frac{\left( -y\right) ^{-s\left( i+1\right) }}{%
F_{s\left( t-i\right) }\left( x,y\right) }\left( 
\begin{array}{c}
-F_{m+ts-is}\left( x,y\right) F_{s\left( i+1\right) }\left( x,y\right) \\ 
+F_{m}\left( x,y\right) F_{s\left( t+1\right) }\left( x,y\right)%
\end{array}%
\right) \\
&=&\left( -1\right) ^{i+s+2+\frac{s\left( i+1\right) \left( i+2\right) }{2}}%
\dbinom{t}{i+1}_{F\!_{s}\!\left( x,y\right) }F_{m-\left( i+1\right) s}\left(
x,y\right) y^{\frac{s\left( i+1\right) \left( i+2\right) }{2}},
\end{eqnarray*}%
as wanted. In the last step we used (\ref{1.11}) in the form%
\begin{equation*}
F_{m}\left( x,y\right) F_{s\left( t+1\right) }\left( x,y\right)
-F_{m+ts-is}\left( x,y\right) F_{s\left( i+1\right) }\left( x,y\right)
=\left( -y\right) ^{s\left( i+1\right) }F_{s\left( t-i\right) }\left(
x,y\right) F_{m-\left( i+1\right) s}\left( x,y\right) .
\end{equation*}

(b) The proof of (\ref{2.29}) is similar, using (at the end of the
procedure) the identity%
\begin{equation*}
L_{m}\left( x,y\right) F_{s\left( t+1\right) }\left( x,y\right)
-L_{m+ts-is}\left( x,y\right) F_{s\left( i+1\right) }\left( x,y\right)
=\left( -y\right) ^{s\left( i+1\right) }F_{s\left( t-i\right) }\left(
x,y\right) L_{m-\left( i+1\right) s}\left( x,y\right) ,
\end{equation*}%
which is essentially (\ref{1.12}). We leave the details to the reader.
\end{proof}

\section{\label{Sec3}The main results}

We can write the sequence $F_{sn+m_{1}}^{k_{1}}\left( x,y\right)
F_{sn+m_{2}}^{k_{2}}\left( x,y\right) $ (where $m_{1},m_{2}\in \mathbb{Z}$
and $k_{1},k_{2}\in \mathbb{N}^{\prime }$ are given) as%
\begin{eqnarray*}
&&F_{sn+m_{1}}^{k_{1}}\left( x,y\right) F_{sn+m_{2}}^{k_{2}}\left( x,y\right)
\\
&=&\left( \frac{\alpha ^{sn+m_{1}}\left( x,y\right) -\beta ^{sn+m_{1}}\left(
x,y\right) }{\sqrt{x^{2}+4y}}\right) ^{k_{1}}\left( \frac{\alpha
^{sn+m_{2}}\left( x,y\right) -\beta ^{sn+m_{2}}\left( x,y\right) }{\sqrt{%
x^{2}+4y}}\right) ^{k_{2}} \\
&=&\left( x^{2}+4y\right) ^{-\frac{k_{1}+k_{2}}{2}}\sum_{i=0}^{k_{1}}\dbinom{%
k_{1}}{i}\left( \alpha ^{sn+m_{1}}\left( x,y\right) \right) ^{i}\left(
-\beta ^{sn+m_{1}}\left( x,y\right) \right) ^{k_{1}-i} \\
&&\times \sum_{j=0}^{k_{2}}\dbinom{k_{2}}{j}\left( \alpha ^{sn+m_{2}}\left(
x,y\right) \right) ^{j}\left( -\beta ^{sn+m_{2}}\left( x,y\right) \right)
^{k_{2}-j} \\
&=&\left( x^{2}+4y\right) ^{-\frac{k_{1}+k_{2}}{2}}\beta
^{m_{1}k_{1}+m_{2}k_{2}}\left( x,y\right)
\sum_{j=0}^{k_{1}+k_{2}}\sum_{i=0}^{k_{1}}\left( -1\right) ^{k_{1}+k_{2}-j}
\\
&&\times \dbinom{k_{1}}{i}\dbinom{k_{2}}{j-i}\left( \frac{\alpha \left(
x,y\right) }{\beta \left( x,y\right) }\right) ^{\left( m_{1}-m_{2}\right)
i+m_{2}j}\left( \alpha ^{sj}\left( x,y\right) \beta ^{s\left(
k_{1}+k_{2}-j\right) }\left( x,y\right) \right) ^{n}.
\end{eqnarray*}

Then the $Z$ transform of $F_{sn+m_{1}}^{k_{1}}\left( x,y\right)
F_{sn+m_{2}}^{k_{2}}\left( x,y\right) $ is%
\begin{eqnarray}
&&\mathcal{Z}\left( F_{sn+m_{1}}^{k_{1}}\left( x,y\right)
F_{sn+m_{2}}^{k_{2}}\left( x,y\right) \right)  \label{3.1} \\
&=&\left( x^{2}+4y\right) ^{-\frac{k_{1}+k_{2}}{2}}\beta
^{m_{1}k_{1}+m_{2}k_{2}}\left( x,y\right)
\sum_{j=0}^{k_{1}+k_{2}}\sum_{i=0}^{k_{1}}\left( -1\right) ^{k_{1}+k_{2}-j}%
\dbinom{k_{1}}{i}\dbinom{k_{2}}{j-i}\left( \frac{\alpha \left( x,y\right) }{%
\beta \left( x,y\right) }\right) ^{\left( m_{1}-m_{2}\right) i+m_{2}j} 
\notag \\
&&\times \frac{z}{z-\alpha ^{sj}\left( x,y\right) \beta ^{s\left(
k_{1}+k_{2}-j\right) }\left( x,y\right) }.  \notag
\end{eqnarray}

The following theorem tells us that the right-hand side of (\ref{3.1}) can
be written in a special form.

\begin{theorem}
\label{Th3.1}\textit{Let }$m_{1},m_{2}\in \mathbb{Z}$\textit{\ and }$%
k_{1},k_{2}\in \mathbb{N}^{\prime }$\textit{\ be given. The sequence }$%
F_{sn+m_{1}}^{k_{1}}\left( x,y\right) F_{sn+m_{2}}^{k_{2}}\left( x,y\right) $%
\textit{\ has }$Z$\textit{\ transform given by}%
\begin{eqnarray}
&&\mathcal{Z}\left( F_{sn+m_{1}}^{k_{1}}\left( x,y\right)
F_{sn+m_{2}}^{k_{2}}\left( x,y\right) \right)  \label{3.2} \\
&=&z\frac{\sum\limits_{i=0}^{k_{1}+k_{2}}\sum\limits_{j=0}^{i}\left(
-1\right) ^{\frac{\left( sj+2(s+1)\right) \left( j+1\right) }{2}}\dbinom{%
k_{1}+k_{2}+1}{j}_{F\!_{s}\!\left( x,y\right) }F_{m_{1}+s\left( i-j\right)
}^{k_{1}}\left( x,y\right) F_{m_{2}+s\left( i-j\right) }^{k_{2}}\left(
x,y\right) y^{\frac{sj\left( j-1\right) }{2}}z^{k_{1}+k_{2}-i}}{%
\sum\limits_{i=0}^{k_{1}+k_{2}+1}\left( -1\right) ^{\frac{\left(
si+2(s+1)\right) \left( i+1\right) }{2}}\dbinom{k_{1}+k_{2}+1}{i}%
_{F\!_{s}\!\left( x,y\right) }y^{\frac{si\left( i-1\right) }{2}%
}z^{k_{1}+k_{2}+1-i}}.  \notag
\end{eqnarray}
\end{theorem}

\begin{proof}
We have to show that%
\begin{eqnarray}
&&\left( x^{2}+4y\right) ^{-\frac{k_{1}+k_{2}}{2}}\beta
^{m_{1}k_{1}+m_{2}k_{2}}\left( x,y\right)  \label{3.3} \\
&&\times \sum_{j=0}^{k_{1}+k_{2}}\sum_{i=0}^{k_{1}}\left( -1\right)
^{k_{1}+k_{2}-j}\dbinom{k_{1}}{i}\dbinom{k_{2}}{j-i}\left( \frac{\alpha
\left( x,y\right) }{\beta \left( x,y\right) }\right) ^{\left(
m_{1}-m_{2}\right) i+m_{2}j}\frac{z}{z-\alpha ^{sj}\left( x,y\right) \beta
^{s\left( k_{1}+k_{2}-j\right) }\left( x,y\right) }  \notag \\
&=&z\frac{\sum\limits_{i=0}^{k_{1}+k_{2}}\sum\limits_{j=0}^{i}\left(
-1\right) ^{\frac{\left( sj+2(s+1)\right) \left( j+1\right) }{2}}\dbinom{%
k_{1}+k_{2}+1}{j}_{F\!_{s}\!\left( x,y\right) }F_{m_{1}+s\left( i-j\right)
}^{k_{1}}\left( x,y\right) F_{m_{2}+s\left( i-j\right) }^{k_{2}}\left(
x,y\right) y^{\frac{sj\left( j-1\right) }{2}}z^{k_{1}+k_{2}-i}}{%
\sum_{i=0}^{k_{1}+k_{2}+1}\left( -1\right) ^{\frac{\left( si+2(s+1)\right)
\left( i+1\right) }{2}}\dbinom{k_{1}+k_{2}+1}{i}_{F\!_{s}\!\left( x,y\right)
}y^{\frac{si\left( i-1\right) }{2}}z^{k_{1}+k_{2}+1-i}}.  \notag
\end{eqnarray}

We will proceed by induction on $k_{1}$ and/or $k_{2}$ (the symmetry of (\ref%
{3.3}) with respect to $k_{1}$ and $k_{2}$ allows us to use induction on any
of these parameters). If $k_{1}=k_{2}=1$ the left hand side of (\ref{3.3})
(LHS$_{\ref{3.3}}$) is%
\begin{eqnarray*}
&&\text{LHS}_{\ref{3.3}} \\
&=&\left( x^{2}+4y\right) ^{-\frac{1+1}{2}}\beta ^{m_{1}+m_{2}}\left(
x,y\right) \sum_{i=0}^{1}\sum_{j=0}^{1}\dbinom{1}{i}\dbinom{1}{j}\left(
-1\right) ^{i+j} \\
&&\times \left( \frac{\alpha \left( x,y\right) }{\beta \left( x,y\right) }%
\right) ^{m_{1}i+m_{2}j}\frac{z}{z-\alpha ^{s\left( i+j\right) }\left(
x,y\right) \beta ^{s\left( 2-i-j\right) }\left( x,y\right) } \\
&=&\left( x^{2}+4y\right) ^{-1}z\left( \frac{\beta ^{m_{1}+m_{2}}\left(
x,y\right) }{z-\beta ^{2s}\left( x,y\right) }-\frac{\beta ^{m_{2}}\left(
x,y\right) \alpha ^{m_{1}}\left( x,y\right) }{z-\alpha ^{s}\left( x,y\right)
\beta ^{s}\left( x,y\right) }-\frac{\beta ^{m_{1}}\left( x,y\right) \alpha
^{m_{2}}\left( x,y\right) }{z-\alpha ^{s}\left( x,y\right) \beta ^{s}\left(
x,y\right) }+\frac{\alpha ^{m_{1}+m_{2}}\left( x,y\right) }{z-\alpha
^{2s}\left( x,y\right) }\right) .
\end{eqnarray*}

What follows is simply algebraic manipulation of the expression in
parenthesis, mixed with some identities from (\ref{1.6}), (\ref{1.7}) and (%
\ref{1.71}). We show some steps of this procedure. We have%
\begin{eqnarray*}
&&\text{LHS}_{\ref{3.3}} \\
&=&\left( x^{2}+4y\right) ^{-1}z\left( 
\begin{array}{c}
\dfrac{%
\begin{array}{c}
\left( \alpha ^{m_{1}+m_{2}}\left( x,y\right) +\beta ^{m_{1}+m_{2}}\left(
x,y\right) \right) z-\beta ^{m_{1}+m_{2}}\left( x,y\right) \alpha
^{2s}\left( x,y\right) \\ 
-\alpha ^{m_{1}+m_{2}}\left( x,y\right) \beta ^{2s}\left( x,y\right)%
\end{array}%
}{z^{2}-L_{2s}\left( x,y\right) z+y^{2s}} \\ 
\\ 
-\dfrac{%
\begin{array}{c}
\alpha ^{m_{1}}\left( x,y\right) \beta ^{m_{2}}\left( x,y\right) \\ 
+\alpha ^{m_{2}}\left( x,y\right) \beta ^{m_{1}}\left( x,y\right)%
\end{array}%
}{z-\left( -y\right) ^{s}}%
\end{array}%
\right) \\
&& \\
&=&\left( x^{2}+4y\right) ^{-1}\frac{z}{z^{3}-\left( L_{2s}\left( x,y\right)
+\left( -y\right) ^{s}\right) z^{2}+\left( -y\right) ^{s}\left( L_{2s}\left(
x,y\right) +\left( -y\right) ^{s}\right) z-\left( -y\right) ^{3s}} \\
&& \\
&&\times \left( \!\!\!%
\begin{array}{c}
\left( \alpha ^{m_{1}+m_{2}}\left( x,y\right) +\beta ^{m_{1}+m_{2}}\left(
x,y\right) -\alpha ^{m_{1}}\left( x,y\right) \beta ^{m_{2}}\left( x,y\right)
-\alpha ^{m_{2}}\left( x,y\right) \beta ^{m_{1}}\left( x,y\right) \right)
z^{2} \\ 
\\ 
-\left( 
\begin{array}{c}
\beta ^{m_{1}+m_{2}}\left( x,y\right) \alpha ^{2s}\!\left( x,y\right)
+\!\alpha ^{m_{1}+m_{2}}\left( x,y\right) \beta ^{2s}\!\left( x,y\right) \\ 
-\!\left( \alpha ^{m_{1}}\left( x,y\right) \beta ^{m_{2}}\!\left( x,y\right)
+\!\alpha ^{m_{2}}\left( x,y\right) \beta ^{m_{1}}\left( x,y\right) \right)
\left( \alpha ^{2s}\!\left( x,y\right) +\!\beta ^{2s}\left( x,y\right)
\right) \! \\ 
+\!\left( -y\right) ^{s}\left( \alpha ^{m_{1}+m_{2}}\!\left( x,y\right)
+\!\beta ^{m_{1}+m_{2}}\left( x,y\right) \right)%
\end{array}%
\right) z \\ 
\\ 
+\left( -y\right) ^{s}\beta ^{m_{1}+m_{2}}\left( x,y\right) \alpha
^{2s}\left( x,y\right) +\left( -y\right) ^{s}\alpha ^{m_{1}+m_{2}}\left(
x,y\right) \beta ^{2s}\left( x,y\right) \\ 
-y^{2s}\left( \alpha ^{m_{1}}\left( x,y\right) \beta ^{m_{2}}\left(
x,y\right) +\alpha ^{m_{2}}\left( x,y\right) \beta ^{m_{1}}\left( x,y\right)
\right)%
\end{array}%
\!\!\!\right) \\
&=&\frac{z}{\left( -1\right) ^{s+1}z^{3}+\left( -1\right) ^{s}\frac{%
F_{3s}\left( x,y\right) }{F\!_{s}\!\left( x,y\right) }z^{2}-y^{s}\frac{%
F_{3s}\left( x,y\right) }{F\!_{s}\!\left( x,y\right) }z+y^{3s}} \\
&& \\
&&\times \left( \!\!\!%
\begin{array}{c}
\left( -1\right) ^{s+1}F_{m_{1}}\left( x,y\right) F_{m_{2}}\left( x,y\right)
z^{2} \\ 
+ \\ 
\left( -1\right) ^{s+1}\left( F_{m_{1}+s}\left( x,y\right) F_{m_{2}+s}\left(
x,y\right) -\frac{F_{3s}\left( x,y\right) }{F\!_{s}\!\left( x,y\right) }%
F_{m_{1}}\left( x,y\right) F_{m_{2}}\left( x,y\right) \right) z \\ 
+ \\ 
\left( -1\right) ^{s+1}F_{m_{1}+2s}\left( x,y\right) F_{m_{2}+2s}\left(
x,y\right) +\left( -1\right) ^{s}\frac{F_{3s}\left( x,y\right) }{%
F\!_{s}\!\left( x,y\right) }F_{m_{1}+s}\left( x,y\right) F_{m_{2}+s}\left(
x,y\right) \\ 
-\frac{F_{3s}\left( x,y\right) }{F\!_{s}\!\left( x,y\right) }F_{m_{1}}\left(
x,y\right) F_{m_{2}}\left( x,y\right) y^{s}%
\end{array}%
\!\!\!\right) \\
&& \\
&=&z\frac{\sum_{i=0}^{2}\sum_{j=0}^{i}\left( -1\right) ^{\frac{\left(
si+2(s+1)\right) \left( i+1\right) }{2}}\dbinom{3}{j}_{F\!_{s}\!\left(
x,y\right) }F_{m_{1}+s\left( i-j\right) }\left( x,y\right) F_{m_{2}+s\left(
i-j\right) }\left( x,y\right) y^{\frac{sj\left( j-1\right) }{2}}z^{2-i}}{%
\sum_{i=0}^{3}\left( -1\right) ^{\frac{\left( si+2(s+1)\right) \left(
i+1\right) }{2}}\dbinom{3}{i}_{F\!_{s}\!\left( x,y\right) }y^{\frac{si\left(
i-1\right) }{2}}z^{3-i}},
\end{eqnarray*}%
which ends to show that (\ref{3.3}) is valid with $k_{1}=k_{2}=1$. Suppose
now that (\ref{3.3}) is true for a given $k_{1}$. We will show that it is
also true for $k_{1}+1$. We have%
\begin{eqnarray}
&&\mathcal{Z}\left( F_{sn+m_{1}}^{k_{1}+1}\left( x,y\right)
F_{sn+m_{2}}^{k_{2}}\left( x,y\right) \right)  \label{3.35} \\
&=&\left( x^{2}+4y\right) ^{-\frac{k_{1}+1+k_{2}}{2}}\beta ^{m_{1}\left(
k_{1}+1\right) +m_{2}k_{2}}\left( x,y\right)
\sum_{j=0}^{k_{1}+k_{2}+1}\sum_{i=0}^{k_{1}+1}\left( -1\right)
^{k_{1}+1+k_{2}-j}\dbinom{k_{1}+1}{i}\dbinom{k_{2}}{j-i}  \notag \\
&&\times \left( \frac{\alpha \left( x,y\right) }{\beta \left( x,y\right) }%
\right) ^{\left( m_{1}-m_{2}\right) i+m_{2}j}\frac{z}{z-\alpha ^{sj}\left(
x,y\right) \beta ^{s\left( k_{1}+1+k_{2}-j\right) }\left( x,y\right) }, 
\notag
\end{eqnarray}%
and we want to show that the right-hand side of (\ref{3.35}) is equal to 
\begin{equation}
z\frac{\sum\limits_{i=0}^{k_{1}+1+k_{2}}\sum\limits_{j=0}^{i}\left(
-1\right) ^{\frac{\left( sj+2(s+1)\right) \left( j+1\right) }{2}}\dbinom{%
k_{1}+k_{2}+2}{j}_{F\!_{s}\!\left( x,y\right) }F_{m_{1}+s\left( i-j\right)
}^{k_{1}+1}\left( x,y\right) F_{m_{2}+s\left( i-j\right) }^{k_{2}}\left(
x,y\right) y^{\frac{sj\left( j-1\right) }{2}}z^{k_{1}+1+k_{2}-i}}{%
\sum_{i=0}^{k_{1}+k_{2}+2}\left( -1\right) ^{\frac{\left( si+2(s+1)\right)
\left( i+1\right) }{2}}\dbinom{k_{1}+k_{2}+2}{i}_{F\!_{s}\!\left( x,y\right)
}y^{\frac{si\left( i-1\right) }{2}}z^{k_{1}+k_{2}+2-i}},  \label{3.36}
\end{equation}

Write the binomial coefficient $\binom{k_{1}+1}{i}$ (of the right-hand side
of (\ref{3.35})) as $\binom{k_{1}}{i}+\binom{k_{1}}{i-1}$, separate in two
sums and shift the indices of the second sum, to write (\ref{3.35}) as%
\begin{eqnarray*}
&&\mathcal{Z}\left( F_{sn+m_{1}}^{k_{1}+1}\left( x,y\right)
F_{sn+m_{2}}^{k_{2}}\left( x,y\right) \right) \\
&& \\
&=&\left( x^{2}+4y\right) ^{-\frac{k_{1}+1+k_{2}}{2}}\beta ^{m_{1}\left(
k_{1}+1\right) +m_{2}k_{2}}\left( x,y\right)
\sum_{j=0}^{k_{1}+k_{2}}\sum_{i=0}^{k_{1}}\left( -1\right) ^{k_{1}+1+k_{2}-j}
\\
&&\times \binom{k_{1}}{i}\dbinom{k_{2}}{j-i}\left( \frac{\alpha \left(
x,y\right) }{\beta \left( x,y\right) }\right) ^{\left( m_{1}-m_{2}\right)
i+m_{2}j}\frac{z}{z-\alpha ^{sj}\left( x,y\right) \beta ^{s\left(
k_{1}+1+k_{2}-j\right) }\left( x,y\right) } \\
&& \\
&&%
\begin{array}{ccccccccccccccccc}
&  &  &  &  &  &  &  &  &  &  &  &  &  &  &  & 
\end{array}%
+ \\
&& \\
&&\left( x^{2}+4y\right) ^{-\frac{k_{1}+1+k_{2}}{2}}\beta ^{m_{1}\left(
k_{1}+1\right) +m_{2}k_{2}}\left( x,y\right)
\sum_{j=0}^{k_{1}+k_{2}}\sum_{i=0}^{k_{1}}\left( -1\right) ^{k_{1}+k_{2}-j}
\\
&&\times \binom{k_{1}}{i}\dbinom{k_{2}}{j-i}\left( \frac{\alpha \left(
x,y\right) }{\beta \left( x,y\right) }\right) ^{\left( m_{1}-m_{2}\right)
\left( i+1\right) +m_{2}\left( j+1\right) }\frac{z}{z-\alpha ^{s\left(
j+1\right) }\left( x,y\right) \beta ^{s\left( k_{1}+k_{2}-j\right) }\left(
x,y\right) }
\end{eqnarray*}

or%
\begin{eqnarray*}
&&\mathcal{Z}\left( F_{sn+m_{1}}^{k_{1}+1}\left( x,y\right)
F_{sn+m_{2}}^{k_{2}}\left( x,y\right) \right) \\
&& \\
&=&-\left( x^{2}+4y\right) ^{-\frac{k_{1}+1+k_{2}}{2}}\beta ^{m_{1}\left(
k_{1}+1\right) +m_{2}k_{2}}\left( x,y\right)
\sum_{j=0}^{k_{1}+k_{2}}\sum_{i=0}^{k_{1}}\left( -1\right) ^{k_{1}+k_{2}-j}
\\
&&\times \binom{k_{1}}{i}\dbinom{k_{2}}{j-i}\left( \frac{\alpha \left(
x,y\right) }{\beta \left( x,y\right) }\right) ^{\left( m_{1}-m_{2}\right)
i+m_{2}j}\frac{\frac{z}{\beta ^{s}\left( x,y\right) }}{\frac{z}{\beta
^{s}\left( x,y\right) }-\alpha ^{sj}\left( x,y\right) \beta ^{s\left(
k_{1}+k_{2}-j\right) }\left( x,y\right) } \\
&& \\
&&%
\begin{array}{ccccccccccccccccc}
&  &  &  &  &  &  &  &  &  &  &  &  &  &  &  & 
\end{array}%
+ \\
&& \\
&&\left( x^{2}+4y\right) ^{-\frac{k_{1}+1+k_{2}}{2}}\beta ^{m_{1}\left(
k_{1}+1\right) +m_{2}k_{2}}\left( x,y\right) \left( \frac{\alpha \left(
x,y\right) }{\beta \left( x,y\right) }\right)
^{m_{1}}\sum_{j=0}^{k_{1}+k_{2}}\sum_{i=0}^{k_{1}}\left( -1\right)
^{k_{1}+k_{2}-j} \\
&&\times \binom{k_{1}}{i}\dbinom{k_{2}}{j-i}\left( \frac{\alpha \left(
x,y\right) }{\beta \left( x,y\right) }\right) ^{\left( m_{1}-m_{2}\right)
i+m_{2}j}\frac{\frac{z}{\alpha ^{s}\left( x,y\right) }}{\frac{z}{\alpha
^{s}\left( x,y\right) }-\alpha ^{sj}\left( x,y\right) \beta ^{s\left(
k_{1}+k_{2}-j\right) }\left( x,y\right) }.
\end{eqnarray*}%
Now use the induction hypothesis to write%
\begin{eqnarray*}
&&\mathcal{Z}\left( F_{sn+m_{1}}^{k_{1}+1}\left( x,y\right)
F_{sn+m_{2}}^{k_{2}}\left( x,y\right) \right) \\
&=&\frac{-\beta ^{m_{1}}\left( x,y\right) z}{\sqrt{x^{2}+4y}\beta ^{s}\left(
x,y\right) }\frac{%
\begin{array}{c}
\sum\limits_{i=0}^{k_{1}+k_{2}}\sum\limits_{j=0}^{i}\left( -1\right) ^{\frac{%
\left( sj+2(s+1)\right) \left( j+1\right) }{2}}\dbinom{k_{1}+k_{2}+1}{j}%
_{F\!_{s}\!\left( x,y\right) } \\ 
\times F_{m_{1}+s\left( i-j\right) }^{k_{1}}\left( x,y\right)
F_{m_{2}+s\left( i-j\right) }^{k_{2}}\left( x,y\right) y^{\frac{sj\left(
j-1\right) }{2}}\left( \frac{z}{\beta ^{s}\left( x,y\right) }\right)
^{k_{1}+k_{2}-i}%
\end{array}%
}{\sum\limits_{i=0}^{k_{1}+k_{2}+1}\left( -1\right) ^{\frac{\left(
si+2(s+1)\right) \left( i+1\right) }{2}}\dbinom{k_{1}+k_{2}+1}{i}%
_{F\!_{s}\!\left( x,y\right) }y^{\frac{si\left( i-1\right) }{2}}\left( \frac{%
z}{\beta ^{s}\left( x,y\right) }\right) ^{k_{1}+k_{2}+1-i}} \\
&& \\
&&%
\begin{array}{ccccccccccccccccc}
&  &  &  &  &  &  &  &  &  &  &  &  &  &  &  & 
\end{array}%
+ \\
&& \\
&&\frac{\alpha ^{m_{1}}\left( x,y\right) z}{\sqrt{x^{2}+4y}\alpha ^{s}\left(
x,y\right) }\frac{%
\begin{array}{c}
\sum\limits_{i=0}^{k_{1}+k_{2}}\sum\limits_{j=0}^{i}\left( -1\right) ^{\frac{%
\left( sj+2(s+1)\right) \left( j+1\right) }{2}}\dbinom{k_{1}+k_{2}+1}{j}%
_{F\!_{s}\!\left( x,y\right) } \\ 
\times F_{m_{1}+s\left( i-j\right) }^{k_{1}}\left( x,y\right)
F_{m_{2}+s\left( i-j\right) }^{k_{2}}\left( x,y\right) y^{\frac{sj\left(
j-1\right) }{2}}\left( \frac{z}{\alpha ^{s}\left( x,y\right) }\right)
^{k_{1}+k_{2}-i}%
\end{array}%
}{\sum\limits_{i=0}^{k_{1}+k_{2}+1}\left( -1\right) ^{\frac{\left(
si+2(s+1)\right) \left( i+1\right) }{2}}\dbinom{k_{1}+k_{2}+1}{i}%
_{F\!_{s}\!\left( x,y\right) }y^{\frac{si\left( i-1\right) }{2}}\left( \frac{%
z}{\alpha ^{s}\left( x,y\right) }\right) ^{k_{1}+k_{2}+1-i}}.
\end{eqnarray*}

Some further simplifications give us 
\begin{eqnarray*}
&&\mathcal{Z}\left( F_{sn+m_{1}}^{k_{1}+1}\left( x,y\right)
F_{sn+m_{2}}^{k_{2}}\left( x,y\right) \right) \\
&=&-\frac{\beta ^{m_{1}}\left( x,y\right) z}{\sqrt{x^{2}+4y}}\frac{%
\begin{array}{c}
\sum\limits_{i=0}^{k_{1}+k_{2}}\sum\limits_{j=0}^{i}\left( -1\right) ^{\frac{%
\left( sj+2(s+1)\right) \left( j+1\right) }{2}}\!\dbinom{k_{1}+k_{2}+1}{j}%
_{F\!_{s}\!\left( x,y\right) }\! \\ 
\times F_{m_{1}+s\left( i-j\right) }^{k_{1}}\left( x,y\right)
F_{m_{2}+s\left( i-j\right) }^{k_{2}}\left( x,y\right) \!y^{\frac{sj\left(
j-1\right) }{2}}\beta ^{si}\left( x,y\right) z^{k_{1}+k_{2}-i}%
\end{array}%
}{\sum\limits_{i=0}^{k_{1}+k_{2}+1}\left( -1\right) ^{\frac{\left(
si+2(s+1)\right) \left( i+1\right) }{2}}\dbinom{k_{1}+k_{2}+1}{i}%
_{F\!_{s}\!\left( x,y\right) }y^{\frac{si\left( i-1\right) }{2}}\beta
^{si}\left( x,y\right) z^{k_{1}+k_{2}+1-i}} \\
&& \\
&&+\frac{\alpha ^{m_{1}}\left( x,y\right) z}{\sqrt{x^{2}+4y}}\frac{%
\begin{array}{c}
\sum\limits_{i=0}^{k_{1}+k_{2}}\sum\limits_{j=0}^{i}\left( -1\right) ^{\frac{%
\left( sj+2(s+1)\right) \left( j+1\right) }{2}}\!\dbinom{k_{1}+k_{2}+1}{j}%
_{F\!_{s}\!\left( x,y\right) }\! \\ 
\times \!F_{m_{1}+s\left( i-j\right) }^{k_{1}}\left( x,y\right)
F_{m_{2}+s\left( i-j\right) }^{k_{2}}\left( x,y\right) \!y^{\frac{sj\left(
j-1\right) }{2}}\alpha ^{si}\left( x,y\right) z^{k_{1}+k_{2}-i}%
\end{array}%
}{\sum\limits_{i=0}^{k_{1}+k_{2}+1}\left( -1\right) ^{\frac{\left(
si+2(s+1)\right) \left( i+1\right) }{2}}\dbinom{k_{1}+k_{2}+1}{i}%
_{F\!_{s}\!\left( x,y\right) }y^{\frac{si\left( i-1\right) }{2}}\alpha
^{si}\left( x,y\right) z^{k_{1}+k_{2}+1-i}} \\
&& \\
&& \\
&=&\frac{z}{\sqrt{x^{2}+4y}}\!\sum_{i=0}^{k_{1}+k_{2}}\sum_{j=0}^{i}\left(
-1\right) ^{\frac{\left( sj+2(s+1)\right) \left( j+1\right) }{2}}\!\dbinom{%
k_{1}+k_{2}+1}{j}_{F\!_{s}\!\left( x,y\right) }\!\!F_{m_{1}+s\left(
i-j\right) }^{k_{1}}\!\left( x,y\right) F_{m_{2}+s\left( i-j\right)
}^{k_{2}}\!\left( x,y\right) \!y^{\frac{sj\left( j-1\right) }{2}} \\
&&\times \left( 
\begin{array}{c}
\dfrac{\alpha ^{si+m_{1}}\left( x,y\right) }{\sum%
\limits_{i=0}^{k_{1}+k_{2}+1}\left( -1\right) ^{\frac{\left(
si+2(s+1)\right) \left( i+1\right) }{2}}\dbinom{k_{1}+k_{2}+1}{i}%
_{F\!_{s}\!\left( x,y\right) }y^{\frac{si\left( i-1\right) }{2}}\alpha
^{si}\left( x,y\right) z^{k_{1}+k_{2}+1-i}} \\ 
\\ 
-\dfrac{\beta ^{si+m_{1}}\left( x,y\right) }{\sum%
\limits_{i=0}^{k_{1}+k_{2}+1}\left( -1\right) ^{\frac{\left(
si+2(s+1)\right) \left( i+1\right) }{2}}\dbinom{k_{1}+k_{2}+1}{i}%
_{F\!_{s}\!\left( x,y\right) }y^{\frac{si\left( i-1\right) }{2}}\beta
^{si}\left( x,y\right) z^{k_{1}+k_{2}+1-i}}%
\end{array}%
\right) z^{k_{1}+k_{2}-i}
\end{eqnarray*}

By using (\ref{2.25}) we can write%
\begin{eqnarray}
&&\mathcal{Z}\left( F_{sn+m_{1}}^{k_{1}+1}\left( x,y\right)
F_{sn+m_{2}}^{k_{2}}\left( x,y\right) \right)  \label{3.31} \\
&=&z\frac{%
\begin{array}{c}
\sum\limits_{i=0}^{k_{1}+k_{2}}\sum\limits_{j=0}^{i}\left( -1\right) ^{\frac{%
\left( sj+2(s+1)\right) \left( j+1\right) }{2}}\dbinom{k_{1}+k_{2}+1}{j}%
_{F\!_{s}\!\left( x,y\right) }F_{m_{1}+s\left( i-j\right) }^{k_{1}}\left(
x,y\right) F_{m_{2}+s\left( i-j\right) }^{k_{2}}\left( x,y\right) \\ 
\times \left( F_{si+m_{1}}\left( x,y\right) z+\left( -y\right)
^{si+m_{1}}F_{s\left( k_{1}+k_{2}-i+1\right) -m_{1}}\left( x,y\right)
\right) y^{\frac{sj\left( j-1\right) }{2}}z^{k_{1}+k_{2}-i}%
\end{array}%
}{\sum\limits_{i=0}^{k_{1}+k_{2}+2}\left( -1\right) ^{\frac{\left(
si+2(s+1)\right) \left( i+1\right) }{2}}\dbinom{k_{1}+k_{2}+2}{i}%
_{F\!_{s}\!\left( x,y\right) }y^{\frac{si\left( i-1\right) }{2}%
}z^{k_{1}+k_{2}+2-i}}.  \notag
\end{eqnarray}

Observe that (\ref{3.31}) has the expected denominator (of (\ref{3.36})).
Let us work with the numerator. We have%
\begin{eqnarray}
&&z\sum_{i=0}^{k_{1}+k_{2}}\sum_{j=0}^{i}\left( -1\right) ^{\frac{\left(
sj+2(s+1)\right) \left( j+1\right) }{2}}\dbinom{k_{1}+k_{2}+1}{j}%
_{F\!_{s}\!\left( x,y\right) }F_{m_{1}+s\left( i-j\right) }^{k_{1}}\left(
x,y\right) F_{m_{2}+s\left( i-j\right) }^{k_{2}}\left( x,y\right)
\label{3.4} \\
&&\times \left( F_{si+m_{1}}\left( x,y\right) z+\left( -y\right)
^{si+m_{1}}F_{s\left( k_{1}+k_{2}-i+1\right) -m_{1}}\left( x,y\right)
\right) y^{\frac{sj\left( j-1\right) }{2}}z^{k_{1}+k_{2}-i}  \notag \\
&&  \notag \\
&=&z\sum_{i=0}^{k_{1}+k_{2}}\sum_{j=0}^{i}\left( -1\right) ^{\frac{\left(
sj+2(s+1)\right) \left( j+1\right) }{2}}\dbinom{k_{1}+k_{2}+1}{j}%
_{F\!_{s}\!\left( x,y\right) }  \notag \\
&&\times F_{m_{1}+s\left( i-j\right) }^{k_{1}}\left( x,y\right)
F_{m_{2}+s\left( i-j\right) }^{k_{2}}\left( x,y\right) F_{si+m_{1}}\left(
x,y\right) y^{\frac{sj\left( j-1\right) }{2}}z^{k_{1}+1+k_{2}-i}  \notag \\
&&  \notag \\
&&+z\sum_{i=1}^{k_{1}+k_{2}+1}\sum_{j=1}^{i}\left( -1\right) ^{\frac{\left(
s\left( j-1\right) +2(s+1)\right) j}{2}}\dbinom{k_{1}+k_{2}+1}{j-1}%
_{F\!_{s}\!\left( x,y\right) }  \notag \\
&&\times F_{m_{1}+s\left( i-j\right) }^{k_{1}}\left( x,y\right)
F_{m_{2}+s\left( i-j\right) }^{k_{2}}\left( x,y\right) F_{s\left(
k_{1}+k_{2}-i+2\right) -m_{1}}\left( x,y\right) \left( -y\right)
^{si+m_{1}-s}y^{\frac{s\left( j-1\right) \left( j-2\right) }{2}%
}z^{k_{1}+1+k_{2}-i}.  \notag
\end{eqnarray}

Since%
\begin{equation*}
\sum_{j=0}^{k_{1}+k_{2}+1}\left( -1\right) ^{\frac{\left( sj+2(s+1)\right)
\left( j+1\right) }{2}}\dbinom{k_{1}+k_{2}+1}{j}_{F\!_{s}\!\left( x,y\right)
}F_{m_{1}+s\left( k_{1}+k_{2}+1-j\right) }^{k_{1}}\left( x,y\right)
F_{m_{2}+s\left( k_{1}+k_{2}+1-j\right) }^{k_{2}}\left( x,y\right) y^{\frac{%
sj\left( j-1\right) }{2}}=0,
\end{equation*}%
we can write the right-hand side of (\ref{3.4}) as%
\begin{eqnarray*}
&&z\!\sum_{i=0}^{k_{1}+k_{2}+1}\!\sum_{j=0}^{i}\!\left( -1\right) ^{\frac{%
\left( sj+2(s+1)\right) \left( j+1\right) }{2}}\!\dbinom{k_{1}+k_{2}+2}{j}%
_{\!F\!_{s}\!\left( x,y\right) }\frac{F_{m_{1}+s\left( i-j\right)
}^{k_{1}}\!\left( x,y\right) F_{m_{2}+s\left( i-j\right) }^{k_{2}}\!\left(
x,y\right) }{F_{s\left( k_{1}+k_{2}+2\right) }\!\left( x,y\right) }\text{ }%
y^{\frac{sj\left( j-1\right) }{2}}z^{k_{1}+1+k_{2}-i} \\
&&\times \left( F_{si+m_{1}}\left( x,y\right) F_{s\left(
k_{1}+k_{2}+2-j\right) }\left( x,y\right) -\left( -y\right) ^{s\left(
i-j\right) +m_{1}}F_{s\left( k_{1}+k_{2}-i+2\right) -m_{1}}\left( x,y\right)
F_{sj}\left( x,y\right) \left( x,y\right) \right) .
\end{eqnarray*}

Finally, by using (\ref{1.11}) we see that%
\begin{eqnarray*}
&&F_{si+m_{1}}\left( x,y\right) F_{s\left( k_{1}+k_{2}+2-j\right) }\left(
x,y\right) -\left( -y\right) ^{s\left( i-j\right) +m_{1}}F_{s\left(
k_{1}+k_{2}-i+2\right) -m_{1}}\left( x,y\right) F_{sj}\left( x,y\right) \\
&=&F_{s\left( k_{1}+k_{2}+2\right) }\left( x,y\right) F_{m_{1}+s\left(
i-j\right) }\left( x,y\right) ,
\end{eqnarray*}%
and then the right-hand side of (\ref{3.4}) is our expected numerator, namely%
\begin{equation*}
z\sum_{i=0}^{k_{1}+k_{2}+1}\sum_{j=0}^{i}\left( -1\right) ^{\frac{\left(
sj+2(s+1)\right) \left( j+1\right) }{2}}\dbinom{k_{1}+k_{2}+2}{j}%
_{F\!_{s}\!\left( x,y\right) }F_{m_{1}+s\left( i-j\right) }^{k_{1}+1}\left(
x,y\right) F_{m_{2}+s\left( i-j\right) }^{k_{2}}\left( x,y\right) y^{\frac{%
sj\left( j-1\right) }{2}}z^{k_{1}+1+k_{2}-i}.
\end{equation*}

This ends our induction argument.
\end{proof}

\begin{theorem}
\label{Th3.2}\textit{Let }$m_{1},m_{2}\in \mathbb{Z}$\textit{\ and }$%
t_{1},t_{2}\in \mathbb{N}^{\prime }$\textit{\ be given. The sequence }$%
F_{t_{1}sn+m_{1}}\left( x,y\right) F_{t_{2}sn+m_{2}}\left( x,y\right) $%
\textit{\ has }$Z$\ transform \textit{given by:}%
\begin{eqnarray}
&&\mathcal{Z}\left( F_{t_{1}sn+m_{1}}\left( x,y\right)
F_{t_{2}sn+m_{2}}\left( x,y\right) \right)  \label{3.5} \\
&&  \notag \\
&=&z\frac{\sum\limits_{i=0}^{t_{1}+t_{2}}\sum\limits_{j=0}^{i}\left(
-1\right) ^{\frac{\left( sj+2(s+1)\right) \left( j+1\right) }{2}}\dbinom{%
t_{1}+t_{2}+1}{j}_{F\!_{s}\!\left( x,y\right) }F_{m_{1}+t_{1}s\left(
i-j\right) }\left( x,y\right) F_{m_{2}+t_{2}s\left( i-j\right) }\left(
x,y\right) y^{\frac{sj\left( j-1\right) }{2}}z^{t_{1}+t_{2}-i}}{%
\sum\limits_{i=0}^{t_{1}+t_{2}+1}\left( -1\right) ^{\frac{\left(
si+2(s+1)\right) \left( i+1\right) }{2}}\dbinom{t_{1}+t_{2}+1}{i}%
_{F\!_{s}\!\left( x,y\right) }y^{\frac{si\left( i-1\right) }{2}%
}z^{t_{1}+t_{2}+1-i}}.  \notag
\end{eqnarray}
\end{theorem}

\begin{proof}
We will proceed by induction on the parameters $t_{1}$ and/or $t_{2}$. (As
in the proof of theorem \ref{Th3.1}, the symmetry of (\ref{3.5}) with
respect to $t_{1}$ and $t_{2}$ allows us to use induction on any of these
parameters.) The case $t_{1}=t_{2}=0$ is trivial and in the case $%
t_{1}=t_{2}=1$ the result is true by theorem \ref{Th3.1}. Suppose now the
result is true for a given $t_{1}\in \mathbb{N}$ together with all $t\in 
\mathbb{N}$, $t\leq t_{1}$, and let us prove that it is also true for $%
t_{1}+1$. We will show that%
\begin{eqnarray}
&&\mathcal{Z}\left( F_{\left( t_{1}+1\right) sn+m_{1}}\left( x,y\right)
F_{t_{2}sn+m_{2}}\left( x,y\right) \right)  \label{3.54} \\
&&  \notag \\
&=&z\frac{%
\begin{array}{c}
\sum\limits_{i=0}^{t_{1}+t_{2}+1}\sum\limits_{j=0}^{i}\left( -1\right) ^{%
\frac{\left( sj+2(s+1)\right) \left( j+1\right) }{2}}\dbinom{t_{1}+t_{2}+2}{j%
}_{F\!_{s}\!\left( x,y\right) } \\ 
\times F_{m_{1}+\left( t_{1}+1\right) s\left( i-j\right) }\left( x,y\right)
F_{m_{2}+t_{2}s\left( i-j\right) }\left( x,y\right) y^{\frac{sj\left(
j-1\right) }{2}}z^{t_{1}+t_{2}+1-i}%
\end{array}%
}{\sum\limits_{i=0}^{t_{1}+t_{2}+2}\left( -1\right) ^{\frac{\left(
si+2(s+1)\right) \left( i+1\right) }{2}}\dbinom{t_{1}+t_{2}+2}{i}%
_{F\!_{s}\!\left( x,y\right) }y^{\frac{si\left( i-1\right) }{2}%
}z^{t_{1}+t_{2}+2-i}}.  \notag
\end{eqnarray}

By using (\ref{1.11}) with $N=\left( t_{1}+1\right) sn+m_{1}$, $K=\left(
t_{1}-1\right) sn+m_{1}$ and $M=sn$, we see that%
\begin{equation*}
F_{\left( t_{1}+1\right) sn+m_{1}}\left( x,y\right) =F_{t_{1}sn+m_{1}}\left(
x,y\right) L_{sn}\left( x,y\right) -\left( -y\right) ^{sn}F_{\left(
t_{1}-1\right) sn+m_{1}}\left( x,y\right) ,
\end{equation*}

Then we have that%
\begin{eqnarray}
&&\mathcal{Z}\left( F_{\left( t_{1}+1\right) sn+m_{1}}\left( x,y\right)
F_{t_{2}sn+m_{2}}\left( x,y\right) \right)  \label{3.55} \\
&=&\mathcal{Z}\left( L_{sn}\left( x,y\right) F_{t_{1}sn+m_{1}}\left(
x,y\right) F_{t_{2}sn+m_{2}}\left( x,y\right) \right) -\mathcal{Z}\left(
\left( -y\right) ^{sn}F_{\left( t_{1}-1\right) sn+m_{1}}\left( x,y\right)
F_{t_{2}sn+m_{2}}\left( x,y\right) \right) .  \notag
\end{eqnarray}

Observe that induction hypothesis and (\ref{2.5}) give us that%
\begin{eqnarray*}
&&\mathcal{Z}\left( L_{sn}\left( x,y\right) F_{t_{1}sn+m_{1}}\left(
x,y\right) F_{t_{2}sn+m_{2}}\left( x,y\right) \right) \\
&& \\
&=&\frac{%
\begin{array}{c}
\frac{z}{\alpha ^{s}\left( x,y\right) }\!\sum\limits_{i=0}^{t_{1}+t_{2}}\!%
\sum\limits_{j=0}^{i}\!\!\left( -1\right) ^{\frac{\left( sj+2(s+1)\right)
\left( j+1\right) }{2}}\!\!\dbinom{t_{1}+t_{2}+1}{j}_{\!F\!_{s}\!\left(
x,y\right) } \\ 
\times F_{m_{1}+t_{1}s\left( i-j\right) }\left( x,y\right)
F_{m_{2}+t_{2}s\left( i-j\right) }\left( x,y\right) \!y^{\frac{sj\left(
j-1\right) }{2}}\!\left( \frac{z}{\alpha ^{s}\left( x,y\right) }\right)
^{t_{1}+t_{2}-i}%
\end{array}%
}{\sum\limits_{i=0}^{t_{1}+t_{2}+1}\left( -1\right) ^{\frac{\left(
si+2(s+1)\right) \left( i+1\right) }{2}}\dbinom{t_{1}+t_{2}+1}{i}%
_{F\!_{s}\!\left( x,y\right) }y^{\frac{si\left( i-1\right) }{2}}\left( \frac{%
z}{\alpha ^{s}\left( x,y\right) }\right) ^{t_{1}+t_{2}+1-i}} \\
&& \\
&&%
\begin{array}{ccccccccccccccc}
&  &  &  &  &  &  &  &  &  &  &  &  &  & +%
\end{array}
\\
&& \\
&&\frac{%
\begin{array}{c}
\frac{z}{\beta ^{s}\left( x,y\right) }\!\sum\limits_{i=0}^{t_{1}+t_{2}}\!%
\sum\limits_{j=0}^{i}\!\!\left( -1\right) ^{\frac{\left( sj+2(s+1)\right)
\left( j+1\right) }{2}}\!\!\dbinom{t_{1}+t_{2}+1}{j}_{\!F\!_{s}\!\left(
x,y\right) }\!\! \\ 
\times F_{m_{1}+t_{1}s\left( i-j\right) }\left( x,y\right)
F_{m_{2}+t_{2}s\left( i-j\right) }\left( x,y\right) \!y^{\frac{sj\left(
j-1\right) }{2}}\!\left( \frac{z}{\beta ^{s}\left( x,y\right) }\right)
^{t_{1}+t_{2}-i}%
\end{array}%
}{\sum\limits_{i=0}^{t_{1}+t_{2}+1}\left( -1\right) ^{\frac{\left(
si+2(s+1)\right) \left( i+1\right) }{2}}\dbinom{t_{1}+t_{2}+1}{i}%
_{F\!_{s}\!\left( x,y\right) }y^{\frac{si\left( i-1\right) }{2}}\left( \frac{%
z}{\beta ^{s}\left( x,y\right) }\right) ^{t_{1}+t_{2}+1-i}},
\end{eqnarray*}%
or, after some simplifications%
\begin{eqnarray}
&&\mathcal{Z}\left( L_{sn}\left( x,y\right) F_{t_{1}sn+m_{1}}\left(
x,y\right) F_{t_{2}sn+m_{2}}\left( x,y\right) \right)  \label{3.551} \\
&=&z\sum_{i=0}^{t_{1}+t_{2}}\sum_{j=0}^{i}\left( -1\right) ^{\frac{\left(
sj+2(s+1)\right) \left( j+1\right) }{2}}\dbinom{t_{1}+t_{2}+1}{j}%
_{F\!_{s}\!\left( x,y\right) }F_{m_{1}+t_{1}s\left( i-j\right) }\left(
x,y\right) F_{m_{2}+t_{2}s\left( i-j\right) }\left( x,y\right)  \notag \\
&&  \notag \\
&&\times \left( 
\begin{array}{c}
\dfrac{\alpha ^{si}\left( x,y\right) }{\sum\limits_{i=0}^{t_{1}+t_{2}+1}%
\left( -1\right) ^{\frac{\left( si+2(s+1)\right) \left( i+1\right) }{2}}%
\dbinom{t_{1}+t_{2}+1}{i}_{F\!_{s}\!\left( x,y\right) }\alpha ^{si}\left(
x,y\right) y^{\frac{si\left( i-1\right) }{2}}z^{t_{1}+t_{2}+1-i}} \\ 
+ \\ 
\text{ }\dfrac{\beta ^{si}\left( x,y\right) }{\sum%
\limits_{i=0}^{t_{1}+t_{2}+1}\left( -1\right) ^{\frac{\left(
si+2(s+1)\right) \left( i+1\right) }{2}}\dbinom{t_{1}+t_{2}+1}{i}%
_{F\!_{s}\!\left( x,y\right) }\beta ^{si}\left( x,y\right) y^{\frac{si\left(
i-1\right) }{2}}z^{t_{1}+t_{2}+1-i}}%
\end{array}%
\right) y^{\frac{sj\left( j-1\right) }{2}}z^{t_{1}+t_{2}-i}.  \notag
\end{eqnarray}

According to (\ref{2.24}) we can write (\ref{3.551}) as%
\begin{eqnarray}
&&\mathcal{Z}\left( L_{sn}\left( x,y\right) F_{t_{1}sn+m_{1}}\left(
x,y\right) F_{t_{2}sn+m_{2}}\left( x,y\right) \right)  \label{3.552} \\
&=&\frac{z}{\sum\limits_{i=0}^{t_{1}+t_{2}+2}\left( -1\right) ^{\frac{\left(
si+2(s+1)\right) \left( i+1\right) }{2}}\dbinom{t_{1}+t_{2}+2}{i}%
_{F\!_{s}\!\left( x,y\right) }y^{\frac{si\left( i-1\right) }{2}%
}z^{t_{1}+t_{2}+2-i}}  \notag \\
&&\times \sum_{i=0}^{t_{1}+t_{2}}\sum_{j=0}^{i}\left( -1\right) ^{\frac{%
\left( sj+2(s+1)\right) \left( j+1\right) }{2}}\dbinom{t_{1}+t_{2}+1}{j}%
_{F\!_{s}\!\left( x,y\right) }F_{m_{1}+t_{1}s\left( i-j\right) }\left(
x,y\right) F_{m_{2}+t_{2}s\left( i-j\right) }\left( x,y\right)  \notag \\
&&\times \left( L_{si}\left( x,y\right) z-\left( -y\right) ^{si}L_{s\left(
t_{1}+t_{2}-i+1\right) }\left( x,y\right) \right) y^{\frac{sj\left(
j-1\right) }{2}}z^{t_{1}+t_{2}-i}  \notag
\end{eqnarray}

On the other hand, induction hypothesis together with (\ref{2.2}) give us%
\begin{eqnarray*}
&&\mathcal{Z}\left( \left( -y\right) ^{sn}F_{\left( t_{1}-1\right)
sn+m_{1}}\left( x,y\right) F_{t_{2}sn+m_{2}}\left( x,y\right) \right) \\
&& \\
&=&\frac{%
\begin{array}{c}
\frac{z}{\left( -y\right) ^{s}}\!\sum\limits_{i=0}^{t_{1}+t_{2}-1}\!\sum%
\limits_{j=0}^{i}\!\!\left( -1\right) ^{\frac{\left( sj+2(s+1)\right) \left(
j+1\right) }{2}}\!\!\dbinom{t_{1}+t_{2}}{j}_{\!F\!_{s}\!\left( x,y\right)
}\!\! \\ 
\times F_{m_{1}+\left( t_{1}-1\right) s\left( i-j\right) }\left( x,y\right)
F_{m_{2}+t_{2}s\left( i-j\right) }\left( x,y\right) \!y^{\frac{sj\left(
j-1\right) }{2}}\!\left( \frac{z}{\left( -y\right) ^{s}}\right)
^{t_{1}-1+t_{2}-i}%
\end{array}%
}{\sum\limits_{i=0}^{t_{1}+t_{2}}\left( -1\right) ^{\frac{\left(
si+2(s+1)\right) \left( i+1\right) }{2}}\dbinom{t_{1}+t_{2}}{i}%
_{F\!_{s}\!\left( x,y\right) }y^{\frac{si\left( i-1\right) }{2}}\left( \frac{%
z}{\left( -y\right) ^{s}}\right) ^{t_{1}+t_{2}-i}},
\end{eqnarray*}%
or%
\begin{eqnarray}
&&\mathcal{Z}\left( \left( -y\right) ^{sn}F_{\left( t_{1}-1\right)
sn+m_{1}}\left( x,y\right) F_{t_{2}sn+m_{2}}\left( x,y\right) \right)
\label{3.553} \\
&&  \notag \\
&=&z\frac{%
\begin{array}{c}
\sum\limits_{i=0}^{t_{1}+t_{2}-1}\sum\limits_{j=0}^{i}\!\left( -1\right) ^{%
\frac{\left( sj+2(s+1)\right) \left( j+1\right) }{2}}\!\dbinom{t_{1}+t_{2}}{j%
}_{F\!_{s}\!\left( x,y\right) }\! \\ 
\times F_{m_{1}+\left( t_{1}-1\right) s\left( i-j\right) }\left( x,y\right)
F_{m_{2}+t_{2}s\left( i-j\right) }\left( x,y\right) \!\left( -y\right)
^{si}\!y^{\frac{sj\left( j-1\right) }{2}}z^{t_{1}-1+t_{2}-i}%
\end{array}%
}{\sum\limits_{i=0}^{t_{1}+t_{2}}\left( -1\right) ^{\frac{\left(
si+2(s+1)\right) \left( i+1\right) }{2}}\dbinom{t_{1}+t_{2}}{i}%
_{F\!_{s}\!\left( x,y\right) }\left( -y\right) ^{si}y^{\frac{si\left(
i-1\right) }{2}}z^{t_{1}+t_{2}-i}}.  \notag
\end{eqnarray}

We can use (\ref{2.27}) to write (\ref{3.553}) as%
\begin{eqnarray}
&&\mathcal{Z}\left( \left( -y\right) ^{sn}F_{\left( t_{1}-1\right)
sn+m_{1}}\left( x,y\right) F_{t_{2}sn+m_{2}}\left( x,y\right) \right)
\label{3.554} \\
&&  \notag \\
&=&z\frac{%
\begin{array}{c}
\sum\limits_{i=0}^{t_{1}+t_{2}-1}\sum\limits_{j=0}^{i}\!\left( -1\right) ^{%
\frac{\left( sj+2(s+1)\right) \left( j+1\right) }{2}}\!\dbinom{t_{1}+t_{2}}{j%
}_{F\!_{s}\!\left( x,y\right) }\!F_{m_{1}+\left( t_{1}-1\right) s\left(
i-j\right) }\left( x,y\right) F_{m_{2}+t_{2}s\left( i-j\right) }\left(
x,y\right) \! \\ 
\times \left( z^{2}-L_{s\left( t_{1}+t_{2}+1\right) }\left( x,y\right)
z+\left( -y\right) ^{s\left( t_{1}+t_{2}+1\right) }\right) \left( -y\right)
^{si}\!y^{\frac{sj\left( j-1\right) }{2}}z^{t_{1}-1+t_{2}-i}%
\end{array}%
}{\sum\limits_{i=0}^{t_{1}+t_{2}+2}\left( -1\right) ^{\frac{\left(
si+2(s+1)\right) \left( i+1\right) }{2}}\dbinom{t_{1}+t_{2}+2}{i}%
_{F\!_{s}\!\left( x,y\right) }y^{\frac{si\left( i-1\right) }{2}%
}z^{t_{1}+t_{2}+2-i}}.  \notag
\end{eqnarray}

Thus, with (\ref{3.552}) and (\ref{3.554}) we can write (\ref{3.55}) as%
\begin{eqnarray}
&&\mathcal{Z}\left( F_{\left( t_{1}+1\right) sn+m_{1}}\left( x,y\right)
F_{t_{2}sn+m_{2}}\left( x,y\right) \right)  \label{3.7} \\
&&  \notag \\
&=&\frac{z}{\sum\limits_{i=0}^{t_{1}+t_{2}+2}\left( -1\right) ^{\frac{\left(
si+2(s+1)\right) \left( i+1\right) }{2}}\dbinom{t_{1}+t_{2}+2}{i}%
_{F\!_{s}\!\left( x,y\right) }y^{\frac{si\left( i-1\right) }{2}%
}z^{t_{1}+t_{2}+2-i}}  \notag \\
&&  \notag \\
&&\times \left( 
\begin{array}{c}
\sum\limits_{i=0}^{t_{1}+t_{2}}\sum\limits_{j=0}^{i}\left( -1\right) ^{\frac{%
\left( sj+2(s+1)\right) \left( j+1\right) }{2}}\dbinom{t_{1}+t_{2}+1}{j}%
_{F\!_{s}\!\left( x,y\right) }F_{m_{1}+t_{1}s\left( i-j\right) }\left(
x,y\right) F_{m_{2}+t_{2}s\left( i-j\right) }\left( x,y\right) \\ 
\times \left( L_{si}\left( x,y\right) z-\left( -y\right) ^{si}L_{s\left(
t_{1}+t_{2}-i+1\right) }\left( x,y\right) \right) y^{\frac{sj\left(
j-1\right) }{2}}z^{t_{1}+t_{2}-i} \\ 
\\ 
-\sum\limits_{i=0}^{t_{1}+t_{2}-1}\sum\limits_{j=0}^{i}\left( -1\right) ^{%
\frac{\left( sj+2(s+1)\right) \left( j+1\right) }{2}}\dbinom{t_{1}+t_{2}}{j}%
_{F\!_{s}\!\left( x,y\right) }F_{m_{1}+\left( t_{1}-1\right) s\left(
i-j\right) }\left( x,y\right) F_{m_{2}+t_{2}s\left( i-j\right) }\left(
x,y\right) \\ 
\times \left( z^{2}-L_{s\left( t_{1}+t_{2}+1\right) }\left( x,y\right)
z+\left( -y\right) ^{s\left( t_{1}+t_{2}+1\right) }\right) \left( -y\right)
^{si}y^{\frac{sj\left( j-1\right) }{2}}z^{t_{1}-1+t_{2}-i}%
\end{array}%
\right) .  \notag
\end{eqnarray}%
We have now the expected denominator (of (\ref{3.54})). Let us work with the
corresponding numerator (of (\ref{3.7})), $z\left( A\left( x,y;z\right)
-B\left( x,y;z\right) \right) $ say, where%
\begin{eqnarray*}
A\left( x,y;z\right)
&=&\sum\limits_{i=0}^{t_{1}+t_{2}}\sum\limits_{j=0}^{i}\left( -1\right) ^{%
\frac{\left( sj+2(s+1)\right) \left( j+1\right) }{2}}\dbinom{t_{1}+t_{2}+1}{j%
}_{F\!_{s}\!\left( x,y\right) }F_{m_{1}+t_{1}s\left( i-j\right) }\left(
x,y\right) F_{m_{2}+t_{2}s\left( i-j\right) }\left( x,y\right) \\
&&\times \left( L_{si}\left( x,y\right) z-\left( -y\right) ^{si}L_{s\left(
t_{1}+t_{2}-i+1\right) }\left( x,y\right) \right) y^{\frac{sj\left(
j-1\right) }{2}}z^{t_{1}+t_{2}-i},
\end{eqnarray*}%
and%
\begin{eqnarray*}
B\left( x,y;z\right)
&=&\sum\limits_{i=0}^{t_{1}+t_{2}-1}\sum\limits_{j=0}^{i}\left( -1\right) ^{%
\frac{\left( sj+2(s+1)\right) \left( j+1\right) }{2}}\dbinom{t_{1}+t_{2}}{j}%
_{F\!_{s}\!\left( x,y\right) }F_{m_{1}+\left( t_{1}-1\right) s\left(
i-j\right) }\left( x,y\right) F_{m_{2}+t_{2}s\left( i-j\right) }\left(
x,y\right) \\
&&\times \left( -y\right) ^{si}y^{\frac{sj\left( j-1\right) }{2}%
}z^{t_{1}-1+t_{2}-i}\left( z^{2}-L_{s\left( t_{1}+t_{2}+1\right) }\left(
x,y\right) z+\left( -y\right) ^{s\left( t_{1}+t_{2}+1\right) }\right)
\end{eqnarray*}

We have that%
\begin{eqnarray*}
&&A\left( x,y;z\right) \\
&=&\sum_{i=0}^{t_{1}+t_{2}}\sum_{j=0}^{i}\!\left( -1\right) ^{\frac{\left(
sj+2(s+1)\right) \left( j+1\right) }{2}}\!\!\dbinom{t_{1}+t_{2}+1}{j}%
_{\!F\!_{s}\!\left( x,y\right) }\!\! \\
&&\times F_{m_{1}+t_{1}s\left( i-j\right) }\left( x,y\right)
F_{m_{2}+t_{2}s\left( i-j\right) }\left( x,y\right) L_{si}\left( x,y\right)
y^{\frac{sj\left( j-1\right) }{2}}z^{t_{1}+t_{2}+1-i} \\
&& \\
&&-\sum_{i=1}^{t_{1}+t_{2}+1}\sum_{j=1}^{i}\!\left( -1\right) ^{\frac{\left(
sj-s+2(s+1)\right) j}{2}}\!\!\dbinom{t_{1}+t_{2}+1}{j-1}_{\!F\!_{s}\!\left(
x,y\right) }\!\! \\
&&\times F_{m_{1}+t_{1}s\left( i-j\right) }\left( x,y\right)
F_{m_{2}+t_{2}s\left( i-j\right) }\left( x,y\right) L_{s\left(
t_{1}+t_{2}-i+2\right) }\left( x,y\right) \left( -y\right) ^{s\left(
i-1\right) }y^{\frac{s\left( j-1\right) \left( j-2\right) }{2}%
}z^{t_{1}+t_{2}+1-i} \\
&& \\
&=&\sum_{i=0}^{t_{1}+t_{2}+1}\sum_{j=0}^{i}\left( -1\right) ^{\frac{\left(
sj+2(s+1)\right) \left( j+1\right) }{2}}\dbinom{t_{1}+t_{2}+2}{j}%
_{F\!_{s}\!\left( x,y\right) }\frac{F_{m_{2}+t_{2}s\left( i-j\right) }\left(
x,y\right) F_{m_{1}+t_{1}s\left( i-j\right) }\left( x,y\right) }{F_{s\left(
t_{1}+t_{2}+2\right) }\left( x,y\right) }y^{\frac{sj\left( j-1\right) }{2}}
\\
&&\times \left( F_{s\left( t_{1}+t_{2}+2-j\right) }\left( x,y\right)
L_{si}\left( x,y\right) +\left( -y\right) ^{s\left( i-j\right) }F_{sj}\left(
x,y\right) L_{s\left( t_{1}+t_{2}-i+2\right) }\left( x,y\right) \right)
z^{t_{1}+1+t_{2}-i} \\
&=&\sum_{i=0}^{t_{1}+t_{2}+1}\sum_{j=0}^{i}\left( -1\right) ^{\frac{\left(
sj+2(s+1)\right) \left( j+1\right) }{2}}\dbinom{t_{1}+t_{2}+2}{j}%
_{F\!_{s}\!\left( x,y\right) } \\
&&\times F_{m_{2}+t_{2}s\left( i-j\right) }\left( x,y\right)
F_{m_{1}+t_{1}s\left( i-j\right) }\left( x,y\right) L_{s\left( i-j\right)
}\left( x,y\right) y^{\frac{sj\left( j-1\right) }{2}}z^{t_{1}+1+t_{2}-i}.
\end{eqnarray*}

In the last step we used (\ref{1.12}) with $M=si,N=s\left(
t_{1}+t_{2}+2-j\right) $ and $K=-sj$.

Now let us work with $B\left( x,y;z\right) $. We have%
\begin{eqnarray*}
&&B\left( x,y;z\right) \\
&=&\sum\limits_{i=0}^{t_{1}+t_{2}-1}\sum\limits_{j=0}^{i}\!\left( -1\right)
^{\frac{\left( sj+2(s+1)\right) \left( j+1\right) }{2}}\!\dbinom{t_{1}+t_{2}%
}{j}_{\!F_{s}\!\left( x,y\right) }\! \\
&&\times F_{m_{1}+\left( t_{1}-1\right) s\left( i-j\right) }\left(
x,y\right) F_{m_{2}+t_{2}s\left( i-j\right) }\left( x,y\right) \left(
-y\right) ^{si}y^{\frac{sj\left( j-1\right) }{2}}z^{t_{1}+1+t_{2}-i} \\
&& \\
&&-\!\sum\limits_{i=1}^{t_{1}+t_{2}}\sum\limits_{j=1}^{i}\!\left( -1\right)
^{\frac{\left( s\left( j-1\right) +2(s+1)\right) j}{2}}\!\!\dbinom{%
t_{1}+t_{2}}{j-1}_{\!F_{s}\!\left( x,y\right) }\!\! \\
&&\times F_{m_{1}+\left( t_{1}-1\right) s\left( i-j\right) }\left(
x,y\right) F_{m_{2}+t_{2}s\left( i-j\right) }\left( x,y\right) L_{s\left(
t_{1}+t_{2}+1\right) }\left( x,y\right) \left( -y\right) ^{s\left(
i-1\right) }y^{\frac{s\left( j-1\right) \left( j-2\right) }{2}%
}z^{t_{1}+1+t_{2}-i} \\
&& \\
&&+\!\sum\limits_{i=2}^{t_{1}+t_{2}+1}\!\sum\limits_{j=2}^{i}\!\left(
\!-1\right) ^{\frac{\left( s\left( j-2\right) +2(s+1)\right) \left(
j-1\right) }{2}}\!\!\dbinom{t_{1}+t_{2}}{j-2}_{\!\!F_{s}\!\left( x,y\right)
}\!\!\! \\
&&\times \left( F_{m_{1}+\left( t_{1}-1\right) s\left( i-j\right) }\left(
x,y\right) F_{m_{2}+t_{2}s\left( i-j\right) }\left( x,y\right) \right) y^{%
\frac{s\left( j-2\right) \left( j-3\right) }{2}}\left( -y\right) ^{s\left(
i-2\right) }\left( -y\right) ^{s\left( t_{1}+t_{2}+1\right)
}z^{t_{1}+1+t_{2}-i} \\
&& \\
&=&\sum_{i=0}^{t_{1}+t_{2}+1}\sum_{j=0}^{i}\left( -1\right) ^{\frac{\left(
sj+2(s+1)\right) \left( j+1\right) }{2}}\dbinom{t_{1}+t_{2}+2}{j}%
_{F\!_{s}\!\left( x,y\right) }\frac{F_{m_{2}+t_{2}s\left( i-j\right) }\left(
x,y\right) F_{m_{1}+\left( t_{1}-1\right) s\left( i-j\right) }\left(
x,y\right) }{F_{s\left( t_{1}+t_{2}+2\right) }\left( x,y\right) F_{s\left(
t_{1}+t_{2}+1\right) }\left( x,y\right) } \\
&&\times \left( -y\right) ^{s\left( i-j\right) }\left( 
\begin{array}{c}
\left( -y\right) ^{sj}F_{s\left( t_{1}+t_{2}+1-j\right) }\left( x,y\right)
F_{s\left( t_{1}+t_{2}+2-j\right) }\left( x,y\right) \\ 
+F_{sj}\left( x,y\right) F_{s\left( t_{1}+t_{2}+2-j\right) }\left(
x,y\right) L_{s\left( t_{1}+t_{2}+1\right) }\left( x,y\right) \\ 
+\left( -y\right) ^{s\left( t_{1}+t_{2}+2-j\right) }F_{sj}\left( x,y\right)
F_{s\left( j-1\right) }\left( x,y\right)%
\end{array}%
\right) y^{\frac{sj\left( j-1\right) }{2}}z^{t_{1}+1+t_{2}-i}
\end{eqnarray*}

Lemma \ref{Lemma2.3} allows us to write%
\begin{eqnarray*}
B\left( x,y;z\right) &=&\!\sum_{i=0}^{t_{1}+t_{2}+1}\!\sum_{j=0}^{i}\!\left(
-1\right) ^{\frac{\left( sj+2(s+1)\right) \left( j+1\right) }{2}}\!\dbinom{%
t_{1}+t_{2}+2}{j}_{\!F\!_{s}\!\left( x,y\right) }\!\! \\
&&\times F_{m_{2}+t_{2}s\left( i-j\right) }\left( x,y\right) F_{m_{1}+\left(
t_{1}-1\right) s\left( i-j\right) }\left( x,y\right) \left( -y\right)
^{s\left( i-j\right) }y^{\frac{sj\left( j-1\right) }{2}}z^{t_{1}+1+t_{2}-i}.
\end{eqnarray*}

Thus, the numerator of (\ref{3.7}) is%
\begin{eqnarray}
&&z\left( A\left( x,y;z\right) -B\left( x,y;z\right) \right)  \notag \\
&=&z\sum_{i=0}^{t_{1}+t_{2}+1}\sum_{j=0}^{i}\left( -1\right) ^{\frac{\left(
sj+2(s+1)\right) \left( j+1\right) }{2}}\dbinom{t_{1}+t_{2}+2}{j}%
_{F\!_{s}\!\left( x,y\right) }  \notag \\
&&\times F_{m_{2}+t_{2}s\left( i-j\right) }\left( x,y\right)
F_{m_{1}+t_{1}s\left( i-j\right) }\left( x,y\right) L_{s\left( i-j\right)
}\left( x,y\right) y^{\frac{sj\left( j-1\right) }{2}}z^{t_{1}+1+t_{2}-i} 
\notag \\
&&-z\sum_{i=0}^{t_{1}+t_{2}+1}\!\sum_{j=0}^{i}\!\left( -1\right) ^{\frac{%
\left( sj+2(s+1)\right) \left( j+1\right) }{2}}\!\dbinom{t_{1}+t_{2}+2}{j}%
_{\!F\!_{s}\!\left( x,y\right) }\!\!  \notag \\
&&\times F_{m_{2}+t_{2}s\left( i-j\right) }\left( x,y\right) F_{m_{1}+\left(
t_{1}-1\right) s\left( i-j\right) }\left( x,y\right) \left( -y\right)
^{s\left( i-j\right) }y^{\frac{sj\left( j-1\right) }{2}}z^{t_{1}+1+t_{2}-i} 
\notag \\
&=&z\sum_{i=0}^{t_{1}+t_{2}+1}\sum_{j=0}^{i}\left( -1\right) ^{\frac{\left(
sj+2(s+1)\right) \left( j+1\right) }{2}}\dbinom{t_{1}+t_{2}+2}{j}%
_{F\!_{s}\!\left( x,y\right) }F_{m_{2}+t_{2}s\left( i-j\right) }\left(
x,y\right)  \notag \\
&&\times \left( F_{m_{1}+t_{1}s\left( i-j\right) }\left( x,y\right)
L_{s\left( i-j\right) }\left( x,y\right) -\left( -y\right) ^{s\left(
i-j\right) }F_{m_{1}+\left( t_{1}-1\right) s\left( i-j\right) }\left(
x,y\right) \right) y^{\frac{sj\left( j-1\right) }{2}}z^{t_{1}+1+t_{2}-i}.
\label{3.8}
\end{eqnarray}

Finally, by using (\ref{1.11}) with $N=m_{1}+t_{1}s\left( i-j\right) $, $%
M=2s\left( i-j\right) $ and $K=-s\left( i-j\right) $ we see that 
\begin{equation*}
F_{m_{1}+t_{1}s\left( i-j\right) }\left( x,y\right) L_{s\left( i-j\right)
}\left( x,y\right) -\left( -y\right) ^{s\left( i-j\right) }F_{m_{1}+\left(
t_{1}-1\right) s\left( i-j\right) }\left( x,y\right) =F_{m_{1}+\left(
t_{1}+1\right) s\left( i-j\right) }\left( x,y\right) ,
\end{equation*}%
and then (\ref{3.8}) becomes the numerator of (\ref{3.54}), as wanted.
\end{proof}

By using the same sort of arguments of the proofs of theorems \ref{Th3.1}
and \ref{Th3.2}, one can prove the natural generalization of these theorems,
namely%
\begin{eqnarray}
&&\mathcal{Z}\left( F_{t_{1}sn+m_{1}}^{k_{1}}\left( x,y\right) \cdots
F_{t_{l}sn+m_{l}}^{k_{l}}\left( x,y\right) \right)  \label{3.11} \\
&&  \notag \\
&=&z\frac{%
\begin{array}{c}
\sum\limits_{i=0}^{k_{1}t_{1}+\cdots
+k_{l}t_{l}}\!\sum\limits_{j=0}^{i}\left( -1\right) ^{\frac{\left(
sj+2(s+1)\right) \left( j+1\right) }{2}}\!\dbinom{k_{1}t_{1}+\cdots
+k_{l}t_{l}+1}{j}_{F\!_{s}\!\left( x,y\right) }\! \\ 
\\ 
\times F_{m_{1}+t_{1}s\left( i-j\right) }^{k_{1}}\!\left( x,y\right) \cdots
\!F_{m_{l}+t_{l}s\left( i-j\right) }^{k_{l}}\left( x,y\right) y^{\frac{%
sj\left( j-1\right) }{2}}z^{k_{1}t_{1}+\cdots +k_{l}t_{l}-i}%
\end{array}%
}{\sum\limits_{i=0}^{k_{1}t_{1}+\cdots +k_{l}t_{l}+1}\left( -1\right) ^{%
\frac{\left( si+2(s+1)\right) \left( i+1\right) }{2}}\dbinom{%
k_{1}t_{1}+\cdots +k_{l}t_{l}+1}{i}_{F\!_{s}\!\left( x,y\right) }\!y^{\frac{%
si\left( i-1\right) }{2}}z^{k_{1}t_{1}+\cdots +k_{l}t_{l}+1-i}}.  \notag
\end{eqnarray}

\section{\label{Sec4}Some corollaries}

In this section we will obtain some consequences of (\ref{3.11}).

\begin{corollary}
\label{Cor4.1}\textit{For }$p\in \mathbb{N}^{\prime }$\textit{\ given, the }$%
Z$\textit{\ transform of the sequence }$\binom{n}{p}_{F_{s}\left( x,y\right)
}$\textit{\ is}%
\begin{eqnarray}
\mathcal{Z}\left( \dbinom{n}{p}_{F_{s}\left( x,y\right) }\right) &=&\frac{%
\left( -1\right) ^{s+1}z}{D_{s,p+1}\left( x,y;z\right) }  \label{4.1} \\
&=&\frac{\left( -1\right) ^{s+1}z}{\sum_{i=0}^{p+1}\left( -1\right) ^{\frac{%
\left( si+2\left( s+1\right) \right) \left( i+1\right) }{2}}\dbinom{p+1}{i}%
_{F_{s}\left( x,y\right) }y^{\frac{si\left( i-1\right) }{2}}z^{p+1-i}}. 
\notag
\end{eqnarray}
\end{corollary}

\begin{proof}
We have%
\begin{eqnarray}
&&\mathcal{Z}\left( \dbinom{n}{p}_{F_{s}\left( x,y\right) }\right)  \notag \\
&=&\frac{1}{\left( F_{p}\left( x,y\right) !\right) _{s}}\mathcal{Z}\left(
F_{s\left( n-p+1\right) }\left( x,y\right) F_{s\left( n-p+2\right) }\left(
x,y\right) \cdots F_{sn}\left( x,y\right) \right)  \notag \\
&=&\frac{z}{\left( F_{p}\left( x,y\right) !\right)
_{s}\sum\limits_{i=0}^{p+1}\left( -1\right) ^{\frac{\left( si+2\left(
s+1\right) \right) \left( i+1\right) }{2}}\dbinom{p+1}{i}_{F_{s}\left(
x,y\right) }z^{p+1-i}}  \notag \\
&&\!\!\!\times \!\sum\limits_{i=0}^{p}\sum\limits_{j=0}^{i}\!\left(
\!-1\right) ^{\frac{\left( sj+2\left( s+1\right) \right) \left( j+1\right) }{%
2}}\!\dbinom{p+1}{j}_{\!\!F_{s}\left( x,y\right) }\!\!F_{s\left(
1-p+i-j\right) }\!\left( x,y\right) \!F_{s\left( 2-p+i-j\right) }\!\left(
x,y\right) \cdots F_{s\left( i-j\right) }\!\left( x,y\right) \!z^{p-i}.
\label{4.15}
\end{eqnarray}

But the product $F_{s\left( 1-p+i-j\right) }\left( x,y\right) F_{s\left(
2-p+i-j\right) }\left( x,y\right) \cdots F_{s\left( i-j\right) }\left(
x,y\right) $ is different from zero if and only if $i=p$ and $j=0$. In such
a case that product is $\left( F_{p}\left( x,y\right) !\right) _{s}$ and the
numerator of (\ref{4.15}) reduces to $\left( -1\right) ^{s+1}\left(
F_{p}\left( x,y\right) !\right) _{s}z$. Thus (\ref{4.1}) follows.
\end{proof}

In the rest of this section we will be using extensively (\ref{4.1}), most
of times in its shifted version: according to (\ref{2.1}), if $0\leq
p_{0}\leq p$, we have that%
\begin{equation}
\mathcal{Z}\left( \dbinom{n+p_{0}}{p}_{F_{s}\left( x,y\right) }\right)
=z^{p_{0}}\frac{\left( -1\right) ^{s+1}z}{\sum_{i=0}^{p+1}\left( -1\right) ^{%
\frac{\left( si+2(s+1)\right) \left( i+1\right) }{2}}\dbinom{p+1}{i}%
_{F_{s}\left( x,y\right) }y^{\frac{si\left( i-1\right) }{2}}z^{p+1-i}}.
\label{4.2}
\end{equation}

By using that%
\begin{equation*}
G_{sn+m}\left( x,y\right) =yG_{m}\left( x,y\right) F_{sn-1}\left( x,y\right)
+G_{m+1}\left( x,y\right) F_{sn}\left( x,y\right) ,
\end{equation*}%
together with a simple linearity argument, we can see that formula (\ref%
{3.11}) is valid for bivariate Gibonacci polynomials $G_{n}\left( x,y\right) 
$ replacing the bivariate Fibonacci polynomials $F_{n}\left( x,y\right) $.
That is, we have%
\begin{eqnarray}
&&\mathcal{Z}\left( G_{st_{1}n+m_{1}}^{k_{1}}\left( x,y\right) \cdots
G_{st_{l}n+m_{l}}^{k_{l}}\left( x,y\right) \right)  \label{4.3} \\
&&  \notag \\
&=&\frac{%
\begin{array}{c}
z\sum\limits_{i=0}^{t_{1}k_{1}+\cdots
+t_{l}k_{l}}\!\sum\limits_{j=0}^{i}\!\left( -1\right) ^{\frac{\left(
sj+2(s+1)\right) \left( j+1\right) }{2}}\!\!\dbinom{t_{1}k_{1}+\cdots
+t_{l}k_{l}+1}{j}_{F_{s}\left( x,y\right) }\!\! \\ 
\times G_{m_{1}+st_{1}\left( i-j\right) }^{k_{1}}\left( x,y\right) \cdots
G_{m_{l}+st_{l}\left( i-j\right) }^{k_{l}}\left( x,y\right) y^{\frac{%
sj\left( j-1\right) }{2}}z^{t_{1}k_{1}+\cdots +t_{l}k_{l}-i}%
\end{array}%
}{\sum\limits_{i=0}^{t_{1}k_{1}+\cdots +t_{l}k_{l}+1}\left( -1\right) ^{%
\frac{\left( si+2(s+1)\right) \left( i+1\right) }{2}}\dbinom{%
t_{1}k_{1}+\cdots +t_{l}k_{l}+1}{i}_{F_{s}\left( x,y\right) }y^{\frac{%
si\left( i-1\right) }{2}}z^{t_{1}k_{1}+\cdots +t_{l}k_{l}+1-i}}.  \notag
\end{eqnarray}

\begin{corollary}
\label{Cor4.2}Let $k_{1},\ldots ,k_{l},t_{1},\ldots ,t_{l}\in \mathbb{N}%
^{\prime }$ and $m_{1},\ldots ,m_{l}\in \mathbb{Z}$ be given. The product of
bivariate $s$-Gibonacci polynomials $G_{st_{1}n+m_{1}}^{k_{1}}\left(
x,y\right) \cdots G_{st_{l}n+m_{l}}^{k_{l}}\left( x,y\right) $ can be
written as a linear combination of bivariate $s$-Fibopolynomials according
to 
\begin{eqnarray}
&&G_{st_{1}n+m_{1}}^{k_{1}}\left( x,y\right) \cdots
G_{st_{l}n+m_{l}}^{k_{l}}\left( x,y\right)  \label{4.4} \\
&=&\!\left( -1\right) ^{s+1}\!\!\sum\limits_{i=0}^{t_{1}k_{1}+\cdots
+t_{l}k_{l}}\!\sum\limits_{j=0}^{i}\!\left( -1\right) ^{\frac{\left(
sj+2(s+1)\right) \left( j+1\right) }{2}}\!\!\dbinom{t_{1}k_{1}+\cdots
+t_{l}k_{l}+1}{j}_{F_{s}\left( x,y\right) }\!\!  \notag \\
&&\times G_{m_{1}+st_{1}\left( i-j\right) }^{k_{1}}\left( x,y\right) \cdots
G_{m_{l}+st_{l}\left( i-j\right) }^{k_{l}}\left( x,y\right) y^{\frac{%
sj\left( j-1\right) }{2}}\dbinom{n+t_{1}k_{1}+\cdots +t_{l}k_{l}-i}{%
t_{1}k_{1}+\cdots +t_{l}k_{l}}_{F_{s}\left( x,y\right) }.  \notag
\end{eqnarray}
\end{corollary}

\begin{proof}
Formula (\ref{4.4}) follows from (\ref{4.2}) and (\ref{4.3}).
\end{proof}

In particular, (\ref{4.4}) tells us that%
\begin{eqnarray}
G_{stn+m}^{k}\left( x,y\right) &=&\left( -1\right)
^{s+1}\sum_{i=0}^{tk}\sum_{j=0}^{i}\left( -1\right) ^{\frac{\left(
sj+2(s+1)\right) \left( j+1\right) }{2}}\!\!\dbinom{tk+1}{j}_{F_{s}\left(
x,y\right) }\!\!  \label{4.52} \\
&&\times G_{m+st\left( i-j\right) }^{k}\left( x,y\right) y^{\frac{sj\left(
j-1\right) }{2}}\dbinom{n+tk-i}{tk}_{F_{s}\left( x,y\right) }.  \notag
\end{eqnarray}

The following corollary shows that when $k=1$ and $G=F$ or $G=L$, formula (%
\ref{4.52}) can be written in a simpler form.

\begin{corollary}
\label{Cor4.3}\textit{Let }$t\in \mathbb{N}^{\prime }$\textit{\ be given.
The following identities hold}

(a)%
\begin{equation}
F_{tsn+m}\left( x,y\right) =\left( -y\right) ^{m}\sum\limits_{i=0}^{t}\left(
-1\right) ^{i+1}\dbinom{t}{i}_{F_{s}\left( x,y\right) }F_{is-m}\left(
x,y\right) \left( -y\right) ^{\frac{si\left( i-1\right) }{2}}\dbinom{n+t-i}{t%
}_{F_{s}\left( x,y\right) }.  \label{4.11}
\end{equation}

(b) 
\begin{equation}
L_{tsn+m}\left( x,y\right) =\left( -y\right) ^{m}\sum\limits_{i=0}^{t}\left(
-1\right) ^{i}\dbinom{t}{i}_{F_{s}\left( x,y\right) }L_{is-m}\left(
x,y\right) \left( -y\right) ^{\frac{si\left( i-1\right) }{2}}\dbinom{n+t-i}{t%
}_{F_{s}\left( x,y\right) }.  \label{4.11b}
\end{equation}
\end{corollary}

\begin{proof}
These results are direct consequences of (\ref{2.28}), (\ref{2.29}) and (\ref%
{4.52}).
\end{proof}

If we set $t=1$ in (\ref{4.11}) and (\ref{4.11b}) we get (\ref{1.11}) and (%
\ref{1.12}), respectively. When $t=2$ we obtain from (\ref{4.11}) and (\ref%
{4.11b}) the identities%
\begin{eqnarray}
F_{2sn+m}\left( x,y\right) &=&F_{m}\left( x,y\right) \dbinom{n+2}{2}%
_{F_{s}\left( x,y\right) }+\left( -y\right) ^{m}L_{s}\!\left( x,y\right)
F_{s-m}\left( x,y\right) \dbinom{n+1}{2}_{F_{s}\left( x,y\right) }
\label{4.111} \\
&&-F_{2s-m}\left( x,y\right) \left( -y\right) ^{m+s}\dbinom{n}{2}%
_{F_{s}\left( x,y\right) },  \notag
\end{eqnarray}%
and%
\begin{eqnarray}
L_{2sn+m}\left( x,y\right) &=&L_{m}\left( x,y\right) \dbinom{n+2}{2}%
_{F_{s}\left( x,y\right) }-\left( -y\right) ^{m}L_{s}\!\left( x,y\right)
L_{s-m}\left( x,y\right) \dbinom{n+1}{2}_{F_{s}\left( x,y\right) }
\label{4.112} \\
&&+L_{2s-m}\left( x,y\right) \left( -y\right) ^{s+m}\dbinom{n}{2}%
_{F_{s}\left( x,y\right) },  \notag
\end{eqnarray}%
respectively. Observe also that if we set $t=1$, $k=2$ and $G=F$ in (\ref%
{4.52}), we obtain

\begin{equation}
F_{sn}^{2}\left( x,y\right) =F_{s}^{2}\left( x,y\right) \left( \dbinom{n+1}{2%
}_{F_{s}\left( x,y\right) }+\left( -y\right) ^{s}\dbinom{n}{2}_{F_{s}\left(
x,y\right) }\right) ,  \label{4.1124}
\end{equation}%
which is (\ref{1.81}). Some other examples of (\ref{4.52}), besides (\ref%
{1.191}) and (\ref{4.1124}), are the following (after some simplifications
on the coefficients of the bivariate $s$-Fibopolynomials of the right-hand
sides, with the help of (\ref{1.6}), (\ref{1.7}) and/or (\ref{1.71}))%
\begin{eqnarray}
\frac{F_{2sn}^{2}\left( x,y\right) }{F_{2s}^{2}\left( x,y\right) } &=&%
\dbinom{n+3}{4}_{F_{s}\left( x,y\right) }+y^{6s}\dbinom{n}{4}_{F_{s}\left(
x,y\right) }  \label{4.1126} \\
&&+y^{s}\left( \left( -1\right) ^{s+1}L_{2s}\left( x,y\right) +y^{s}\right)
\left( \dbinom{n+2}{4}_{F_{s}\left( x,y\right) }+y^{2s}\dbinom{n+1}{4}%
_{F_{s}\left( x,y\right) }\right) .  \notag
\end{eqnarray}%
\begin{eqnarray}
L_{sn}^{2}\left( x,y\right) &=&4\dbinom{n+2}{2}_{F_{s}\left( x,y\right)
}-\left( 3L_{2s}\left( x,y\right) +2\left( -y\right) ^{s}\right) \dbinom{n+1%
}{2}_{F_{s}\left( x,y\right) }  \label{4.113} \\
&&+L_{s}^{2}\left( x,y\right) \left( -y\right) ^{s}\dbinom{n}{2}%
_{F_{s}\left( x,y\right) }.  \notag
\end{eqnarray}%
\begin{equation}
\frac{F_{sn}^{3}\left( x,y\right) }{F_{s}^{3}\left( x,y\right) }=\!\dbinom{%
n+2}{3}_{F_{s}\left( x,y\right) }+2\left( -y\right) ^{s}\!\!L_{s}\!\left(
x,y\right) \dbinom{n+1}{3}_{F_{s}\left( x,y\right) }+\left( -y\right)
^{3s}\!\!\dbinom{n}{3}_{F_{s}\left( x,y\right) }.  \label{4.114}
\end{equation}%
\begin{eqnarray}
\frac{F_{sn}\left( x,y\right) F_{2sn}\left( x,y\right) F_{3sn}\left(
x,y\right) }{F_{s}\left( x,y\right) F_{2s}\left( x,y\right) F_{3s}\left(
x,y\right) } &=&\dbinom{n+5}{6}_{F_{s}\left( x,y\right) }+\left( -1\right)
^{s+1}y^{15s}\dbinom{n}{6}_{F_{s}\left( x,y\right) }  \label{4.115} \\
&&+y^{3s}\left( \left( -1\right) ^{s}\dbinom{n+4}{6}_{F_{s}\left( x,y\right)
}-y^{9s}\dbinom{n+1}{6}_{F_{s}\left( x,y\right) }\right)  \notag \\
&&\!+y^{3s}L_{2s}\!\left( x,y\right) \!\left( L_{2s}\!\left( x,y\right)
-\left( -y\right) ^{s}\right) \!\left( L_{2s}\!\left( x,y\right) +2\left(
-y\right) ^{s}\right) \!  \notag \\
&&\times \left( \left( -1\right) ^{s+1}\!\dbinom{n+3}{6}_{F_{s}\left(
x,y\right) }\!+y^{3s}\dbinom{n+2}{6}_{F_{s}\left( x,y\right) }\right) . 
\notag
\end{eqnarray}

\begin{eqnarray}
&&\frac{L_{2sn}\left( x,y\right) F_{sn}\left( x,y\right) }{F_{s}\left(
x,y\right) }  \label{4.116} \\
&=&L_{2s}\left( x,y\right) \dbinom{n+2}{3}_{F_{s}\left( x,y\right)
}-2y^{2s}L_{s}\!\left( x,y\right) \dbinom{n+1}{3}_{F_{s}\left( x,y\right)
}+\left( -y\right) ^{3s}L_{2s}\left( x,y\right) \dbinom{n}{3}_{F_{s}\left(
x,y\right) }.  \notag
\end{eqnarray}

\begin{eqnarray}
L_{2sn}\left( x,y\right) L_{sn}\left( x,y\right) &=&4\dbinom{n+3}{3}%
_{F_{s}\left( x,y\right) }+2\left( -y\right) ^{s}L_{2s}\!\left( x,y\right)
\left( L_{2s}\!\left( x,y\right) +\left( -y\right) ^{s}\right) \dbinom{n+1}{3%
}_{F_{s}\left( x,y\right) }  \label{4.118} \\
&&-\left( L_{3s}\left( x,y\right) +\left( -y\right) ^{s}L_{s}\left(
x,y\right) \right) \left( 3\dbinom{n+2}{3}_{F_{s}\left( x,y\right) }+\left(
-y\right) ^{3s}\dbinom{n}{3}_{F_{s}\left( x,y\right) }\right) .  \notag
\end{eqnarray}%
\begin{eqnarray}
\frac{F_{2s\left( n+1\right) }\left( x,y\right) F_{s\left( n+2\right)
}^{2}\left( x,y\right) }{F_{2s}\left( x,y\right) F_{s}^{2}\left( x,y\right) }
&=&L_{s}^{2}\left( x,y\right) \dbinom{n+4}{4}_{F_{s}\left( x,y\right)
}+\left( -1\right) ^{s+1}L_{4s}\left( x,y\right) y^{s}\dbinom{n+3}{4}%
_{F_{s}\left( x,y\right) }  \label{4.119} \\
&&-L_{2s}\left( x,y\right) y^{4s}\dbinom{n+2}{4}_{F_{s}\left( x,y\right) }. 
\notag
\end{eqnarray}

In the following corollary we consider sequences involving bivariate $s$%
-Gibopolynomials\ $\binom{n}{p}_{G_{s}\left( x,y\right) }$.

\begin{corollary}
\label{Cor4.4}\textit{Let }$t_{1},\ldots ,t_{l}\in \mathbb{N}$ \textit{\ and
\ }$r_{1}\ldots ,r_{l},p_{1}\ldots ,p_{l}\in \mathbb{N}^{\prime }$\textit{\
be given. Then the }$Z$\textit{\ transform of the sequence }$\binom{n}{p_{1}}%
_{G_{st_{1}}\left( x,y\right) }^{r_{1}}\cdots \binom{n}{p_{l}}%
_{G_{st_{l}}\left( x,y\right) }^{r_{l}}$\textit{\ is given by}%
\begin{eqnarray}
&&\mathcal{Z}\left( \dbinom{n}{p_{1}}_{G_{st_{1}}\left( x,y\right)
}^{r_{1}}\cdots \dbinom{n}{p_{l}}_{G_{st_{l}}\left( x,y\right)
}^{r_{l}}\right)  \label{4.12} \\
&&  \notag \\
&=&z\frac{%
\begin{array}{c}
\dsum\limits_{i=0}^{t_{1}r_{1}p_{1}+\cdots
+t_{l}r_{l}p_{l}}\dsum\limits_{j=0}^{i}\left( -1\right) ^{\frac{\left(
sj+2(s+1)\right) \left( j+1\right) }{2}}\dbinom{t_{1}r_{1}p_{1}+\cdots
+t_{l}r_{l}p_{l}+1}{j}_{F_{s}\left( x,y\right) } \\ 
\\ 
\times \dbinom{i-j}{p_{1}}_{G_{st_{1}}\left( x,y\right) }^{r_{1}}\cdots 
\dbinom{i-j}{p_{l}}_{G_{st_{l}}\left( x,y\right) }^{r_{l}}y^{\frac{sj\left(
j-1\right) }{2}}z^{t_{1}r_{1}p_{1}+\cdots +t_{l}r_{l}p_{l}-i}%
\end{array}%
}{\dsum\limits_{i=0}^{t_{1}r_{1}p_{1}+\cdots +t_{l}r_{l}p_{l}+1}\left(
-1\right) ^{\frac{\left( si+2(s+1)\right) \left( i+1\right) }{2}}\dbinom{%
t_{1}r_{1}p_{1}+\cdots +t_{l}r_{l}p_{l}+1}{i}_{F_{s}\left( x,y\right) }y^{%
\frac{si\left( i-1\right) }{2}}z^{t_{1}r_{1}p_{1}+\cdots
+t_{l}r_{l}p_{l}+1-i}}.  \notag
\end{eqnarray}
\end{corollary}

\begin{proof}
First we write%
\begin{eqnarray*}
&&\mathcal{Z}\left( \dbinom{n}{p_{1}}_{G_{st_{1}}\left( x,y\right)
}^{r_{1}}\cdots \dbinom{n}{p_{k}}_{G_{st_{l}}\left( x,y\right)
}^{r_{l}}\right) \\
&=&\frac{1}{G_{st_{1}}^{r_{1}}\left( x,y\right) \cdots
G_{st_{1}p_{1}}^{r_{1}}\left( x,y\right) \cdots G_{st_{l}}^{r_{l}}\left(
x,y\right) \cdots G_{st_{l}p_{l}}^{r_{l}}\left( x,y\right) } \\
&&\times \mathcal{Z}\left( G_{st_{1}n}^{r_{1}}\left( x,y\right) \cdots
G_{st_{1}\left( n-p_{1}+1\right) }^{r_{1}}\left( x,y\right) \cdots
G_{st_{l}n}^{r_{l}}\left( x,y\right) \cdots G_{st_{l}\left( n-p_{l}+1\right)
}^{r_{l}}\left( x,y\right) \right) ,
\end{eqnarray*}%
and then we use (\ref{4.3}) to get%
\begin{eqnarray*}
&&\mathcal{Z}\left( \dbinom{n}{p_{1}}_{G_{st_{1}}\left( x,y\right)
}^{r_{1}}\cdots \dbinom{n}{p_{l}}_{G_{st_{l}}\left( x,y\right)
}^{r_{l}}\right) \\
&& \\
&=&\frac{1}{G_{st_{1}}^{r_{1}}\left( x,y\right) \cdots
G_{st_{1}p_{1}}^{r_{1}}\left( x,y\right) \cdots G_{st_{l}}^{r_{l}}\left(
x,y\right) \cdots G_{st_{l}p_{l}}^{r_{l}}\left( x,y\right) } \\
&& \\
&&\times \frac{%
\begin{array}{c}
z\dsum\limits_{i=0}^{t_{1}r_{1}p_{1}+\cdots
+t_{l}r_{l}p_{l}}\dsum\limits_{j=0}^{i}\left( -1\right) ^{\frac{\left(
sj+2(s+1)\right) \left( j+1\right) }{2}}\dbinom{t_{1}r_{1}p_{1}+\cdots
+t_{l}r_{l}p_{l}+1}{j}_{F_{s}\left( x,y\right) } \\ 
\\ 
\times G_{st_{1}\left( i-j\right) }^{r_{1}}\left( x,y\right) \cdots
G_{st_{1}\left( i-j-p_{1}+1\right) }^{r_{1}}\left( x,y\right) \cdots
G_{st_{k}\left( i-j\right) }^{r_{l}}\left( x,y\right) \cdots G_{st_{l}\left(
i-j-p_{l}+1\right) }^{r_{l}}\left( x,y\right) \\ 
\\ 
\times y^{\frac{sj\left( j-1\right) }{2}}z^{t_{1}r_{1}p_{1}+\cdots
+t_{l}r_{l}p_{l}-i}%
\end{array}%
}{\dsum\limits_{i=0}^{t_{1}r_{1}p_{1}+\cdots +t_{l}r_{l}p_{l}+1}\left(
-1\right) ^{\frac{\left( si+2(s+1)\right) \left( i+1\right) }{2}}\dbinom{%
t_{1}r_{1}p_{1}+\cdots +t_{l}r_{l}p_{l}+1}{i}_{F_{s}\left( x,y\right) }y^{%
\frac{si\left( i-1\right) }{2}}z^{t_{1}r_{1}p_{1}+\cdots
+t_{l}r_{l}p_{l}+1-i}},
\end{eqnarray*}%
which implies the desired formula (\ref{4.12}).
\end{proof}

\begin{corollary}
\label{Cor4.5}\textit{Let }$t_{1},\ldots ,t_{l}\in \mathbb{N}$ \textit{\ and
\ }$r_{1}\ldots ,r_{l},p_{1}\ldots ,p_{l}\in \mathbb{N}^{\prime }$\textit{\
be given. The sequence }$\prod\limits_{i=1}^{l}\binom{n}{p_{i}}%
_{G_{st_{i}}\left( x,y\right) }^{r_{i}}$\textit{\ can be expressed as a
linear combination of the bivariate }$s$-\textit{Fibopolynomials }$\binom{%
n+t_{1}r_{1}p_{1}+\cdots +t_{l}r_{l}p_{l}-i}{t_{1}r_{1}p_{1}+\cdots
+t_{l}r_{l}p_{l}}_{F_{s}\left( x,y\right) }$\textit{, }$i=0,1,\ldots
,t_{1}r_{1}p_{1}+\cdots +t_{l}r_{l}p_{l}$, \textit{according to}%
\begin{eqnarray}
\prod\limits_{i=1}^{l}\binom{n}{p_{i}}_{G_{st_{i}}\left( x,y\right)
}^{r_{i}} &=&\left( -1\right)
^{s+1}\sum\limits_{i=0}^{t_{1}r_{1}p_{1}+\cdots
+t_{l}r_{l}p_{l}}\sum\limits_{j=0}^{i}\left( -1\right) ^{\frac{\left(
sj+2(s+1)\right) \left( j+1\right) }{2}}\dbinom{t_{1}r_{1}p_{1}+\cdots
+t_{l}r_{l}p_{l}+1}{j}_{F_{s}\left( x,y\right) }  \notag \\
&&\!\!\!\times \!\dbinom{i-j}{p_{1}}_{\!G_{st_{1}}\!\left( x,y\right)
}^{r_{1}}\!\!\cdots \!\dbinom{i-j}{p_{l}}_{\!G_{st_{l}}\!\left( x,y\right)
}^{r_{l}}y^{\frac{sj\left( j-1\right) }{2}}\!\dbinom{n+t_{1}r_{1}p_{1}+%
\cdots +t_{l}r_{l}p_{l}-i}{t_{1}r_{1}p_{1}+\cdots +t_{l}r_{l}p_{l}}%
_{\!F_{s}\!\left( x,y\right) }.  \label{4.13}
\end{eqnarray}
\end{corollary}

\begin{proof}
This comes directly from (\ref{4.12}) and (\ref{4.2}).
\end{proof}

Some examples of (\ref{4.13}) are%
\begin{equation}
\binom{n}{2}_{F_{s}\left( x,y\right) }^{2}=\dbinom{n+2}{4}_{F_{s}\left(
x,y\right) }+\left( -y\right) ^{s}L_{s}^{2}\left( x,y\right) \dbinom{n+1}{4}%
_{F_{s}\left( x,y\right) }+y^{4s}\dbinom{n}{4}_{F_{s}\left( x,y\right) }.
\label{4.14}
\end{equation}%
\begin{eqnarray}
\binom{n}{3}_{F_{s}\left( x,y\right) }^{2} &=&\dbinom{n+3}{6}_{F_{s}\left(
x,y\right) }+\left( -1\right) ^{s}y^{9s}\dbinom{n}{6}_{F_{s}\left(
x,y\right) }  \label{4.16} \\
&&+y^{s}\left( L_{2s}\left( x,y\right) +(-y)^{s}\right) ^{2}\left( \left(
-1\right) ^{s}\dbinom{n+2}{6}_{F_{s}\left( x,y\right) }+y^{3s}\dbinom{n+1}{6}%
_{F_{s}\left( x,y\right) }\right) .  \notag
\end{eqnarray}%
\begin{eqnarray}
\binom{n}{2}_{F_{s}\left( x,y\right) }\binom{n}{3}_{F_{s}\left( x,y\right) }
&=&\left( L_{2s}\left( x,y\right) +\left( -y\right) ^{s}\right) \!\binom{n+2%
}{5}_{F_{s}\left( x,y\right) }\!  \label{4.161} \\
&&+y^{2s}\left( L_{3s}\left( x,y\right) +2\left( -y\right) ^{s}L_{s}\left(
x,y\right) \right) \!\binom{n+1}{5}_{F_{s}\left( x,y\right) }\!+y^{6s}\binom{%
n}{5}_{F_{s}\left( x,y\right) }.  \notag
\end{eqnarray}%
\begin{equation}
\binom{n}{2}_{F_{2s}\left( x,y\right) }=\dbinom{n+2}{4}_{F_{s}\left(
x,y\right) }-\left( -y\right) ^{s}L_{2s}\left( x,y\right) \dbinom{n+1}{4}%
_{F_{s}\left( x,y\right) }+y^{4s}\dbinom{n}{4}_{F_{s}\left( x,y\right) }.
\label{4.17}
\end{equation}%
\begin{eqnarray}
\binom{n}{3}_{F_{2s}\left( x,y\right) } &=&\dbinom{n+3}{6}_{F_{s}\left(
x,y\right) }+\left( -1\right) ^{s+1}y^{9s}\dbinom{n}{6}_{F_{s}\left(
x,y\right) }  \label{4.18} \\
&&+y^{s}\left( y^{2s}+L_{4s}\left( x,y\right) \right) \left( \left(
-1\right) ^{s+1}\dbinom{n+2}{6}_{F_{s}\left( x,y\right) }+y^{3s}\dbinom{n+1}{%
6}_{F_{s}\left( x,y\right) }\right) .  \notag
\end{eqnarray}%
\begin{eqnarray}
\binom{n}{3}_{L_{s}\!\left( x,y\right) } &=&\frac{2\left( -1\right) ^{s}}{%
y^{3s}L_{3s}\left( x,y\right) }\dbinom{n+3}{3}_{F_{s}\left( x,y\right) }+%
\frac{2\left( -1\right) ^{s+1}\left( L_{2s}\left( x,y\right) +\left(
-y\right) ^{s}\right) }{y^{3s}L_{2s}\left( x,y\right) }\dbinom{n+2}{3}%
_{F_{s}\left( x,y\right) }  \label{4.19} \\
&&+\frac{2\left( L_{2s}\left( x,y\right) +\left( -y\right) ^{s}\right) }{%
y^{2s}L_{s}\!\left( x,y\right) }\dbinom{n+1}{3}_{F_{s}\left( x,y\right) }-%
\dbinom{n}{3}_{F_{s}\left( x,y\right) }.  \notag
\end{eqnarray}%
\begin{eqnarray}
\binom{n}{2}_{L_{2s}\!\left( x,y\right) }\binom{n}{2}_{F_{s}\!\left(
x,y\right) } &=&\dbinom{n+4}{6}_{F_{s}\left( x,y\right) }+y^{12s}\dbinom{n}{6%
}_{F_{s}\left( x,y\right) }  \label{4.191} \\
&&-y^{2s}L_{s}^{2}\left( x,y\right) \left( \dbinom{n+3}{6}_{F_{s}\left(
x,y\right) }+y^{6s}\dbinom{n+1}{6}_{F_{s}\left( x,y\right) }\right)  \notag
\\
&&+\frac{\left( -y\right) ^{3s}}{L_{4s}\left( x,y\right) }\left(
L_{10s}\left( x,y\right) +3y^{4s}L_{2s}\left( x,y\right) +4\left( -y\right)
^{5s}\right) \dbinom{n+2}{6}_{F_{s}\left( x,y\right) }.  \notag
\end{eqnarray}

Note that from (\ref{4.14}) and (\ref{4.17}) we can see at once that 
\begin{equation}
\binom{n}{2}_{F_{s}\left( x,y\right) }^{2}-\binom{n}{2}_{F_{2s}\left(
x,y\right) }=2\left( -y\right) ^{s}\frac{F_{3s}\left( x,y\right) }{%
F_{s}\left( x,y\right) }\dbinom{n+1}{4}_{F_{s}\left( x,y\right) }.
\label{4.192}
\end{equation}

After some simplifications we can write (\ref{4.192}) as 
\begin{equation*}
L_{2s}\left( x,y\right) F_{sn}\left( x,y\right) F_{s\left( n-1\right)
}\left( x,y\right) -F_{s}^{2}\left( x,y\right) L_{sn}\left( x,y\right)
L_{s\left( n-1\right) }\left( x,y\right) =2\left( -y\right) ^{s}F_{s\left(
n+1\right) }\left( x,y\right) F_{s\left( n-2\right) }\left( x,y\right) ,
\end{equation*}%
(which resembles (\ref{1.11})).

\begin{corollary}
\label{Cor4.6}\textit{(a) Let }$m_{1},\ldots ,m_{l}\in \mathbb{Z}$\textit{\
and }$t_{1},\ldots ,t_{l},k_{1}\ldots ,k_{l}\in \mathbb{N}^{\prime }$\textit{%
\ be given. For }$n\geq t_{1}k_{1}+\cdots +t_{l}k_{l}+1$\textit{\ we have
that}%
\begin{equation}
\sum_{j=0}^{t_{1}k_{1}+\cdots +t_{l}k_{l}+1}\left( -1\right) ^{\frac{\left(
sj+2(s+1)\right) \left( j+1\right) }{2}}\dbinom{t_{1}k_{1}+\cdots
+t_{l}k_{l}+1}{j}_{F_{s}\left( x,y\right) }y^{\frac{sj\left( j-1\right) }{2}%
}\prod\limits_{i=1}^{l}G_{m_{i}+st_{i}\left( n-j\right) }^{k_{i}}\left(
x,y\right) =0.  \label{4.20}
\end{equation}

\textit{(b) Let }$t_{1},\ldots ,t_{l}\in \mathbb{N}$ \textit{and }$%
r_{1}\ldots ,r_{l},p_{1}\ldots ,p_{l}\in \mathbb{N}^{\prime }$\textit{\ be
given. For }$n\geq t_{1}r_{1}p_{1}+\cdots +t_{k}r_{k}p_{k}+1$\textit{\ we
have that}%
\begin{equation}
\sum_{j=0}^{t_{1}r_{1}p_{1}+\cdots +t_{l}r_{l}p_{l}+1}\left( -1\right) ^{%
\frac{\left( sj+2(s+1)\right) \left( j+1\right) }{2}}\dbinom{%
t_{1}r_{1}p_{1}+\cdots +t_{l}r_{l}p_{l}+1}{j}_{F_{s}\left( x,y\right) }y^{%
\frac{sj\left( j-1\right) }{2}}\prod\limits_{i=1}^{l}\dbinom{n-j}{p_{i}}%
_{G_{st_{i}}\left( x,y\right) }^{r_{i}}=0.  \label{4.21}
\end{equation}
\end{corollary}

\begin{proof}
These results are consequences of (the numerators in) formulas (\ref{4.3})
and (\ref{4.12}).
\end{proof}

\begin{corollary}
\label{Cor4.8}\textit{Let} $p\in \mathbb{N}^{\prime }$\textit{\ be given.
The following identities hold}

\textit{(a)}%
\begin{equation}
\dbinom{n+1}{p+2}_{F_{s}\left( x,y\right) }=\frac{1}{F_{s\left( p+2\right)
}\left( x,y\right) }F_{s\left( p+2\right) n}\left( x,y\right) \ast \left(
-y\right) ^{s\left( n-p\right) }\dbinom{n}{p}_{F_{s}\left( x,y\right) }.
\label{4.23}
\end{equation}

\textit{(b)}%
\begin{eqnarray}
&&\dbinom{n+2}{p+4}_{F_{s}\left( x,y\right) }  \label{4.24} \\
&=&\frac{1}{F_{s\left( p+4\right) }\left( x,y\right) F_{s\left( p+2\right)
}\left( x,y\right) }F_{s\left( p+4\right) n}\left( x,y\right) \ast \left(
-y\right) ^{s\left( n-1\right) }F_{s\left( p+2\right) n}\left( x,y\right)
\ast y^{2s\left( n-p\right) }\dbinom{n}{p}_{F_{s}\left( x,y\right) }.  \notag
\end{eqnarray}
\end{corollary}

\begin{proof}
(a) First observe that%
\begin{eqnarray*}
D_{s,p+3}\left( x,y;z\right) &=&\dprod\limits_{j=0}^{p+2}\left( z-\alpha
^{sj}\left( x,y\right) \beta ^{s\left( p+2-j\right) }\left( x,y\right)
\right) \\
&=&\dprod\limits_{j=-1}^{p+1}\left( z-\left( -y\right) ^{s}\alpha
^{sj}\left( x,y\right) \beta ^{s\left( p-j\right) }\left( x,y\right) \right)
\\
&=&\left( z-\left( -y\right) ^{s}\alpha ^{-s}\left( x,y\right) \beta
^{s\left( p+1\right) }\left( x,y\right) \right) \left( z-\left( -y\right)
^{s}\alpha ^{s\left( p+1\right) }\left( x,y\right) \beta ^{-s}\left(
x,y\right) \right) \\
&&\times \dprod\limits_{j=0}^{p}\left( -y\right) ^{s}\left( \left( -y\right)
^{-s}z-\alpha ^{sj}\left( x,y\right) \beta ^{s\left( p-j\right) }\left(
x,y\right) \right) \\
&=&\left( -y\right) ^{s\left( p+1\right) }\left( z^{2}-L_{s\left( p+2\right)
}\left( x,y\right) z+\left( -y\right) ^{s\left( p+2\right) }\right)
D_{s,p+1}\left( x,y;\left( -y\right) ^{-s}z\right) .
\end{eqnarray*}

Then%
\begin{eqnarray*}
&&\mathcal{Z}\left( \dbinom{n+1}{p+2}_{F_{s}\left( x,y\right) }\right) \\
&=&\frac{\left( -1\right) ^{s+1}z^{2}}{D_{s,p+3}\left( x,y;z\right) } \\
&=&\frac{\left( -1\right) ^{s+1}z^{2}}{\left( -y\right) ^{s\left( p+1\right)
}\left( z^{2}-L_{s\left( p+2\right) }\left( x,y\right) z+\left( -y\right)
^{s\left( p+2\right) }\right) D_{s,p+1}\left( x,y;\left( -y\right)
^{-s}z\right) } \\
&=&\left( -y\right) ^{s}\frac{1}{\left( -y\right) ^{s\left( p+1\right)
}F_{s\left( p+2\right) }\left( x,y\right) }\frac{F_{s\left( p+2\right)
}\left( x,y\right) z}{z^{2}-L_{s\left( p+2\right) }\left( x,y\right)
z+\left( -y\right) ^{s\left( p+2\right) }}\frac{\left( -1\right)
^{s+1}\left( -y\right) ^{-s}z}{D_{s,p+1}\left( x,y;\left( -y\right)
^{-s}z\right) },
\end{eqnarray*}%
from where (according to (\ref{2.8}) and convolution theorem)%
\begin{eqnarray*}
\dbinom{n+1}{p+2}_{F_{s}\left( x,y\right) } &=&\left( -y\right) ^{-sp}\frac{1%
}{F_{s\left( p+2\right) }\left( x,y\right) }F_{s\left( p+2\right) n}\left(
x,y\right) \ast \left( -y\right) ^{sn}\dbinom{n}{p}_{F_{s}\left( x,y\right) }
\\
&=&\frac{1}{F_{s\left( p+2\right) }\left( x,y\right) }F_{s\left( p+2\right)
n}\left( x,y\right) \ast \left( -y\right) ^{s\left( n-p\right) }\dbinom{n}{p}%
_{F_{s}\left( x,y\right) },
\end{eqnarray*}%
as wanted.

(b) Let us consider the polynomial $D_{s,p+5}\left( x,y;z\right) $ and
observe that%
\begin{eqnarray*}
D_{s,p+5}\left( x,y;z\right) &=&\dprod\limits_{j=0}^{p+4}\left( z-\alpha
^{sj}\left( x,y\right) \beta ^{s\left( p+4-j\right) }\left( x,y\right)
\right) \\
&=&\dprod\limits_{j=-2}^{p+2}\left( z-y^{2s}\alpha ^{sj}\left( x,y\right)
\beta ^{s\left( p-j\right) }\left( x,y\right) \right) \\
&=&\left( z-y^{2s}\alpha ^{-2s}\left( x,y\right) \beta ^{s\left( p+2\right)
}\left( x,y\right) \right) \left( z-y^{2s}\alpha ^{s\left( p+2\right)
}\left( x,y\right) \beta ^{-2s}\left( x,y\right) \right) \\
&&\times \left( z-y^{2s}\alpha ^{-s}\left( x,y\right) \beta ^{s\left(
p+1\right) }\left( x,y\right) \right) \left( z-y^{2s}\alpha ^{s\left(
p+1\right) }\left( x,y\right) \beta ^{-s}\left( x,y\right) \right) \\
&&\times \dprod\limits_{j=0}^{p}y^{2s}\left( y^{-2s}z-\alpha ^{sj}\left(
x,y\right) \beta ^{s\left( p-j\right) }\left( x,y\right) \right) \\
&=&y^{2s\left( p+1\right) }\left( z^{2}-L_{s\left( p+4\right) }\left(
x,y\right) z+\left( -y\right) ^{s\left( p+4\right) }\right) \\
&&\times \left( z^{2}-\left( -y\right) ^{s}L_{s\left( p+2\right) }\left(
x,y\right) z+y^{2s}\left( -y\right) ^{s\left( p+2\right) }\right)
D_{s,p+1}\left( x,y;y^{-2s}z\right) \\
&=&y^{2s\left( p+2\right) }\left( z^{2}-L_{s\left( p+4\right) }\left(
x,y\right) z+\left( -y\right) ^{s\left( p+4\right) }\right) \\
&&\times \left( \left( \left( -y\right) ^{-s}z\right) ^{2}-L_{s\left(
p+2\right) }\left( x,y\right) \left( \left( -y\right) ^{-s}z\right) +\left(
-y\right) ^{s\left( p+2\right) }\right) D_{s,p+1}\left( x,y;y^{-2s}z\right) .
\end{eqnarray*}

Then%
\begin{eqnarray*}
&&\mathcal{Z}\left( \dbinom{n+2}{p+4}_{F_{s}\left( x,y\right) }\right) \\
&=&\frac{\left( -1\right) ^{s+1}z^{3}}{D_{s,p+5}\left( x,y;z\right) } \\
&=&\frac{1}{y^{2s\left( p+2\right) }}\frac{1}{F_{s\left( p+4\right) }\left(
x,y\right) }\frac{F_{s\left( p+4\right) }\left( x,y\right) z}{%
z^{2}-L_{s\left( p+4\right) }\left( x,y\right) z+\left( -y\right) ^{s\left(
p+4\right) }} \\
&&\times \frac{\left( -y\right) ^{s}}{F_{s\left( p+2\right) }\left(
x,y\right) }\frac{F_{s\left( p+2\right) }\left( x,y\right) \left( -y\right)
^{-s}z}{\left( \left( -y\right) ^{-s}z\right) ^{2}-L_{s\left( p+2\right)
}\left( x,y\right) \left( -y\right) ^{-s}z+\left( -y\right) ^{s\left(
p+2\right) }}\frac{y^{2s}\left( -1\right) ^{s+1}y^{-2s}z}{D_{s,p+1}\left(
x,y;y^{-2s}z\right) },
\end{eqnarray*}%
from where (according to (\ref{2.8}) and convolution theorem)%
\begin{eqnarray*}
&&\dbinom{n+2}{p+4}_{F_{s}\left( x,y\right) } \\
&=&\frac{1}{y^{2sp}\left( -y\right) ^{s}}\frac{1}{F_{s\left( p+4\right)
}\left( x,y\right) }F_{s\left( p+4\right) n}\left( x,y\right) \ast \frac{1}{%
F_{s\left( p+2\right) }\left( x,y\right) }\left( -y\right) ^{sn}F_{s\left(
p+2\right) n}\left( x,y\right) \ast y^{2sn}\dbinom{n}{p}_{F_{s}\left(
x,y\right) } \\
&=&\frac{1}{F_{s\left( p+4\right) }\left( x,y\right) F_{s\left( p+2\right)
}\left( x,y\right) }F_{s\left( p+4\right) n}\left( x,y\right) \ast \left(
-y\right) ^{s\left( n-1\right) }F_{s\left( p+2\right) n}\left( x,y\right)
\ast y^{2s\left( n-p\right) }\dbinom{n}{p}_{F_{s}\left( x,y\right) },
\end{eqnarray*}%
as wanted.
\end{proof}

Some examples of (\ref{4.23}) and (\ref{4.24}) are%
\begin{equation}
\dbinom{n+1}{3}_{F_{s}\left( x,y\right) }=\frac{1}{F_{3s}\left( x,y\right)
F_{s}\left( x,y\right) }\sum_{t=0}^{n}\left( -y\right) ^{s\left( t-1\right)
}F_{3s\left( n-t\right) }\left( x,y\right) F_{st}\left( x,y\right) .
\label{4.25}
\end{equation}%
\begin{equation}
\dbinom{n+2}{4}_{F_{s}\left( x,y\right) }=\frac{1}{F_{4s}\left( x,y\right)
F_{2s}\left( x,y\right) }\sum_{i=0}^{n}\sum_{j=0}^{i}\left( -y\right)
^{s\left( 2n-j-i-1\right) }F_{4sj}\left( x,y\right) F_{2s\left( i-j\right)
}\left( x,y\right) .  \label{4.262}
\end{equation}%
\begin{equation}
\dbinom{n+1}{5}_{F_{s}\left( x,y\right) }=\frac{1}{F_{5s}\left( x,y\right) }%
\sum_{t=0}^{n}\left( -y\right) ^{s\left( t-3\right) }F_{5s\left( n-t\right)
}\left( x,y\right) \dbinom{t}{3}_{F_{s}\left( x,y\right) }.  \label{4.261}
\end{equation}%
\begin{eqnarray}
\dbinom{n+2}{5}_{F_{s}\left( x,y\right) } &=&\frac{1}{F_{5s}\left(
x,y\right) F_{3s}\left( x,y\right) F_{s}\left( x,y\right) }  \label{4.30} \\
&&\times \sum_{i=0}^{n}\sum_{j=0}^{i}\left( -y\right) ^{s\left( j+i-3\right)
}F_{5s\left( n-i\right) }\left( x,y\right) F_{3s\left( i-j\right) }\left(
x,y\right) F_{sj}\left( x,y\right) .  \notag
\end{eqnarray}%
\begin{eqnarray}
\dbinom{n+2}{9}_{F_{s}\left( x,y\right) } &=&\frac{1}{F_{9s}\left(
x,y\right) F_{7s}\left( x,y\right) }  \label{4.31} \\
&&\times \sum_{t=0}^{n}\sum_{j=0}^{t}\left( -y\right) ^{s\left(
2n-t-j-11\right) }F_{9sj}\left( x,y\right) F_{7s\left( t-j\right) }\left(
x,y\right) \dbinom{n-t}{5}_{F_{s}\left( x,y\right) }.  \notag
\end{eqnarray}

In the following corollary we denote as $\ast _{j=0}^{k}\left( a_{n}\right)
_{j}$, the convolution $\left( a_{n}\right) _{0}\ast \left( a_{n}\right)
_{1}\ast \cdots \ast \left( a_{n}\right) _{k}$ (of $k+1$ given sequences).

\begin{corollary}
\label{Cor4.9}\textit{Let} $p\in \mathbb{N}$\textit{\ be given. The
following identities hold}

\textit{(a)} 
\begin{equation}
\dbinom{n+p}{2p}_{F_{s}\left( x,y\right) }=\left( -y\right) ^{spn}\ast
_{j=0}^{p-1}\frac{\left( -y\right) ^{sj\left( n-1\right) }}{F_{2s\left(
p-j\right) }\left( x,y\right) }F_{2s\left( p-j\right) n}\left( x,y\right) .
\label{4.27}
\end{equation}

\textit{(b)}%
\begin{equation}
\dbinom{n+p-1}{2p-1}_{F_{s}\left( x,y\right) }=\ast _{j=0}^{p-1}\frac{\left(
-y\right) ^{sj\left( n-1\right) }}{F_{s\left( 2p-1-2j\right) }\left(
x,y\right) }F_{s\left( 2p-1-2j\right) n}\left( x,y\right) .  \label{4.28}
\end{equation}
\end{corollary}

\begin{proof}
(a) According to (\ref{4.1}) and (\ref{2.11}) we have that%
\begin{equation*}
z^{p}\mathcal{Z}\left( \dbinom{n}{2p}_{F_{s}\left( x,y\right) }\right) =%
\frac{z^{p+1}}{\left( z-\left( -y\right) ^{sp}\right)
\dprod\limits_{j=0}^{p-1}\left( z^{2}-\left( -y\right) ^{sj}L_{2s\left(
p-j\right) }\left( x,y\right) z+y^{2ps}\right) },
\end{equation*}%
or (by using (\ref{4.2}))%
\begin{eqnarray*}
\mathcal{Z}\left( \dbinom{n+p}{2p}_{F_{s}\left( x,y\right) }\right) &=&\frac{%
z}{z-\left( -y\right) ^{sp}}\dprod\limits_{j=0}^{p-1}\frac{z}{z^{2}-\left(
-y\right) ^{sj}L_{2s\left( p-j\right) }\left( x,y\right) z+y^{2ps}} \\
&=&\frac{z}{z-\left( -y\right) ^{sp}}\dprod\limits_{j=0}^{p-1}\frac{\left(
-y\right) ^{sj}}{\left( -y\right) ^{2sj}}\frac{\left( -y\right) ^{-sj}z}{%
\left( -y\right) ^{-2sj}z^{2}-\left( -y\right) ^{-sj}L_{2s\left( p-j\right)
}\left( x,y\right) z+y^{2s\left( p-j\right) }} \\
&=&\frac{z}{z-\left( -y\right) ^{sp}}\dprod\limits_{j=0}^{p-1}\left(
-y\right) ^{-sj}\frac{\frac{z}{\left( -y\right) ^{sj}}}{\left( \frac{z}{%
\left( -y\right) ^{sj}}\right) ^{2}-\left( -y\right) ^{-sj}L_{2s\left(
p-j\right) }\left( x,y\right) \frac{z}{\left( -y\right) ^{sj}}+y^{2s\left(
p-j\right) }}.
\end{eqnarray*}

Then we have%
\begin{eqnarray*}
\dbinom{n+p}{2p}_{F_{s}\left( x,y\right) } &=&\left( -y\right) ^{spn}\ast
_{j=0}^{p-1}\frac{\left( -y\right) ^{-sj}\left( -y\right) ^{sjn}}{%
F_{2s\left( p-j\right) }\left( x,y\right) }F_{2s\left( p-j\right) n}\left(
x,y\right) \\
&=&\left( -y\right) ^{spn}\ast _{j=0}^{p-1}\frac{\left( -y\right) ^{sj\left(
n-1\right) }}{F_{2s\left( p-j\right) }\left( x,y\right) }F_{2s\left(
p-j\right) n}\left( x,y\right) ,
\end{eqnarray*}%
as wanted.

(b) According to (\ref{4.1}) and (\ref{2.12}) we have that%
\begin{eqnarray*}
z^{p-1}\mathcal{Z}\left( \dbinom{n}{2p-1}_{F_{s}\left( x,y\right) }\right)
&=&\frac{z^{p}}{\dprod\limits_{j=0}^{p-1}\left( z^{2}-\left( -y\right)
^{sj}L_{s\left( 2p-1-2j\right) }\left( x,y\right) z+\left( -y\right)
^{\left( 2p-1\right) s}\right) } \\
&=&\dprod\limits_{j=0}^{p-1}\frac{z}{z^{2}-\left( -y\right) ^{sj}L_{s\left(
2p-1-2j\right) }\left( x,y\right) z+\left( -y\right) ^{\left( 2p-1\right) s}}
\\
&=&\dprod\limits_{j=0}^{p-1}\left( -y\right) ^{-sj}\frac{\frac{z}{\left(
-y\right) ^{sj}}}{\left( \frac{z}{\left( -y\right) ^{sj}}\right)
^{2}-L_{s\left( 2p-1-2j\right) }\left( x,y\right) \frac{z}{\left( -y\right)
^{sj}}+\left( -y\right) ^{\left( 2p-1-2j\right) s}},
\end{eqnarray*}%
from where (\ref{4.28}) follows.
\end{proof}

Observe that the case $p=0$ of (\ref{4.23}) and (\ref{4.24}) corresponds to
the cases $p=1$ and $p=2$ of (\ref{4.27}), respectively. Also, the case $p=1$
of (\ref{4.23}) and (\ref{4.24}) corresponds to the cases $p=2$ and $p=3$ of
(\ref{4.28}), respectively. Two additional examples of (\ref{4.27}) and (\ref%
{4.28}) are%
\begin{eqnarray}
\dbinom{n+3}{6}_{F_{s}\left( x,y\right) } &=&\frac{1}{F_{6s}\left(
x,y\right) F_{4s}\left( x,y\right) F_{2s}\left( x,y\right) }  \label{4.29} \\
&&\times \sum_{i=0}^{n}\sum_{j=0}^{i}\sum_{t=0}^{j}\left( -y\right)
^{s\left( j+t+3n-3i-3\right) }F_{6s\left( i-j\right) }\left( x,y\right)
F_{4s\left( j-t\right) }\left( x,y\right) F_{2st}\left( x,y\right) .  \notag
\end{eqnarray}%
\begin{eqnarray}
\dbinom{n+3}{7}_{F_{s}\left( x,y\right) } &=&\frac{1}{F_{7s}\left(
x,y\right) F_{5s}\left( x,y\right) F_{3s}\left( x,y\right) F_{s}\left(
x,y\right) }  \label{4.32} \\
&&\times \sum_{i=0}^{n}\sum_{j=0}^{i}\sum_{t=0}^{j}\left( -y\right)
^{s\left( i+j+t-6\right) }F_{7s\left( n-i\right) }\left( x,y\right)
F_{5s\left( i-j\right) }\left( x,y\right) F_{3s\left( j-t\right) }\left(
x,y\right) F_{st}\left( x,y\right) .  \notag
\end{eqnarray}

In the last corollary of this section we will see that some bivariate $s$%
-Fibopolynomials can be decomposed as linear combinations of certain
bivariate $s$-Fibonacci polynomials. This is shown by having an adequate
partial fractions decompositions of the $Z$ transform of the corresponding
bivariate $s$-Fibopolynomials. This decomposition is presented in lemma \ref%
{Lem4.1}.

We introduce the notation (for given $p\in \mathbb{N}$ and $j=0,1,\ldots
,p-1 $)%
\begin{equation}
\mathcal{P}_{2s\left( p-j\right) }\left( x,y\right) =\dprod_{i=0,i\neq
j}^{p-1}\left( \left( -y\right) ^{sj}L_{2s\left( p-j\right) }\left(
x,y\right) -\left( -y\right) ^{si}L_{2s\left( p-i\right) }\left( x,y\right)
\right) .  \label{4.33}
\end{equation}%
\begin{equation}
\mathcal{R}_{s\left( 2p-1-2j\right) }\left( x,y\right) =\dprod_{i=0,i\neq
j}^{p-1}\left( \left( -y\right) ^{sj}L_{s\left( 2p-1-2j\right) }\left(
x,y\right) -\left( -y\right) ^{si}L_{s\left( 2p-1-2i\right) }\left(
x,y\right) \right) .  \label{4.34}
\end{equation}

\begin{lemma}
\label{Lem4.1}Let $k\in \mathbb{N}$ be given. For $p\geq k$, we have the
following partial fractions decompositions

(a) 
\begin{eqnarray}
&&\frac{z^{p-k}}{\dprod\limits_{j=0}^{p-1}\left( z^{2}-\left( -y\right)
^{sj}L_{2s\left( p-j\right) }\left( x,y\right) z+y^{2ps}\right) }
\label{4.35} \\
&=&\sum_{j=0}^{p-1}\frac{\left( -y\right) ^{sj\left( k-2\right) -2sp\left(
k-1\right) }}{\mathcal{P}_{2s\left( p-j\right) }\left( x,y\right)
F_{2s\left( p-j\right) }\left( x,y\right) }\frac{-F_{2s\left( p-j\right)
\left( k-1\right) }\left( x,y\right) z+\left( -y\right) ^{sj}F_{2s\left(
p-j\right) k}\left( x,y\right) }{z^{2}-\left( -y\right) ^{sj}L_{2s\left(
p-j\right) }\!\left( x,y\right) z+y^{2ps}},  \notag
\end{eqnarray}

(b)%
\begin{eqnarray}
&&\frac{z^{p-k}}{\dprod\limits_{j=0}^{p-1}\left( z^{2}-\left( -y\right)
^{sj}L_{s\left( 2p-1-2j\right) }\left( x,y\right) z+\left( -y\right)
^{\left( 2p-1\right) s}\right) }  \label{4.36} \\
&=&\sum_{j=0}^{p-1}\frac{\left( -y\right) ^{sj\left( k-2\right) -s\left(
2p-1\right) \left( k-1\right) }}{\mathcal{R}_{s\left( 2p-1-2j\right) }\left(
x,y\right) F_{s\left( 2p-1-2j\right) }\left( x,y\right) }\frac{-F_{s\left(
2p-1-2j\right) \left( k-1\right) }\left( x,y\right) z+\left( -y\right)
^{sj}F_{s\left( 2p-1-2j\right) k}\left( x,y\right) }{z^{2}-\left( -y\right)
^{sj}L_{s\left( 2p-1-2j\right) }\!\left( x,y\right) z+\left( -y\right)
^{s\left( 2p-1\right) }}.  \notag
\end{eqnarray}
\end{lemma}

\begin{proof}
Let us prove (\ref{4.35}). For each $j=0,1,\ldots ,p-1$ we have%
\begin{eqnarray}
&&\frac{z^{p-k}}{\dprod\limits_{j=0}^{p-1}\left( z^{2}-\left( -y\right)
^{sj}L_{2s\left( p-j\right) }\left( x,y\right) z+y^{2ps}\right) }
\label{4.37} \\
&=&\frac{z^{p-k}}{\left( z-\left( -y\right) ^{sj}\alpha ^{2s\left(
p-j\right) }\left( x,y\right) \right) \left( z-\left( -y\right) ^{sj}\beta
^{2s\left( p-j\right) }\left( x,y\right) \right) \dprod\limits_{\substack{ %
i=0  \\ i\neq j}}^{p-1}\left( z^{2}-\left( -y\right) ^{si}L_{2s\left(
p-i\right) }\left( x,y\right) z+y^{2ps}\right) }  \notag
\end{eqnarray}

Thus the partial fractions decomposition of (\ref{4.37}) is%
\begin{eqnarray*}
&&\frac{z^{p-k}}{\dprod\limits_{j=0}^{p-1}\left( z^{2}-\left( -y\right)
^{sj}L_{2s\left( p-j\right) }\left( x,y\right) z+y^{2ps}\right) } \\
&=&\sum_{j=0}^{p-1}\frac{1}{\dprod\limits_{i=0.i\neq j}^{p-1}\left( \left(
-y\right) ^{2sj}\alpha ^{4s\left( p-j\right) }\left( x,y\right) -\left(
-y\right) ^{si}L_{2s\left( p-i\right) }\left( x,y\right) \left( -y\right)
^{sj}\alpha ^{2s\left( p-j\right) }\left( x,y\right) +y^{2ps}\right) } \\
&&\times \frac{\left( \left( -y\right) ^{sj}\alpha ^{2s\left( p-j\right)
}\left( x,y\right) \right) ^{p-k}}{\left( -y\right) ^{sj}\alpha ^{2s\left(
p-j\right) }\left( x,y\right) -\left( -y\right) ^{sj}\beta ^{2s\left(
p-j\right) }\left( x,y\right) }\frac{1}{z-\left( -y\right) ^{sj}\alpha
^{2s\left( p-j\right) }\left( x,y\right) } \\
&&+\sum_{j=0}^{p-1}\frac{1}{\dprod\limits_{i=0.i\neq j}^{p-1}\left( \left(
-y\right) ^{2sj}\beta ^{4s\left( p-j\right) }\left( x,y\right) -\left(
-y\right) ^{si}L_{2s\left( p-i\right) }\left( x,y\right) \left( -y\right)
^{sj}\beta ^{2s\left( p-j\right) }\left( x,y\right) +y^{2ps}\right) } \\
&&\times \frac{\left( \left( -y\right) ^{sj}\beta ^{2s\left( p-j\right)
}\left( x,y\right) \right) ^{p-k}}{\left( -y\right) ^{sj}\beta ^{2s\left(
p-j\right) }\left( x,y\right) -\left( -y\right) ^{sj}\alpha ^{2s\left(
p-j\right) }\left( x,y\right) }\frac{1}{z-\left( -y\right) ^{sj}\beta
^{2s\left( p-j\right) }\left( x,y\right) }.
\end{eqnarray*}

Observe that%
\begin{eqnarray*}
&&\dprod\limits_{i=0.i\neq j}^{p-1}\left( \left( -y\right) ^{2sj}\alpha
^{4s\left( p-j\right) }\left( x,y\right) -\left( -y\right) ^{si}L_{2s\left(
p-i\right) }\left( x,y\right) \left( -y\right) ^{sj}\alpha ^{2s\left(
p-j\right) }\left( x,y\right) +y^{2ps}\right) \\
&=&\left( \left( -y\right) ^{sj}\alpha ^{2s\left( p-j\right) }\left(
x,y\right) \right) ^{p-1} \\
&&\times \dprod\limits_{\substack{ i=0  \\ i\neq j}}^{p-1}\left( \left(
-y\right) ^{sj}\left( \alpha ^{2s\left( p-j\right) }\left( x,y\right)
+y^{2s\left( p-j\right) }\alpha ^{-2s\left( p-j\right) }\left( x,y\right)
\right) -\left( -y\right) ^{si}L_{2s\left( p-i\right) }\left( x,y\right)
\right) \\
&=&\left( \left( -y\right) ^{sj}\alpha ^{2s\left( p-j\right) }\left(
x,y\right) \right) ^{p-1}\dprod\limits_{\substack{ i=0  \\ i\neq j}}%
^{p-1}\left( \left( -y\right) ^{sj}\left( \alpha ^{2s\left( p-j\right)
}\left( x,y\right) +\beta ^{2s\left( p-j\right) }\left( x,y\right) \right)
-\left( -y\right) ^{si}L_{2s\left( p-i\right) }\left( x,y\right) \right) \\
&=&\left( \left( -y\right) ^{sj}\alpha ^{2s\left( p-j\right) }\left(
x,y\right) \right) ^{p-1}\mathcal{P}_{2s\left( p-j\right) }\left( x,y\right)
\end{eqnarray*}

Similarly one sees that 
\begin{eqnarray*}
&&\dprod\limits_{i=0.i\neq j}^{p-1}\left( \left( -y\right) ^{2sj}\beta
^{4s\left( p-j\right) }\left( x,y\right) -\left( -y\right) ^{si}L_{2s\left(
p-i\right) }\left( x,y\right) \left( -y\right) ^{sj}\beta ^{2s\left(
p-j\right) }\left( x,y\right) +y^{2ps}\right) \\
&=&\left( \left( -y\right) ^{sj}\beta ^{2s\left( p-j\right) }\left(
x,y\right) \right) ^{p-1}\mathcal{P}_{2s\left( p-j\right) }\left( x,y\right)
\end{eqnarray*}

Thus we have%
\begin{eqnarray}
&&\frac{z^{p-k}}{\dprod\limits_{j=0}^{p-1}\left( z^{2}-\left( -y\right)
^{sj}L_{2s\left( p-j\right) }\left( x,y\right) z+y^{2ps}\right) }
\label{4.38} \\
&=&\sum_{j=0}^{p-1}\frac{\left( \left( -y\right) ^{sj}\alpha ^{2s\left(
p-j\right) }\left( x,y\right) \right) ^{1-k}}{\left( -y\right) ^{sj}\left(
\alpha ^{2s\left( p-j\right) }\left( x,y\right) -\beta ^{2s\left( p-j\right)
}\left( x,y\right) \right) \mathcal{P}_{2s\left( p-j\right) }\left(
x,y\right) }\frac{1}{z-\left( -y\right) ^{sj}\alpha ^{2s\left( p-j\right)
}\left( x,y\right) }  \notag \\
&&-\sum_{j=0}^{p-1}\frac{\left( \left( -y\right) ^{sj}\beta ^{2s\left(
p-j\right) }\left( x,y\right) \right) ^{1-k}}{\left( -y\right) ^{sj}\left(
\alpha ^{2s\left( p-j\right) }\left( x,y\right) -\beta ^{2s\left( p-j\right)
}\left( x,y\right) \right) \mathcal{P}_{2s\left( p-j\right) }\left(
x,y\right) }\frac{1}{z-\left( -y\right) ^{sj}\beta ^{2s\left( p-j\right)
}\left( x,y\right) }  \notag \\
&=&\sum_{j=0}^{p-1}\frac{\left( -y\right) ^{-sjk}}{\sqrt{x^{2}+4y}\mathcal{P}%
_{2s\left( p-j\right) }\left( x,y\right) F_{2s\left( p-j\right) }\left(
x,y\right) }\left( \frac{\alpha ^{2s\left( p-j\right) \left( 1-k\right)
}\left( x,y\right) }{z-\left( -y\right) ^{sj}\alpha ^{2s\left( p-j\right)
}\left( x,y\right) }-\frac{\beta ^{2s\left( p-j\right) \left( 1-k\right)
}\left( x,y\right) }{z-\left( -y\right) ^{sj}\beta ^{2s\left( p-j\right)
}\left( x,y\right) }\right) .  \notag
\end{eqnarray}

Some further simplifications of the expression in parenthesis of the
right-hand side of (\ref{4.38}) give us%
\begin{eqnarray*}
&&\frac{\alpha ^{2s\left( p-j\right) \left( 1-k\right) }\left( x,y\right) }{%
z-\left( -y\right) ^{sj}\alpha ^{2s\left( p-j\right) }\left( x,y\right) }-%
\frac{\beta ^{2s\left( p-j\right) \left( 1-k\right) }\left( x,y\right) }{%
z-\left( -y\right) ^{sj}\beta ^{2s\left( p-j\right) }\left( x,y\right) } \\
&=&\frac{\left( \alpha ^{2s\left( p-j\right) \left( 1-k\right) }\left(
x,y\right) -\beta ^{2s\left( p-j\right) \left( 1-k\right) }\left( x,y\right)
\right) z-\left( -y\right) ^{2s\left( p-j\right) +sj}\left( \alpha
^{-2s\left( p-j\right) k}\left( x,y\right) -\beta ^{-2s\left( p-j\right)
k}\left( x,y\right) \right) }{z^{2}-L_{2s\left( p-j\right) }\!\left(
x,y\right) \left( -y\right) ^{sj}z+y^{2ps}} \\
&=&\sqrt{x^{2}+4y}\frac{F_{2s\left( p-j\right) \left( 1-k\right) }\left(
x,y\right) z-\left( -y\right) ^{2s\left( p-j\right) +sj}F_{-2s\left(
p-j\right) k}\left( x,y\right) }{z^{2}-L_{2s\left( p-j\right) }\!\left(
x,y\right) \left( -y\right) ^{sj}z+y^{2ps}} \\
&=&\sqrt{x^{2}+4y}\frac{\left( -\left( -y\right) ^{2s\left( p-j\right)
\left( 1-k\right) }\right) F_{2s\left( p-j\right) \left( k-1\right) }\left(
x,y\right) z+\left( -y\right) ^{2s\left( p-j\right) +sj}\left( -y\right)
^{-2s\left( p-j\right) k}F_{2s\left( p-j\right) k}\left( x,y\right) }{%
z^{2}-L_{2s\left( p-j\right) }\!\left( x,y\right) \left( -y\right)
^{sj}z+y^{2ps}} \\
&=&\sqrt{x^{2}+4y}\left( -y\right) ^{2s\left( p-j\right) \left( 1-k\right) }%
\frac{-F_{2s\left( p-j\right) \left( k-1\right) }\left( x,y\right) z+\left(
-y\right) ^{sj}F_{2s\left( p-j\right) k}\left( x,y\right) }{%
z^{2}-L_{2s\left( p-j\right) }\!\left( x,y\right) \left( -y\right)
^{sj}z+y^{2ps}}
\end{eqnarray*}

Then (\ref{4.38}) becomes%
\begin{eqnarray*}
&&\frac{z^{p-k}}{\dprod\limits_{j=0}^{p-1}\left( z^{2}-\left( -y\right)
^{sj}L_{2s\left( p-j\right) }\left( x,y\right) z+y^{2ps}\right) } \\
&=&\sum_{j=0}^{p-1}\frac{\left( -y\right) ^{-sjk}}{\sqrt{x^{2}+4y}\mathcal{P}%
_{2s\left( p-j\right) }\left( x,y\right) F_{2s\left( p-j\right) }\left(
x,y\right) } \\
&&\times \sqrt{x^{2}+4y}\left( -y\right) ^{2s\left( p-j\right) \left(
1-k\right) }\frac{-F_{2s\left( p-j\right) \left( k-1\right) }\left(
x,y\right) z+\left( -y\right) ^{sj}F_{2s\left( p-j\right) k}\left(
x,y\right) }{z^{2}-L_{2s\left( p-j\right) }\!\left( x,y\right) \left(
-y\right) ^{sj}z+y^{2ps}} \\
&=&\sum_{j=0}^{p-1}\frac{\left( -y\right) ^{sj\left( k-2\right) -2sp\left(
k-1\right) }}{\mathcal{P}_{2s\left( p-j\right) }\left( x,y\right)
F_{2s\left( p-j\right) }\left( x,y\right) }\frac{-F_{2s\left( p-j\right)
\left( k-1\right) }\left( x,y\right) z+\left( -y\right) ^{sj}F_{2s\left(
p-j\right) k}\left( x,y\right) }{z^{2}-L_{2s\left( p-j\right) }\!\left(
x,y\right) \left( -y\right) ^{sj}z+y^{2ps}},
\end{eqnarray*}%
as wanted. The proof of (\ref{4.36}) is similar and left to the reader.
\end{proof}

\begin{corollary}
\label{Cor4.10}Let $k\in \mathbb{N}$ be given. For $p\geq k$ we have%
\begin{equation}
\dbinom{n+p+1-k}{2p}_{F_{s}\left( x,y\right) }=\left( -y\right) ^{spn}\ast
\sum_{j=0}^{p-1}\frac{\left( -y\right) ^{sj\left( n-k\right) }F_{2s\left(
p-j\right) \left( n+1-k\right) }\left( x,y\right) }{\mathcal{P}_{2s\left(
p-j\right) }\left( x,y\right) F_{2s\left( p-j\right) }\left( x,y\right) }.
\label{4.39}
\end{equation}%
\begin{equation}
\dbinom{n+p-k}{2p-1}_{F_{s}\left( x,y\right) }=\sum_{j=0}^{p-1}\frac{\left(
-y\right) ^{sj\left( n-k\right) }F_{s\left( 2p-1-2j\right) \left(
n+1-k\right) }\left( x,y\right) }{\mathcal{R}_{s\left( 2p-1-2j\right)
}\left( x,y\right) F_{s\left( 2p-1-2j\right) }\left( x,y\right) }.
\label{4.40}
\end{equation}
\end{corollary}

\begin{proof}
From (\ref{4.1}) and (\ref{2.17}) we see that%
\begin{eqnarray*}
z^{p+1-k}\mathcal{Z}\left( \dbinom{n}{2p}_{F_{s}\left( x,y\right) }\right)
&=&\frac{z^{p+2-k}}{\left( z-\left( -y\right) ^{sp}\right)
\dprod\limits_{j=0}^{p-1}\left( z^{2}-\left( -y\right) ^{sj}L_{2s\left(
p-j\right) }\left( x,y\right) z+y^{2ps}\right) } \\
&=&\frac{z^{2}}{z-\left( -y\right) ^{sp}}\frac{z^{p-k}}{\dprod%
\limits_{j=0}^{p-1}\left( z^{2}-\left( -y\right) ^{sj}L_{2s\left( p-j\right)
}\left( x,y\right) z+y^{2ps}\right) },
\end{eqnarray*}%
and then, by using (\ref{4.35}) we obtain%
\begin{eqnarray*}
&&z^{p+1-k}\mathcal{Z}\left( \dbinom{n}{2p}_{F_{s}}\right) \\
&=&\frac{z^{2}}{z-\left( -y\right) ^{sp}}\sum_{j=0}^{p-1}\frac{\left(
-y\right) ^{sj\left( k-2\right) -2sp\left( k-1\right) }}{\mathcal{P}%
_{2s\left( p-j\right) }\left( x,y\right) F_{2s\left( p-j\right) }\left(
x,y\right) }\frac{-F_{2s\left( p-j\right) \left( k-1\right) }\left(
x,y\right) z+\left( -y\right) ^{sj}F_{2s\left( p-j\right) k}\left(
x,y\right) }{z^{2}-L_{2s\left( p-j\right) }\!\left( x,y\right) \left(
-y\right) ^{sj}z+y^{2ps}} \\
&=&\frac{z}{z-\left( -y\right) ^{sp}}\sum_{j=0}^{p-1}\frac{\left( -y\right)
^{-sjk}}{\mathcal{P}_{2s\left( p-j\right) }\left( x,y\right) F_{2s\left(
p-j\right) }\left( x,y\right) } \\
&&\times \frac{\frac{z}{\left( -y\right) ^{sj}}\left( F_{2s\left( p-j\right)
\left( 1-k\right) }\left( x,y\right) \frac{z}{\left( -y\right) ^{sj}}+\left(
-y\right) ^{2s\left( p-j\right) \left( 1-k\right) }F_{2s\left( p-j\right)
-2s\left( p-j\right) \left( 1-k\right) }\left( x,y\right) \right) }{\left( 
\frac{z}{\left( -y\right) ^{sj}}\right) ^{2}-L_{2s\left( p-j\right) }\left( 
\frac{z}{\left( -y\right) ^{sj}}\right) +y^{2s\left( p-j\right) }}.
\end{eqnarray*}

Thus, according to (\ref{4.2}), convolution theorem (\ref{2.3}) and (\ref%
{2.8}) we have that%
\begin{equation*}
\dbinom{n+p+1-k}{2p}_{F_{s}\left( x,y\right) }=\left( -y\right) ^{spn}\ast
\sum_{j=0}^{p-1}\frac{\left( -y\right) ^{sj\left( n-k\right) }F_{2s\left(
p-j\right) \left( n+1-k\right) }\left( x,y\right) }{\mathcal{P}_{2s\left(
p-j\right) }\left( x,y\right) F_{2s\left( p-j\right) }\left( x,y\right) },
\end{equation*}%
which proves (\ref{4.39}). Similarly, by using (\ref{4.1}), (\ref{2.18}) and
(\ref{4.36}) we have that 
\begin{eqnarray*}
&&z^{p-k}\mathcal{Z}\left( \dbinom{n}{2p-1}_{F_{s}\left( x,y\right) }\right)
\\
&=&\frac{z^{p+1-k}}{\dprod\limits_{j=0}^{p-1}\left( z^{2}-\left( -y\right)
^{sj}L_{s\left( 2p-1-2j\right) }\left( x,y\right) z+\left( -y\right)
^{s\left( 2p-1\right) }\right) } \\
&=&\sum_{j=0}^{p-1}\frac{\left( -y\right) ^{sj\left( k-2\right) -s\left(
2p-1\right) \left( k-1\right) }}{\mathcal{R}_{s\left( 2p-1-2j\right) }\left(
x,y\right) F_{s\left( 2p-1-2j\right) }\left( x,y\right) }\frac{z\left(
-F_{s\left( 2p-1-2j\right) \left( k-1\right) }\left( x,y\right) z+\left(
-y\right) ^{sj}F_{s\left( 2p-1-2j\right) k}\left( x,y\right) \right) }{%
z^{2}-\left( -y\right) ^{sj}L_{s\left( 2p-1-2j\right) }\!\left( x,y\right)
z+\left( -y\right) ^{s\left( 2p-1\right) }} \\
&=&\sum_{j=0}^{p-1}\frac{\left( -y\right) ^{-sjk}}{\mathcal{R}_{s\left(
2p-1-2j\right) }\left( x,y\right) F_{s\left( 2p-1-2j\right) }\left(
x,y\right) } \\
&&\times \frac{\frac{z}{\left( -y\right) ^{sj}}\left( F_{s\left(
2p-1-2j\right) \left( 1-k\right) }\left( x,y\right) \frac{z}{\left(
-y\right) ^{sj}}+\left( -y\right) ^{s\left( 2p-1-2j\right) \left( 1-k\right)
}F_{s\left( 2p-1-2j\right) -s\left( 2p-1-2j\right) \left( 1-k\right) }\left(
x,y\right) \right) }{\left( \frac{z}{\left( -y\right) ^{sj}}\right)
^{2}-L_{s\left( 2p-1-2j\right) }\left( \frac{z}{\left( -y\right) ^{sj}}%
\right) +\left( -y\right) ^{s\left( 2p-1-2j\right) }}
\end{eqnarray*}%
from where (by using (\ref{4.2}) and (\ref{2.8})) we obtain (\ref{4.40}).
\end{proof}

The case $p=1$ of (\ref{4.39}) is the same that the case $p=0$ of (\ref{4.23}%
) and the case $p=1$ of (\ref{4.27}), namely%
\begin{equation}
\dbinom{n+1}{2}_{F_{s}\left( x,y\right) }=\frac{1}{F_{2s}\left( x,y\right) }%
\sum_{t=0}^{n}\left( -y\right) ^{s\left( n-t\right) }F_{2st}\left(
x,y\right) .  \label{4.46}
\end{equation}

(The case $p=1$ of (\ref{4.40}) is a trivial identity.) Some more examples
from (\ref{4.39}) and (\ref{4.40}) are the following:

If $p=2$, we have for $k=1,2$ the following identities:%
\begin{eqnarray}
\dbinom{n+3-k}{4}_{F_{s}\left( x,y\right) } &=&\frac{1}{L_{4s}\left(
x,y\right) -\left( -y\right) ^{s}L_{2s}\left( x,y\right) }  \label{4.47} \\
&&\times \sum_{t=0}^{n}\left( -y\right) ^{2s\left( n-t\right) }\left( \frac{%
F_{4s\left( t+1-k\right) }\left( x,y\right) }{F_{4s}\left( x,y\right) }%
-\left( -y\right) ^{s\left( t-k\right) }\frac{F_{2s\left( t+1-k\right)
}\left( x,y\right) }{F_{2s}\left( x,y\right) }\right) .  \notag
\end{eqnarray}%
\begin{equation}
\dbinom{n+2-k}{3}_{F_{s}\left( x,y\right) }=\frac{1}{L_{3s}\left( x,y\right)
-\left( -y\right) ^{s}L_{s}\left( x,y\right) }\left( \frac{F_{3s\left(
n+1-k\right) }\left( x,y\right) }{F_{3s}\left( x,y\right) }-\frac{\left(
-y\right) ^{s\left( n-k\right) }F_{s\left( n+1-k\right) }\left( x,y\right) }{%
F_{s}\left( x,y\right) }\right) .  \label{4.48}
\end{equation}

(See also (\ref{4.262}) and (\ref{4.25}).)

If $p=3$, we have for $k=1,2,3$ the identities:%
\begin{equation}
\dbinom{n+4-k}{6}_{F_{s}\left( x,y\right) }=\sum_{t=0}^{n}\left( -y\right)
^{3s\left( n-t\right) }\left( 
\begin{array}{l}
\frac{F_{6s\left( t+1-k\right) }\left( x,y\right) }{\left( L_{6s}\left(
x,y\right) -\left( -y\right) ^{s}L_{4s}\left( x,y\right) \right) \left(
L_{6s}\left( x,y\right) -\left( -y\right) ^{2s}L_{2s}\left( x,y\right)
\right) F_{6s}\left( x,y\right) } \\ 
+\frac{\left( -y\right) ^{s\left( t-k\right) }F_{4s\left( t+1-k\right)
}\left( x,y\right) }{\left( \left( -y\right) ^{s}L_{4s}\left( x,y\right)
-L_{6s}\left( x,y\right) \right) \left( \left( -y\right) ^{s}L_{4s}\left(
x,y\right) -\left( -y\right) ^{2s}L_{2s}\left( x,y\right) \right)
F_{4s}\left( x,y\right) } \\ 
+\frac{\left( -y\right) ^{2s\left( t-k\right) }F_{2s\left( t+1-k\right)
}\left( x,y\right) }{\left( \left( -y\right) ^{2s}L_{2s}\left( x,y\right)
-L_{6s}\left( x,y\right) \right) \left( \left( -y\right) ^{2s}L_{2s}\left(
x,y\right) -\left( -y\right) ^{s}L_{4s}\left( x,y\right) \right)
F_{2s}\left( x,y\right) }%
\end{array}%
\right) .  \label{4.49}
\end{equation}

\begin{eqnarray}
\dbinom{n+3-k}{5}_{F_{s}\left( x,y\right) } &=&\frac{F_{5s\left(
n+1-k\right) }\left( x,y\right) }{\left( L_{5s}\left( x,y\right) -\left(
-y\right) ^{s}L_{3s}\left( x,y\right) \right) \left( L_{5s}\left( x,y\right)
-\left( -y\right) ^{2s}L_{s}\left( x,y\right) \right) F_{5s}\left(
x,y\right) }  \label{4.50} \\
&&+\frac{\left( -y\right) ^{s\left( n-k\right) }F_{3s\left( n+1-k\right)
}\left( x,y\right) }{\left( \left( -y\right) ^{s}L_{3s}\left( x,y\right)
-L_{5s}\left( x,y\right) \right) \left( \left( -y\right) ^{s}L_{3s}\left(
x,y\right) -\left( -y\right) ^{2s}L_{s}\left( x,y\right) \right)
F_{3s}\left( x,y\right) }  \notag \\
&&+\frac{\left( -y\right) ^{2s\left( n-k\right) }F_{s\left( n+1-k\right)
}\left( x,y\right) }{\left( \left( -y\right) ^{2s}L_{s}\left( x,y\right)
-L_{5s}\left( x,y\right) \right) \left( \left( -y\right) ^{2s}L_{s}\left(
x,y\right) -\left( -y\right) ^{s}L_{3s}\left( x,y\right) \right) F_{s}\left(
x,y\right) }.  \notag
\end{eqnarray}

(See also (\ref{4.30}).)

\section{\label{Sec5}Derivatives of bivariate $s$-Fibopolynomials}

The partial derivatives of bivariate Lucas polynomials $L_{n}\left(
x,y\right) $ are given by the well-known formulas%
\begin{equation}
\frac{\partial }{\partial x}L_{n}\left( x,y\right) =nF_{n}\left( x,y\right) 
\text{ \ \ \ \ \ \ and \ \ \ \ \ \ }\frac{\partial }{\partial y}L_{n}\left(
x,y\right) =nF_{n-1}\left( x,y\right) .  \label{5.1}
\end{equation}

We will use some of the results obtained in sections \ref{Sec2} and \ref%
{Sec4}, together with (\ref{5.1}), in order to obtain formulas for the
partial derivatives of bivariate $s$-Fibopolynomials $\binom{n}{p}%
_{F_{s}\left( x,y\right) }$.

We begin by noting that, according to (\ref{4.1}) we have that

\begin{equation}
\frac{\mathcal{Z}\left( \frac{\partial }{\partial x}\dbinom{n}{2p}%
_{F_{s}\left( x,y\right) }\right) }{\mathcal{Z}\left( \dbinom{n}{2p}%
_{F_{s}\left( x,y\right) }\right) }=-\frac{\frac{\partial }{\partial x}%
\sum_{i=0}^{2p+1}\left( -1\right) ^{\frac{\left( si+2(s+1)\right) \left(
i+1\right) }{2}}\dbinom{2p+1}{i}_{F_{s}\left( x,y\right) }y^{\frac{si\left(
i-1\right) }{2}}z^{2p+1-i}}{\sum_{i=0}^{2p+1}\left( -1\right) ^{\frac{\left(
si+2(s+1)\right) \left( i+1\right) }{2}}\dbinom{2p+1}{i}_{F_{s}\left(
x,y\right) }y^{\frac{si\left( i-1\right) }{2}}z^{2p+1-i}}.  \label{5.2}
\end{equation}

By using (\ref{2.17}) we get from (\ref{5.2}) that%
\begin{eqnarray*}
\frac{\mathcal{Z}\left( \frac{\partial }{\partial x}\dbinom{n}{2p}%
_{F_{s}\left( x,y\right) }\right) }{\mathcal{Z}\left( \dbinom{n}{2p}%
_{F_{s}\left( x,y\right) }\right) } &=&-\frac{\frac{\partial }{\partial x}%
\prod\limits_{j=0}^{p-1}\left( z^{2}-\left( -y\right) ^{sj}L_{2s\left(
p-j\right) }\left( x,y\right) z+y^{2ps}\right) }{\prod\limits_{j=0}^{p-1}%
\left( z^{2}-\left( -y\right) ^{sj}L_{2s\left( p-j\right) }\left( x,y\right)
z+y^{2ps}\right) } \\
&=&-\frac{\sum\limits_{k=0}^{p-1}\left( -\left( -y\right) ^{sk}\frac{%
\partial }{\partial x}L_{2s\left( p-k\right) }\left( x,y\right) z\right)
\prod\limits_{\substack{ j=0,  \\ j\neq k}}^{p-1}\left( z^{2}-\left(
-y\right) ^{sj}L_{2s\left( p-j\right) }\left( x,y\right) z+y^{2ps}\right) }{%
\prod\limits_{j=0}^{p-1}\left( z^{2}-\left( -y\right) ^{sj}L_{2s\left(
p-j\right) }\left( x,y\right) z+y^{2ps}\right) } \\
&=&\sum_{k=0}^{p-1}\frac{\left( -y\right) ^{sk}2s\left( p-k\right)
F_{2s\left( p-k\right) }\left( x,y\right) z}{z^{2}-\left( -y\right)
^{sk}L_{2s\left( p-k\right) }\left( x,y\right) z+y^{2ps}},
\end{eqnarray*}%
from where%
\begin{eqnarray*}
\mathcal{Z}\left( \frac{\partial }{\partial x}\dbinom{n}{2p}_{F_{s}\left(
x,y\right) }\right) &=&\mathcal{Z}\left( \dbinom{n}{2p}_{F_{s}\left(
x,y\right) }\right) \sum_{k=0}^{p-1}\frac{\left( -y\right) ^{sk}2s\left(
p-k\right) F_{2s\left( p-k\right) }\left( x,y\right) z}{z^{2}-\left(
-y\right) ^{sk}L_{2s\left( p-k\right) }\left( x,y\right) z+y^{2ps}} \\
&=&\mathcal{Z}\left( \dbinom{n}{2p}_{F_{s}\left( x,y\right) }\right)
\sum_{k=0}^{p-1}\frac{2s\left( p-k\right) F_{2s\left( p-k\right) }\left(
x,y\right) \frac{z}{\left( -y\right) ^{sk}}}{\left( \left( \frac{z}{y^{sk}}%
\right) ^{2}-L_{2s\left( p-k\right) }\left( x,y\right) \frac{z}{\left(
-y\right) ^{sk}}+y^{2ps-2sk}\right) } \\
&=&\mathcal{Z}\left( \dbinom{n}{2p}_{F_{s}\left( x,y\right) }\right)
\sum_{k=0}^{p-1}2s\left( p-k\right) \mathcal{Z}\left( \left( -y\right)
^{skn}F_{2s\left( p-k\right) n}\left( x,y\right) \right) ,
\end{eqnarray*}%
and finally, the convolution theorem gives us%
\begin{equation}
\frac{\partial }{\partial x}\dbinom{n}{2p}_{F_{s}\left( x,y\right) }=2s%
\dbinom{n}{2p}_{F_{s}\left( x,y\right) }\ast \sum_{k=0}^{p-1}\left(
p-k\right) \left( -y\right) ^{skn}F_{2s\left( p-k\right) n}\left( x,y\right)
.  \label{5.4}
\end{equation}

Similarly, from (\ref{4.1}) and (\ref{2.18}) we have that%
\begin{eqnarray*}
\frac{\mathcal{Z}\left( \frac{\partial }{\partial x}\dbinom{n}{2p-1}%
_{F_{s}\left( x,y\right) }\right) }{\mathcal{Z}\left( \dbinom{n}{2p-1}%
_{F_{s}\left( x,y\right) }\right) } &=&\frac{-\frac{\partial }{\partial x}%
\prod\limits_{j=0}^{p-1}\left( z^{2}-\left( -y\right) ^{sj}L_{s\left(
2p-1-2j\right) }\left( x,y\right) z+\left( -y\right) ^{\left( 2p-1\right)
s}\right) }{\prod\limits_{j=0}^{p-1}\left( z^{2}-\left( -y\right)
^{sj}L_{s\left( 2p-1-2j\right) }\left( x,y\right) z+\left( -y\right)
^{\left( 2p-1\right) s}\right) } \\
&=&\sum_{k=0}^{p-1}\frac{\left( -y\right) ^{sk}s\left( 2p-1-2k\right)
F_{s\left( 2p-1-2k\right) }\left( x,y\right) z}{z^{2}-\left( -y\right)
^{sk}L_{s\left( 2p-1-2k\right) }\left( x,y\right) z+\left( -y\right)
^{\left( 2p-1\right) s}}.
\end{eqnarray*}

Then%
\begin{equation*}
\mathcal{Z}\left( \frac{\partial }{\partial x}\dbinom{n}{2p-1}_{F_{s}\left(
x,y\right) }\right) =\mathcal{Z}\left( \dbinom{n}{2p-1}_{F_{s}\left(
x,y\right) }\right) \sum_{k=0}^{p-1}\frac{\left( -y\right) ^{sk}s\left(
2p-1-2k\right) F_{s\left( 2p-1-2k\right) }\left( x,y\right) z}{z^{2}-\left(
-y\right) ^{sk}L_{s\left( 2p-1-2k\right) }\left( x,y\right) z+\left(
-y\right) ^{\left( 2p-1\right) s}},
\end{equation*}%
and from the convolution theorem we get%
\begin{equation}
\frac{\partial }{\partial x}\dbinom{n}{2p-1}_{F_{s}\left( x,y\right) }=s%
\dbinom{n}{2p-1}_{F_{s}\left( x,y\right) }\ast \sum_{k=0}^{p-1}\left(
2p-1-2k\right) \left( -y\right) ^{skn}F_{s\left( 2p-1-2k\right) n}\left(
x,y\right) .  \label{5.5}
\end{equation}

Formulas (\ref{5.4}) and (\ref{5.5}) for the derivatives with respect to $x$
of the bivariate $s$-Fibopolynomials $\binom{n}{2p}_{F_{s}\left( x,y\right)
} $ and $\binom{n}{2p-1}_{F_{s}\left( x,y\right) }$ can be written together
as%
\begin{equation*}
\frac{\partial }{\partial x}\dbinom{n}{p}_{F_{s}\left( x,y\right) }=s\dbinom{%
n}{p}_{F_{s}\left( x,y\right) }\ast \sum_{k=0}^{\left\lfloor \frac{p+1}{2}%
\right\rfloor -1}\left( p-2k\right) \left( -y\right) ^{skn}F_{s\left(
p-2k\right) n}\left( x,y\right) .
\end{equation*}

Some examples are%
\begin{equation}
\frac{\partial }{\partial x}\dbinom{n}{2}_{F_{s}\left( x,y\right)
}=2s\sum_{t=0}^{n}F_{2st}\left( x,y\right) \dbinom{n-t}{2}_{F_{s}\left(
x,y\right) }.  \label{5.7}
\end{equation}%
\begin{equation}
\frac{\partial }{\partial x}\dbinom{n}{3}_{F_{s}\left( x,y\right)
}=s\sum_{t=0}^{n}\left( 3F_{3st}\left( x,y\right) +\left( -y\right)
^{st}F_{st}\left( x,y\right) \right) \dbinom{n-t}{3}_{F_{s}\left( x,y\right)
}.  \label{5.8}
\end{equation}%
\begin{equation}
\frac{\partial }{\partial x}\dbinom{n}{4}_{F_{s}\left( x,y\right)
}=2s\sum_{t=0}^{n}\left( 2F_{4st}\left( x,y\right) +\left( -y\right)
^{st}F_{2st}\left( x,y\right) \right) \dbinom{n-t}{4}_{F_{s}\left(
x,y\right) }.  \label{5.9}
\end{equation}

Let us consider now derivatives with respect to $y$ of bivariate $s$%
-Fibopolynomials. From (\ref{4.1}) and (\ref{2.17}) we have that%
\begin{eqnarray}
&&z\frac{\mathcal{Z}\left( \frac{\partial }{\partial y}\dbinom{n}{2p}%
_{F_{s}\left( x,y\right) }\right) }{\mathcal{Z}\left( \dbinom{n}{2p}%
_{F_{s}\left( x,y\right) }\right) }  \label{5.91} \\
&=&-z\frac{\frac{\partial }{\partial y}\left( \left( z-\left( -y\right)
^{sp}\right) \prod\limits_{j=0}^{p-1}\left( z^{2}-\left( -y\right)
^{sj}L_{2s\left( p-j\right) }\left( x,y\right) z+y^{2ps}\right) \right) }{%
\left( z-\left( -y\right) ^{sp}\right) \prod\limits_{j=0}^{p-1}\left(
z^{2}-\left( -y\right) ^{sj}L_{2s\left( p-j\right) }\left( x,y\right)
z+y^{2ps}\right) }  \notag \\
&=&-\frac{sp\left( -y\right) ^{sp-1}z}{z-\left( -y\right) ^{sp}}  \notag \\
&&+\sum_{k=0}^{p-1}\frac{\left( -y\right) ^{sk}2s\left( p-k\right)
F_{2s\left( p-k\right) -1}\left( x,y\right) z^{2}-sk\left( -y\right)
^{sk-1}L_{2s\left( p-k\right) }\left( x,y\right) z^{2}-2psy^{2ps-1}z}{%
z^{2}-\left( -y\right) ^{sk}L_{2s\left( p-k\right) }\left( x,y\right)
z+y^{2ps}}.  \notag
\end{eqnarray}

Observe that%
\begin{eqnarray*}
&&\frac{\left( -y\right) ^{sk}2s\left( p-k\right) F_{2s\left( p-k\right)
-1}\left( x,y\right) z^{2}-sk\left( -y\right) ^{sk-1}L_{2s\left( p-k\right)
}\left( x,y\right) z^{2}-2psy^{2ps-1}z}{z^{2}-\left( -y\right)
^{sk}L_{2s\left( p-k\right) }\left( x,y\right) z+y^{2ps}} \\
&=&2sp\left( -y\right) ^{sk}\left( \frac{z}{\left( -y\right) ^{sk}}\right) 
\frac{F_{2s\left( p-k\right) -1}\left( x,y\right) \left( \frac{z}{\left(
-y\right) ^{sk}}\right) -y^{2ps-2sk-1}}{\left( \frac{z}{\left( -y\right)
^{sk}}\right) ^{2}-L_{2s\left( p-k\right) }\left( x,y\right) \frac{z}{\left(
-y\right) ^{sk}}+y^{2ps-2sk}} \\
&&+\frac{sk}{y}\frac{L_{2s\left( p-k\right) }\left( x,y\right)
-2yF_{2s\left( p-k\right) -1}\left( x,y\right) }{F_{2s\left( p-k\right)
}\left( x,y\right) }z\frac{F_{2s\left( p-k\right) }\left( x,y\right) \frac{z%
}{\left( -y\right) ^{sk}}}{\left( \frac{z}{\left( -y\right) ^{sk}}\right)
^{2}-L_{2s\left( p-k\right) }\left( x,y\right) \frac{z}{\left( -y\right)
^{sk}}+y^{2ps-2sk}}.
\end{eqnarray*}

With (\ref{1.8}) we can see that%
\begin{equation*}
\frac{L_{2s\left( p-k\right) }\left( x,y\right) -2yF_{2s\left( p-k\right)
-1}\left( x,y\right) }{F_{2s\left( p-k\right) }\left( x,y\right) }=x.
\end{equation*}

Then, according to (\ref{2.8}) we can write (\ref{5.91}) as%
\begin{equation*}
z\frac{\mathcal{Z}\left( \frac{\partial }{\partial y}\dbinom{n}{2p}%
_{F_{s}\left( x,y\right) }\right) }{\mathcal{Z}\left( \dbinom{n}{2p}%
_{F_{s}\left( x,y\right) }\right) }=-\frac{sp\left( -y\right) ^{sp-1}z}{%
z-\left( -y\right) ^{sp}}+\sum\limits_{k=0}^{p-1}\left( 
\begin{array}{c}
2sp\left( -y\right) ^{sk}\mathcal{Z}\left( \left( -y\right)
^{skn}F_{2s\left( p-k\right) \left( n+1\right) -1}\left( x,y\right) \right)
\\ 
+\frac{skx}{y}\mathcal{Z}\left( \left( -y\right) ^{sk\left( n+1\right)
}F_{2s\left( p-k\right) \left( n+1\right) }\left( x,y\right) \right)%
\end{array}%
\right) ,
\end{equation*}%
from where we get finally that%
\begin{eqnarray}
&&\frac{\partial }{\partial y}\dbinom{n+1}{2p}_{F_{s}\left( x,y\right) }
\label{5.10} \\
&=&sp\dbinom{n}{2p}_{F_{s}\left( x,y\right) }\ast \left(
\sum\limits_{k=0}^{p-1}\left( -y\right) ^{sk\left( n+1\right) }\left( 
\begin{array}{c}
2F_{2s\left( p-k\right) \left( n+1\right) -1}\left( x,y\right) \\ 
+ \\ 
\frac{kx}{py}F_{2s\left( p-k\right) \left( n+1\right) }\left( x,y\right)%
\end{array}%
\right) -\frac{\left( -y\right) ^{sp\left( n+1\right) -1}}{p}\right) . 
\notag
\end{eqnarray}

Similarly, from (\ref{4.1}) and (\ref{2.18}) we have that%
\begin{eqnarray*}
&&\frac{z\mathcal{Z}\left( \frac{\partial }{\partial y}\dbinom{n}{2p-1}%
_{F_{s}\left( x,y\right) }\right) }{\mathcal{Z}\left( \dbinom{n}{2p-1}%
_{F_{s}\left( x,y\right) }\right) } \\
&=&\frac{-z\frac{\partial }{\partial y}\prod\limits_{j=0}^{p-1}\left(
z^{2}-\left( -y\right) ^{sj}L_{s\left( 2p-1-2j\right) }\left( x,y\right)
z+\left( -y\right) ^{\left( 2p-1\right) s}\right) }{\prod\limits_{j=0}^{p-1}%
\left( z^{2}-\left( -y\right) ^{sj}L_{s\left( 2p-1-2j\right) }\left(
x,y\right) z+\left( -y\right) ^{\left( 2p-1\right) s}\right) } \\
&=&\sum_{k=0}^{p-1}z\frac{%
\begin{array}{c}
\left( -y\right) ^{sk}s\left( 2p-1-2k\right) F_{s\left( 2p-1-2k\right)
-1}\left( x,y\right) z-sk\left( -y\right) ^{sk-1}L_{s\left( 2p-1-2j\right)
}\left( x,y\right) z \\ 
+\left( 2p-1\right) s\left( -y\right) ^{s\left( 2p-1\right) -1}%
\end{array}%
}{z^{2}-\left( -y\right) ^{sk}L_{s\left( 2p-1-2k\right) }\left( x,y\right)
z+\left( -y\right) ^{\left( 2p-1\right) s}} \\
&=&\frac{s\left( 2p-1\right) }{\left( -y\right) ^{sk}}+\sum_{k=0}^{p-1}%
\left( 
\begin{array}{c}
\frac{\frac{z}{\left( -y\right) ^{sk}}\left( F_{s\left( 2p-1-2k\right)
-1}\left( x,y\right) \frac{z}{\left( -y\right) ^{sk}}+\left( -y\right)
^{s\left( 2p-1-2k\right) -1}\right) }{\left( \left( \frac{z}{\left(
-y\right) ^{sk}}\right) ^{2}-L_{s\left( 2p-1-2k\right) }\left( x,y\right) 
\frac{z}{\left( -y\right) ^{sk}}+\left( -y\right) ^{\left( 2p-1-2k\right)
s}\right) } \\ 
\\ 
+\frac{sk}{y}z\frac{L_{s\left( 2p-1-2j\right) }\left( x,y\right)
-2yF_{s\left( 2p-1-2k\right) -1}\left( x,y\right) }{\left( \frac{z}{\left(
-y\right) ^{sk}}\right) ^{2}-L_{s\left( 2p-1-2k\right) }\left( x,y\right) 
\frac{z}{\left( -y\right) ^{sk}}+\left( -y\right) ^{\left( 2p-1-2k\right) s}}%
\frac{z}{\left( -y\right) ^{sk}}%
\end{array}%
\right)
\end{eqnarray*}%
from where we obtain that%
\begin{eqnarray}
&&\frac{\partial }{\partial y}\dbinom{n+1}{2p-1}_{F_{s}\left( x,y\right) }
\label{5.11} \\
&=&s\left( 2p-1\right) \!\dbinom{n}{2p-1}_{F_{s}\left( x,y\right) }\ast
\!\sum_{k=0}^{p-1}\!\left( -y\right) ^{sk\left( n+1\right) }\left( 
\begin{array}{c}
\!F_{s\left( 2p-1-2k\right) \left( n+1\right) -1}\!\left( x,y\right) \\ 
+ \\ 
\frac{kx}{\left( 2p-1\right) y}F_{s\left( 2p-1-2k\right) \left( n+1\right)
}\!\left( x,y\right)%
\end{array}%
\right) .  \notag
\end{eqnarray}

Formulas (\ref{5.10}) and (\ref{5.11}) can be written together as%
\begin{eqnarray*}
&&\frac{\partial }{\partial y}\dbinom{n+1}{p}_{F_{s}\left( x,y\right) } \\
&=&sp\dbinom{n}{p}_{F_{s}\left( x,y\right) }\ast \left(
\sum_{k=0}^{\left\lfloor \frac{p+1}{2}\right\rfloor -1}\left( -y\right)
^{sk\left( n+1\right) }\left( 
\begin{array}{c}
F_{s\left( p-2k\right) \left( n+1\right) -1}\left( x,y\right) \\ 
+ \\ 
\frac{kx}{py}F_{s\left( p-2k\right) \left( n+1\right) }\left( x,y\right)%
\end{array}%
\right) -\frac{1+\left( -1\right) ^{p}}{4}\left( -y\right) ^{\frac{sp}{2}%
\left( n+1\right) -1}\right) .
\end{eqnarray*}

Some examples are%
\begin{equation*}
\frac{\partial }{\partial y}\dbinom{n+1}{2}_{F_{s}\left( x,y\right)
}=s\sum_{t=0}^{n}\left( 2F_{2s\left( t+1\right) -1}\left( x,y\right) -\left(
-y\right) ^{s\left( t+1\right) -1}\right) \dbinom{n-t}{2}_{F_{s}\left(
x,y\right) }.
\end{equation*}%
\begin{eqnarray}
&&\frac{\partial }{\partial y}\dbinom{n+1}{3}_{F_{s}\left( x,y\right) }
\label{5.14} \\
&=&3s\sum_{t=0}^{n}\left( F_{3s\left( t+1\right) -1}\left( x,y\right)
+\left( -y\right) ^{s\left( t+1\right) }\left( F_{s\left( t+1\right)
-1}\left( x,y\right) +\frac{x}{3y}F_{s\left( t+1\right) }\left( x,y\right)
\right) \right) \dbinom{n-t}{3}_{F_{s}\left( x,y\right) }.  \notag
\end{eqnarray}%
\begin{eqnarray}
&&\frac{\partial }{\partial y}\dbinom{n+1}{4}_{F_{s}\left( x,y\right) }
\label{5.15} \\
&=&4s\sum_{t=0}^{n}\left( F_{4s\left( t+1\right) -1}\left( x,y\right)
+\left( -y\right) ^{s\left( t+1\right) }\left( F_{2s\left( t+1\right)
-1}\left( x,y\right) +\frac{x}{4y}F_{2s\left( t+1\right) }\left( x,y\right)
\right) -\frac{\left( -y\right) ^{2s\left( t+1\right) -1}}{2}\right)  \notag
\\
&&\times \dbinom{n-t}{4}_{F_{s}\left( x,y\right) }.  \notag
\end{eqnarray}

Summarizing, we have proved the following result.

\begin{proposition}
The partial derivatives of the bivariate $s$-Fibopolynomial $\binom{n}{p}%
_{F_{s}\left( x,y\right) }$ can be written as%
\begin{equation}
\frac{\partial }{\partial x}\dbinom{n}{p}_{F_{s}\left( x,y\right) }=s\dbinom{%
n}{p}_{F_{s}\left( x,y\right) }\ast \sum\limits_{k=0}^{\lfloor \left(
p+1\right) /2\rfloor -1}\left( p-2k\right) \left( -y\right) ^{skn}F_{s\left(
p-2k\right) n}\left( x,y\right) .  \label{5.6}
\end{equation}%
\begin{eqnarray}
&&\frac{\partial }{\partial y}\dbinom{n+1}{p}_{F_{s}\left( x,y\right) }
\label{5.12} \\
&=&sp\dbinom{n}{p}_{F_{s}\left( x,y\right) }\ast \left( 
\begin{array}{c}
\sum\limits_{k=0}^{\lfloor \left( p+1\right) /2\rfloor -1}\left( -y\right)
^{sk\left( n+1\right) }\left( F_{s\left( p-2k\right) \left( n+1\right)
-1}\left( x,y\right) +\frac{kx}{py}F_{s\left( p-2k\right) \left( n+1\right)
}\left( x,y\right) \right) \\ 
\\ 
-\frac{1+\left( -1\right) ^{p}}{4}\left( -y\right) ^{\frac{sp}{2}\left(
n+1\right) -1}%
\end{array}%
\right) .  \notag
\end{eqnarray}
\end{proposition}

\begin{remark}
The case $p=1$ of (\ref{5.6}) and (\ref{5.12}) gives us explicit formulas
for the partial derivatives of bivariate $s$-Fibonacci polynomials, namely%
\begin{equation}
\frac{\partial }{\partial x}F_{sn}\left( x,y\right) =\frac{\frac{\partial }{%
\partial x}F_{s}\left( x,y\right) }{F_{s}\left( x,y\right) }F_{sn}\left(
x,y\right) +sF_{sn}\left( x,y\right) \ast F_{sn}\left( x,y\right) ,
\label{5.16}
\end{equation}%
and%
\begin{equation}
\frac{\partial }{\partial y}F_{s\left( n+1\right) }\left( x,y\right) =\frac{%
\frac{\partial }{\partial y}F_{s}\left( x,y\right) }{F_{s}\left( x,y\right) }%
F_{s\left( n+1\right) }\left( x,y\right) +sF_{sn}\left( x,y\right) \ast
F_{s\left( n+1\right) -1}\left( x,y\right) ,  \label{5.17}
\end{equation}%
respectively. The case $s=1$ of (\ref{5.16}) and (\ref{5.17}) gives us%
\begin{equation}
\frac{\partial }{\partial x}F_{n}\left( x,y\right) =\frac{\partial }{%
\partial y}F_{n+1}\left( x,y\right) =F_{n}\left( x,y\right) \ast F_{n}\left(
x,y\right) .  \label{5.18}
\end{equation}%
or, according to (\ref{2.156}),%
\begin{equation}
\frac{\partial }{\partial x}F_{n}\left( x,y\right) =\frac{\partial }{%
\partial y}F_{n+1}\left( x,y\right) =\frac{1}{x^{2}+4y}\left( nL_{n}\left(
x,y\right) -xF_{n}\left( x,y\right) \right) .  \label{5.19}
\end{equation}
\end{remark}

\bigskip

\_\_\_\_\_\_\_\_\_\_\_\_\_\_\_\_\_\_\_\_\_\_\_\_\_\_\_\_\_\_\_\_\_\_\_\_\_\_%
\_\_\_\_\_\_\_\_\_\_\_\_\_\_\_\_\_\_\_\_\_\_\_\_\_\_\_\_\_\_\_\_\_\_\_\_\_\_%
\_\_\_\_\_\_\_\_\_\_\_\_\_\_\_\_\_\_\_\_\_\_\_\_\_\_\_\_\_\_\_\_\_\_\_

2000 Mathematics Subject Classification: Primary 11XX; Secondary 11B39.

Keywords: Bivariate Fibonacci and Lucas polynomials, $s$-Fibonomial
coefficients, Z transform.


\begin{thebibliography}{99}
\bibitem{C} L. Carlitz, Generating functions for powers of certain sequence
of numbers, \textit{Duke Math. J. }\textbf{29} (1962), 521--537.

\bibitem{Ca} M. Catalani, Some formulae for bivariate Fibonacci and Lucas
polynomials, arXiv:math/0406323v1

\bibitem{G} Urs Graf, \textit{Applied Laplace Transforms and z-Transforms
for Scientists and Engineers. A Computational Approach using a `Mathematica'
Package,} Birkh\"{a}user, 2004.

\bibitem{Hog} V. E. Hoggatt, Jr. Fibonacci numbers and generalized binomial
coefficients, \textit{Fibonacci Quart.} \textbf{5} (1967), 383--400.

\bibitem{Hog2} V. E. Hoggatt, Jr. and C. T. Long, Divisibility properties of
generalized Fibonacci polynomials, \textit{Fibonacci Quart.} \textbf{12}
(1974), 113--120.

\bibitem{Hor} A. F. Horadam, Generating functions for powers of a certain
generalised sequence of numbers, \textit{Duke Math. J.} \textbf{32} (1965),
437--446.

\bibitem{Jh} R. C. Johnson, \textit{Fibonacci numbers and matrices,} in
www.dur.ac.uk/bob.johnson/fibonacci/

\bibitem{K} Thomas Koshy, \textit{Fibonacci and Lucas Numbers with
Applications,} John Wiley \& Sons, Inc. 2001.

\bibitem{Phi} Phil Mana, Problem B-177, \textit{Fibonacci Quart.} \textbf{8}
(1970), 448.

\bibitem{Pi1} C. Pita, More on Fibonomials, in Florian Luca and Pantelimon St%
\u{a}nic\u{a}, eds., \textit{Proceedings of the Fourteenth International
Conference on Fibonacci Numbers and Their Applications. Sociedad Matem\'{a}%
tica Mexicana}, 2011, pp. 237--274

\bibitem{Pi2} C. Pita, On $s$-Fibonomials, \textit{J. Integer Seq. }\textbf{%
14} (2011). Article 11.3.7.

\bibitem{Pi3} C. Pita, Sums of Products of $s$-Fibonacci Polynomial
Sequences, \textit{J. Integer Seq. }\textbf{14} (2011). Article 11.7.6.

\bibitem{R} J. Riordan, Generating functions for powers of Fibonacci
numbers, \textit{Duke Math. J. }\textbf{29} (1962), 5--12.

\bibitem{Sh} A. G. Shannon, A method of Carlitz applied to the $k$-th power
generating function for Fibonacci numbers, \textit{Fibonacci Quart.} \textbf{%
12} (1974), 293--299.

\bibitem{St} I. Strazdins, Lucas factors and a Fibonomial generating
function, \textit{Applications of Fibonacci Numbers, Vol. 7}, Kluwer
Academic Publishers (1998), 401--404.

\bibitem{Sw} M. N. S. Swamy, Generalized Fibonacci and Lucas polynomials and
their associated diagonal polynomials, \textit{Fibonacci Quart.} \textbf{37}
(1999), 213--222.

\bibitem{V} S. Vajda, \textit{Fibonacci and Lucas Numbers, and the Golden
Section,} Dover, 1989.

\bibitem{Vi} Robert Vilch, \textit{Z Transform Theory and Applications,} D.
Reidel Publishing Company, 1987.

\bibitem{Yu} Hongquan Yu and Chuanguang Liang, Identities involving partial
derivatives of bivariate Fibonacci and Lucas polynomials, \textit{Fib.
Quart. }\textbf{35} (1997), 19--23.
\end{thebibliography}
\end{document}